\documentclass[11pt]{amsart}

\usepackage[all]{xy}

\usepackage{amsmath}
\usepackage[latin1]{inputenc}
\usepackage{amsfonts}
\usepackage{amssymb}
\usepackage{multicol}
\usepackage{verbatim}
\usepackage{amsthm}
\usepackage{amscd}
\usepackage{pb-diagram}
\usepackage[mathscr]{eucal}
\usepackage[latin1]{inputenc}

\usepackage{amsmath}

\usepackage{amsfonts}

\usepackage{amssymb}

\usepackage{amsthm}

\usepackage{graphics}

\usepackage{mathrsfs}

\usepackage{xcolor}

\usepackage{tikz-cd}

\usepackage[normalem]{ulem}

\usepackage{hyperref}

\newcommand{\ZZ}{\mathbb{Z}}

\newcommand{\CC}{\mathbb{C}}

\newcommand{\QQ}{\mathbb{Q}}

\newcommand{\Glie}{\mathfrak{g}}

\newcommand{\U}{\mathcal{U}}

\newtheorem{thm}{Theorem}[section]

\newtheorem{defi}[thm]{Definition}

\newtheorem{cor}[thm]{Corollary}

\newtheorem{prop}[thm]{Proposition}

\newtheorem{rem}[thm]{Remark}

\newtheorem{ex}[thm]{Example}

\usepackage[all]{xy}
\title{Toroidal Grothendieck rings and cluster algebras}
\author{Laura Fedele and David Hernandez}

\address{Universit\'e de Paris, Institut de
  Math\'ematiques de Jussieu-Paris Rive Gauche, CNRS, F-75013 Paris, France}

\email{laura.fedele@imj-prg.fr}

\email{david.hernandez@u-paris.fr}

\begin{document} 

\begin{abstract} 
We study deformations of cluster algebras with several quantum parameters, called
toroidal cluster algebras, which naturally appear in the study of Grothendieck rings of representations 
of quantum affine algebras. In this context, we construct toroidal Grothendieck rings and we establish these
are flat deformations of Grothendieck rings. We prove that for a family of monoidal categories
$\mathscr{C}_1$ of simply-laced quantum affine algebras categorifying finite-type cluster algebras, 
the toroidal Grothendieck ring has a natural structure of a toroidal cluster algebra.
\end{abstract}

\maketitle

\vskip 4.5mm

\tableofcontents

\section{Introduction}

Monoidal categories of finite-dimensional representations of affine quantum groups 
and Yangians have been studied from various points of view. The reader may refer to \cite{k, o} for recent
important developments. Deformations of the corresponding Grothendieck rings have been introduced from
the study of quantum $W$-algebras \cite{Fre:def}, of convolution rings obtained from quiver varieties \cite{Nak:quiver, VV:qGro} and 
of vertex operators \cite{H1}. One important application is the proof by Nakajima \cite{Nak:quiver} 
that one can calculate algorithmically the character of simple finite-dimensional representations of simply-laced quantum affine algebras. 
Although the deformations above have similarities, they give rise to different Poisson structures.
Moreover, the approach by the second author in \cite{H1} gives naturally additional deformation parameters 
whose corresponding deformations have not been exploited so far.
We propose to consider multi-deformed Grothendieck rings which encapsulate 
simultaneously various deformations of the Grothendieck rings.

To do this, let us recall that cluster theoretic methods to study the representation
theory of affine quantum groups have been introduced in \cite{hl}.
Indeed, Grothendieck rings of certain monoidal categories of such representations 
have a natural cluster algebra structure (see \cite{hlrev} for a recent review). 
That is why cluster algebras and their deformations give a natural framework to study
deformations of such Grothendieck rings.

In this perspective, we study toroidal cluster algebras, that is deformations of cluster algebras 
with several quantum parameters. We investigate several general properties of toroidal cluster algebras, 
some of them being new as far as we know. We construct many examples, partly from our study of quantum affine
algebras, but the positive part of certain multi-parameter quantum groups provide also examples.
The notion of a quantum cluster algebra with more than one quantum parameter also appeared in 
the work of Goodearl-Yakimov \cite{GYa, GYb}, with a different motivation related to the study 
of quantum nilpotent algebras. 

Then we construct toroidal Grothendieck rings, that is 
certain deformations of Grothendieck rings with several variables. We follow
the approach to quantum Grothendieck rings in \cite{H1}, that is we 
first prove the existence of classes of fundamental representations in a certain 
quantum torus (with several parameter in our case) by using a relevant algorithm. 
Then it turns out that the deformation of the Grothendieck
ring is the subring of the ambient quantum torus generated by the classes of fundamental representations.
But in the toroidal setting we have to deal with certain specific problems so that the deformation is flat. Indeed, 
the ordered product of classes of fundamental representations, that are called classes
of standard modules, do not give a basis, in opposition to the classical and 
quantum cases. We overcome this difficulty by introducing certain new relations in the ambient (multi-parameter quantum)
tori, and we establish the flatness of the deformation in this context. This is our first main result (Theorem \ref{mainflat}).

We study toroidal Grothendieck rings associated to different monoidal subcategories of the category $\mathscr{C}$ of finite-dimensional representations of the 
quantum affine algebras. In particular we focus on remarkable monoidal subcategories $\mathscr{C}_1$ introduced in \cite{hl} and which categorify cluster algebras of finite-type. 
Our second main result (Theorem \ref{thm:toroidaliso}) states that the toroidal Grothendieck ring of $\mathscr{C}_1$ gives after a suitable specialization a toroidal cluster algebra with two independent parameters. 
Moreover, the classes of fundamental representations correspond to toroidal cluster variables, certain distinguished elements of the toroidal cluster algebra.
From the results in \cite{HL:qGro, hljems, b, Q}, it was known that the deformation of Grothendieck rings with 
one quantum parameter corresponds to a quantization of the cluster algebra structure on commutative Grothendieck rings. 
It is surprising to us that the dependance in certain additional parameters also 
comes from certain quantum cluster algebra structures. Besides, our result gives an incarnation
 of a toroidal cluster algebra for each finite cluster type. 

Other interesting phenomena arise from the examples we consider: 
for instance, from the toroidal Grothendieck ring of the category $ \mathscr{C}_{\mathcal{Q}}$ of \cite{HL:qGro} 
in type $A_2$ one can recover the positive part of a multi-parameter quantum group.

To sum-up, the main motivation for our work is threefold: we aim at giving a natural framework to handle the various 
known deformations of Grothendieck rings of representations of quantum affine algebras (as well as multi-parameters quantum groups),
involving natural deformations of important relations such as $T$-systems or Serre relations. We also aim at obtaining natural examples 
of toroidal cluster algebras, that are deformations of cluster algebras with multi-parameters. Ultimately we hope that these new structures 
will be useful to understand better the intricate structure of categories of representations of quantum affine algebras. 
We hope this work gives new motivations to study the theory of toroidal cluster algebras.

The paper is organized as follows.

In Section \ref{un}, we give the setup for toroidal cluster algebras and discuss 
several results for quantum cluster algebras which can be naturally extended to the toroidal setting : the Laurent
phenomenon (Theorem \ref{thm:laurentphen}), the positivity (see Theorem \ref{pos}) and the invariance of the exchange graph (see Theorem \ref{thm:exchgraph}). 
In Section \ref{deux} we consider multi-parameter quantum tori which appear naturally in the study of the representation theory 
of quantum affine algebras. In Section \ref{trois}, we give a brief review on finite-dimensional representations of quantum affine algebras.
In Section \ref{quatre}, we introduce multi-parameter deformations of the Grothendieck ring of various monoidal categories : 
we introduce a specific quotient of a multi-parameter quantum torus (Definition \ref{quotientorus}) in which we construct the
toroidal Grothendieck ring (Definition \ref{torgr}) which is proved to be a flat deformation (Theorem \ref{mainflat}). 
We prove that the classes of fundamental representations provide a generating family of the toroidal Grothendieck ring (Proposition \ref{fundgen}).
In Section \ref{cinq}  we establish that toroidal Grothendieck rings
of certain monoidal categories $\mathscr{C}_1$ are toroidal cluster algebras in all $ADE$-types (Theorem \ref{thm:toroidaliso}). 
In Section \ref{seven}, we study the quasi-commutation relation for remarkable pairs in the multi-parameters quantum tori : we 
prove it depends on a single parameter, even in the toroidal setting. In Section \ref{eight} we discuss various questions we would like to address in the future in relation to the 
main results of this paper.

{\bf Acknowledgment : } The authors would like to thank the referee for his comments and remarks. 
The authors would like to thank also Martina Lanini, Bernard Leclerc and Bernhard Keller for discussions and references. 
The authors were supported by the European Research Council under the European Union's Framework Programme H2020 with ERC Grant Agreement number 647353 Qaffine.

\section{Toroidal cluster algebras setup}\label{un}

We give the setup for toroidal cluster algebras, mimicking the definition of quantum cluster algebras \cite{bz}. The more recent \cite{GYa} gives an alternative definition and a uniform approach. Several results for quantum cluster algebras can be naturally extended to the toroidal setting : the Laurent
phenomenon established in \cite{bz, GYa} (see Theorem \ref{thm:laurentphen}), the positivity from the main results in \cite{d} (see Theorem \ref{pos}) 
and as a consequence the invariance of the exchange graph (see Theorem \ref{thm:exchgraph}).

\subsection{Setup}\label{sec:setup}

Let $1\leq m\leq n$ be integers. We also fix $r\geq 1$ which we call the number of parameters. 
These parameters are formal variables that we denote by $t_1, t_2, \cdots, t_r$.

We consider a based multi-parameter quantum torus $\mathcal{T}$. It is the algebra over $\ZZ[t_1^{\pm \frac{1}{2}},\cdots, t_r^{\pm \frac{1}{2}}]$ with generators 
$X_1,X_2, \cdots, X_n$ that quasi-commute, that is they are subject to relations : 
$$X_i * X_j = \left(\prod_{1\leq a \leq r} t_a^{\Lambda_a(i,j)}\right)  X_j * X_i.$$
Here the 
$$\Lambda_a : \{1,\cdots, n\}^2\rightarrow \ZZ$$ 
are skew-symmetric maps. The matrices $(\Lambda_a(i,j))_{1\leq i,j\leq n}$ are the corresponding quasi-commutation matrices.
$\mathcal{T}$ is an Ore domain, and it is contained in its skew-field of fractions $\mathcal{F}$.

We remark that in \cite{GYa} the quasi-commutation relations are denoted by $ X_i * X_j = q_{ij} X_j * X_i $, where the $q_{ij}$'s are invertible elements in a base field $\mathbb{K}$, not necessarily given by powers of a same element \textit{q} (the latter special situation 
is called \textit{uniparameter} quantum torus case).
Thus, our notation compares to the one in \cite{GYa} if we take $\mathbb{K} = \mathbb{Q}(t_1^\frac12, \ldots, t_r^\frac12)$, $ D = \mathbb{Z}[t_1^{\pm \frac12},\ldots, t_r^{\pm \frac12}]$ and each $q_{ij} $ corresponds to the product $\prod_{1\leq a \leq r} t_a^{\Lambda_a(i,j)}$.
Our choice of working with the parameters $ t_a $ is motivated by our examples where the quantum parameters naturally appear in this form. 

The datum of $(\Lambda_a)_{1\leq a\leq r}$ is equivalent to the datum of $r$ Poisson brackets $\{,\}_a$ on the commutative polynomial ring $\ZZ[X_i]_{1\leq i\leq n}$ :
$$\{X_i, X_j\}_a = \Lambda_a(i,j) X_iX_j.$$
These Poisson brackets $\{,\}_a$ are compatible, that is any linear combination of the Poisson brackets $\{,\}_a$ is a Poisson bracket.

As for quantum tori with one quantum parameter, we have a $\ZZ[t_1^{\pm \frac{1}{2}},\cdots, t_r^{\pm \frac{1}{2}}]$-basis of the quantum torus given by commutative monomials
$$\prod_{1\leq i\leq n} X_i^{u_i} =  \left( \prod_{1\leq a\leq r}  t_a^{\frac{1}{2}\sum_{1\leq i < j\leq n}u_i u_j\Lambda_a(j,i)} \right)  \underset{1\leq i\leq n}{\overset{\rightarrow}{\mbox{\Large *}}} X_i^{u_i},$$
where the $u_i$ are arbitrary integers. In the following we will use the notation $\prod$ for commutative monomials as in this last formula.

For $1\leq a \leq r$, we have a unique specialization morphism sending commutative monomials to commutative monomials
$$\pi_{\mathcal{T},a} : \mathcal{T} \rightarrow \mathcal{T}_a$$
where $\mathcal{T}_a$ is the quantum torus defined from $\mathcal{T}$ by setting all $t_{a'} = 1$ if $a'\neq a$.

\begin{defi} A toroidal seed in $\mathcal{F}$ is a collection 
$$\mathcal{S} = (Y_1,\cdots, Y_n; \widetilde{B})$$ 
where the $Y_i\in \mathcal{F}$ and $\widetilde{B}$ is a skew-symmetric $n\times m$ integer matrix satisfying the following properties : 

(i) the $Y_i$ are algebraically independent, generate the field $\mathcal{F}$ and quasi-commute with quasi-commutation matrices denoted by $\Lambda_{a,\mathcal{S}}$.

(ii) for any $1\leq a\leq r$, the $m\times n$-matrix $\widetilde{B}^T \Lambda_{a,\mathcal{S}}$ has an $m\times m$-left block diagonal of constant 
sign and an $m\times (n-m)$ right block equal to zero.
If all entries of the diagonal left block are constant equal to $ k \in \mathbb{Z}_{>0}$, we will use the shorthand $\widetilde{B}^T \Lambda_{a,\mathcal{S}} = \begin{pmatrix} k \,\text{Id}_m \mid 0 \end{pmatrix}$.
\end{defi}

The $Y_i$ are called the toroidal cluster variables of the toroidal seed. We denote by $\mathcal{T}_{\mathcal{S}}$ the quantum torus they generate.

\begin{rem}\label{prere} (i) The second condition means that for $1\leq a\leq r$, $(\widetilde{B}, \Lambda_{a,\mathcal{S}})$ is a compatible pair in the sense of \cite{gsv, bz}. 
In particular $\widetilde{B}$ has full-rank and its principal part $(\widetilde{B})_{1\leq i,j\leq m}$ is skew-symmetrizable. If $r = 1$, we recover the definition 
of a quantum seed.

(ii) The maximal number of independent parameters, that is the dimension of the space generated by the $ \Lambda_{a,\mathcal{S}}$, is known. The matrix $\widetilde{B}$ is naturally associated to a quiver $Q$ with $n$ vertices. It admits a subquiver $Q'$ associated to the principal part of $\widetilde{B}$. By \cite[4.1.3.]{gsv} the maximal number of independent parameters is equal to the number of 
connected components of $Q'$ plus $\begin{pmatrix}n- m \\ 2  \end{pmatrix}$. In most cases we will study, $Q'$ will be connected, and this number will be equal to $1 + \begin{pmatrix}n- m \\ 2  \end{pmatrix}$. 

(iii) The number $n - m$ is the number of coefficients in the standard terminology of cluster algebras, that is the variables $ (Y_i)_{m+1\leq i\leq n}$. The quantity $\begin{pmatrix}n- m \\ 2  \end{pmatrix}$ corresponds to the number of quasi-commutation relations between the coefficients.

(iv) For $1\leq a\leq r$, we may set  
$$t_1 = \cdots = t_{a-1} = t_{a+1} = \cdots = t_r = 1$$ 
in $\mathcal{T}_{\mathcal{S}}$, that is we consider a
corresponding specialization map $\pi_{\mathcal{T}_\mathcal{S},a}$ as above. We get a quantum seed of quantum parameter $t_a$ in the sense of \cite{bz}.
More generally, for any $(a_1,\cdots, a_r)\in (\ZZ_{>0})^r$ and $t$ an indeterminate, the specialization 
$$(t_1,\cdots,t_r)\mapsto (t^{a_1},\cdots, t^{a_r})$$
defines a quantum seed of quantum parameter $t$. Indeed the linear combination 
$$\{,\} = \sum_{1\leq i\leq r}a_i\{,\}_i$$ 
is a Poisson-bracket and the corresponding skew-symmetric $\Lambda$ is compatible with $\widetilde{B}$.

(v) The results in this paper (except for Theorem \ref{pos}) can be straightforwardly
generalized to those for skew-symmetrizable toroidal cluster algebras, that is with $\widetilde{B}$ skew-symmetrizable.
\end{rem}

For a toroidal seed $\mathcal{S} = (Y_1,\cdots, Y_n; \widetilde{B})$ and $1\leq k\leq m$, the toroidal mutation in the direction $k$ 
$$\mu_k(\mathcal{S}) = (Y_1,\cdots, Y_{k-1}, Y_k',Y_{k+1},\cdots,Y_n; \widetilde{B}'),$$
is defined in the following way : $\widetilde{B}' = \mu_k(\widetilde{B})$ is obtained by the usual mutation rule \cite{fz} : 
\begin{equation}\label{20190801:eq2}
\widetilde{B}_{i,j}' = \begin{cases}-\widetilde{B}_{i,j}  &\text{if $i = k$ or $j = k$,}
\\\widetilde{B}_{i,j} + \widetilde{B}_{i,k}\widetilde{B}_{k,j}   & \text{if $\widetilde{B}_{i,k} > 0$ and $\widetilde{B}_{k,j} > 0$,}
\\ \widetilde{B}_{i,j} - \widetilde{B}_{i,k}\widetilde{B}_{k,j}  & \text{if $\widetilde{B}_{i,k} < 0$ and $\widetilde{B}_{k,j} < 0$,}
\\ \widetilde{B}_{i,j} & \text{otherwise.}  \end{cases}
\end{equation}
The new toroidal cluster variable $Y_k'$ is obtained by the toroidal mutation rule : 
\begin{equation}\label{mutrel}Y_k' * Y_k = u \prod_{b_{i,k} > 0} Y_i^{b_{i,k}}+ v \prod_{b_{i,k} < 0} Y_i^{-b_{i,k}}.\end{equation}
Here $u, v$ are the Laurent-monomials in the quantum parameters $t_1^{\frac{1}{2}},\cdots, t_r^{\frac{1}{2}}$ defined by : 
$$u^2 Y_k * \prod_{b_{i,k} > 0} Y_i^{b_{i,k}} = \prod_{b_{i,k} > 0} Y_i^{b_{i,k}} * Y_k 
\text{ and }v^2 Y_k * \prod_{b_{i,k} < 0} Y_i^{-b_{i,k}} = \prod_{b_{i,k} < 0} Y_i^{-b_{i,k}} * Y_k.$$
Moreover, the quasi-commutation matrices $\Lambda_{a,\mathcal{S}}$ mutate according to the following formula:
\begin{equation}\label{20190801:eq1}
\Lambda_{a,\mathcal{S}}' = \mu_k(\Lambda_{a,\mathcal{S}}) = E_k^T \Lambda_{a,\mathcal{S}} E_k\,,
\end{equation}
where $E_k$ is the $ n \times n $ matrix with entries
$$ (E_k)_{i,j} = \begin{cases} \delta_{ij} & \text{if } j \neq k, \\ -1 & \text{if } i = j = k, \\ \text{max}(0,-B_{i,k}) & \text{if } i \neq j = k \end{cases} $$

Both the results of Propositions \ref{prop:mutationtoroidalseed}, \ref{prop:mutationinvolutive} below also appeared in \cite{GYa}. We include here a proof for completeness.
\begin{prop}\label{prop:mutationtoroidalseed}
 The toroidal mutation of a toroidal seed in $\mathcal{F}$ is a toroidal seed in $\mathcal{F}$.
\end{prop}

\begin{proof} Let $\mathcal{S} = (Y_1,\cdots, Y_n; \widetilde{B})$ and let us keep the notations above 
for the mutated seed $\mathcal{S}' = \mu_k(\mathcal{S})$ for a certain $1\leq k\leq m$. 
From the toroidal mutation formula (\ref{mutrel}), the toroidal cluster variables in $\mathcal{S}'$ 
are algebraically independent. Let us prove they quasi-commute. 
Let $k'\neq k$ and we prove that $Y_k'$ and $Y_{k'}$ quasi-commute. As $Y_k$ and $Y_{k'}$ quasi-commute, 
it suffices to check that $Y_{k'}$ quasi-commute with the two terms
in the sum in (\ref{mutrel}) and that the quasi-commutation parameters are the same for both. 
For each $1\leq a\leq r$, we can apply $\pi_{\mathcal{S},a}$ to the toroidal mutation formula (\ref{mutrel}) and we get a mutation relation for the corresponding 
quantum cluster algebra. By the result on quantum cluster algebras \cite{bz}, the power of each 
quantum parameter $t_a$ is the same for both quasi-commutation relations. Hence we obtain the result. 
To conclude, again from the result on quantum cluster algebras, we get the compatibility in the mutated seed for each $1\leq a\leq r$.
\end{proof}

The quantum torus $\mathcal{T}_\mathcal{S}$ has an antimultiplicative bar-involution defined by $\overline{Y_i} = Y_i$ for $1\leq i\leq n$ and 
$\overline{t_a} = t_a^{-1}$ for $1\leq a\leq r$. It can be extended to the ring of Laurent polynomials in the $Y_i$ and the $t_a^{\frac{1}{2}}$. The commutative
monomials are bar-invariant.

\begin{prop}\label{prop:mutationinvolutive} The mutation of toroidal seeds is involutive.\end{prop}

\begin{proof} The fact that the mutation of matrices is involutive is well-known. Then we note that the toroidal cluster mutation
(\ref{mutrel}) is equivalent to 
$$Y_k * Y_k' = u^{-1} \prod_{b_{i,k} > 0} Y_i^{b_{i,k}}+ v^{-1} \prod_{b_{i,k} < 0} Y_i^{-b_{i,k}}.$$
It suffices to apply the antimultiplicative bar-involution discussed above. We get exactly the toroidal mutation relation from the toroidal $\mathcal{S}'$ to $\mathcal{S}$.
\end{proof}

\begin{defi}
Let $\mathcal{S} = (Y_1,\cdots, Y_n;\widetilde{B})$ be a toroidal seed in $\mathcal{F}$. The associated toroidal cluster algebra 
$$ \mathcal{A}_{tor}(\mathcal{S})\subset \mathcal{F}$$ 
is the $\ZZ[t_1^{\pm \frac12},\cdots, t_r^{\pm \frac12}]$-subalgebra of $\mathcal{F}$ generated by all toroidal cluster variables of all toroidal seeds 
obtained by iterated toroidal mutations.
\end{defi}

\subsection{An example}\label{sec:afirstexample}
Let $1 = m\leq 3 = n$, and $ r = 2$. 
Let $\mathcal{T}$ be the algebra over $\ZZ[t_1^{\pm \frac{1}{2}}, t_2^{\pm \frac{1}{2}}]$ with generators 
$X_1,X_2, X_3$ subject to the quasi-commutation relations
$$X_i * X_j = t_1^{\Lambda_1(i,j)}t_2^{\Lambda_2(i,j)} X_j * X_i,$$
where
$$\Lambda_1 = \begin{pmatrix} 0 & 1 & -1 \\ -1 & 0 & 0 \\ 1 & 0 & 0\end{pmatrix}
\,,
\quad
\Lambda_2 = \begin{pmatrix} 0 & 0 & -1 \\ 0 & 0 & 0 \\ 1 & 0 & 0\end{pmatrix}.$$ 
Finally, let
$$ \widetilde{B}^T = \begin{pmatrix} 0 & -1 & 1 \end{pmatrix}.$$
Then the collection $\mathcal{S} = (X_1,X_2,X_3, \widetilde{B}) $ is a toroidal seed. In fact, we have
$$\widetilde{B}^T\Lambda_1 = \begin{pmatrix} 2 & 0 & 0\end{pmatrix}
\,,
\quad
\widetilde{B}^T\Lambda_2 = \begin{pmatrix} 1 & 0 & 0\end{pmatrix},$$
hence $(\widetilde{B},\Lambda_1)$ and $(\widetilde{B},\Lambda_2)$ are compatible. Note that the commutation relations give 
$$t_1^{-1} X_1 \ast X_2 = X_2 \ast X_1 \text{ and } t_1t_2 X_1 \ast X_3 = X_3 \ast X_1.$$ 
We can mutate the toroidal seed $\mathcal{S}$ only in direction $ 1 $ and obtain a new toroidal seed 
$$\mu_1(\mathcal{S}) = \mathcal{S}' = (X_1',X_2,X_3, \widetilde{B}'),$$
where the new toroidal cluster variable $X_1'$ is defined as
\begin{equation}\label{20190503:eq1}
X_1' \ast X_1 = t_1^{-\frac12} X_2 + (t_1t_2)^\frac12 X_3.
\end{equation}
We have $\widetilde{B}' = - \widetilde{B}$ and the quasi-commutation matrices of the new seed are $\Lambda_{a,\mathcal{S}'} =  -\Lambda_a $ ($ a=1,2$), hence the pairs $(\widetilde{B}' , \Lambda_{a,\mathcal{S}'})$ ($ a=1,2$) are clearly compatible.

The toroidal cluster algebra $\mathcal{A}_{tor}(\mathcal{S})$ has two toroidal seeds $\mathcal{S}$ and $\mathcal{S}'$. 
So it is the subalgebra of the fraction field of $\mathcal{T}$ generated by the $4$ toroidal cluster variables : 
$$X_1, X_2, X_3, X_1'.$$ 
Note that the parameters $t_1$, $t_2$ are independent, and that in this case $2$ is the maximal number of independent parameters. 

We have two different quantum cluster algebras corresponding to the specialization at
$t_1 = 1$ or $t_2 = 1$ relatively to the initial seed $\mathcal{S}$, obtained by applying the specialization morphisms $\pi_{\mathcal{T}_\mathcal{S},1}$, $\pi_{\mathcal{T}_\mathcal{S},2}$ to 
$\mathcal{A}_{tor}(\mathcal{S})$. The specialization at $t_2 = 1$ is a certain $\mathbb{Z}$-form of the positive part of the quantum group associated to $sl_3$ : 
$$\pi_{\mathcal{T}_\mathcal{S},2}(\mathcal{A}_{tor}(\mathcal{S}))\simeq \U_{t_1}(\mathfrak{n})_{\mathbb{Z}}\subset \U_{t_1}(sl_3).$$
Without a specialization, we have also the relation : 
$$
X_1 \ast X_1' = t_1^{\frac12} X_2 + (t_1t_2)^{-\frac12} X_3.
$$
Hence $\{X_1, X_1'\}$ generates $ \mathcal{A}_{tor}(\mathcal{S})_{\mathbb{Q}}$, where the latter denotes the same toroidal cluster algebra $ \mathcal{A}_{tor}(\mathcal{S}) $ after extending the scalars from  $ \mathbb{Z}[t_{1}^{\pm \frac 12}, t_{2}^{\pm \frac12}] $ to $ \mathbb{Q}(t_{1}^{\frac 12}, t_{2}^{\frac12})$. These toroidal cluster variables satisfy: 
$$X_1*X_1*X_1' - (t_1^{-2}t_2^{-1}+1)t_1 X_1*X_1'*X_1 + t_2^{-1}X_1' * X_1 * X_1 = 0,$$
$$X_1' * X_1' * X_1 - (t_1^{2}t_2+1)t_1^{-1} X_1'*X_1 * X_1' + t_2 X_1 * X_1' * X_1' = 0.$$
This is a complete presentation of the algebra $\mathcal{A}_{tor}(\mathcal{S})_{\mathbb{Q}}$. Indeed this  is equivalent to the fact that 
$$\{X_1^a X_2^bX_3^c|a,b,c\geq 0\}\cup \{(X_1')^aX_2^bX_3^c| a,b,c\geq 0\}.$$
are linearly free over $\mathbb{Z}[t_1^{\pm \frac12},t_2^{\pm \frac12}]$, which is true at $t_2 = 1$ as it is identified with the 
Lusztig dual canonical basis of $\U_{t_1}(sl_3)$.

We recognize the relations above as quantum Serre relations with $2$ parameters. The quantum Serre relations for the multi-parameter quantum group $ \mathcal{U}_{\mathbf{q}} (sl_3)$ in \cite{hpr} are
$$ e_1^2 e_2 - (q_{11}+1)q_{12} e_1e_2e_1 + q_{11}q_{12}^2 e_2 e_1^2 = 0,$$
$$ e_2^2 e_1 - (q_{22}+1)q_{21} e_2e_1e_2 + q_{22}q_{21}^2 e_1 e_2^2 = 0,$$
for some indeterminates $ \textbf{q} = (q_{11},q_{12},q_{21}, q_{22} ) $ over $\mathbb{Q}$ such that $ q_{12}q_{21} = q_{11}^{-1} = q_{22}^{-1} $. Hence, the two sets of quantum Serre relations coincide if we take $ q_{12} = t_1 $, $ q_{21} = t_1t_2$ and $ q_{11} = q_{22} = t_1^{-2}t_2^{-1}$, and therefore the multi-parameter quantum group $\mathcal{U}_{\mathbf{q}} (sl_3)$ is a two-parameter quantum groups as in \cite{hp} (cfr. \cite[Remark 9 (2)]{hpr}).
More precisely, $\mathcal{A}_{tor}(\mathcal{S})_{\mathbb{Q}}$ is isomorphic to the subalgebra $\mathcal{U}_{(t_1t_2)^{-1}, t_1}(\mathfrak{n})$ of a quantum group $\mathcal{U}_{(t_1t_2)^{-1},t_1}(sl_3)$ with two parameters described in \cite{hp}. We get as a by-product that there exists a $\mathbb{Z}$-form $\mathcal{U}_{(t_1t_2)^{-1}, t_1}(\mathfrak{n})_{\mathbb{Z}} \subset \mathcal{U}_{(t_1t_2)^{-1}, t_1}(\mathfrak{n})$
such that there is an embedding 
$$\mathcal{U}_{(t_1t_2)^{-1}, t_1}(\mathfrak{n})_{\mathbb{Z}}\hookrightarrow \ZZ[t_1^{\pm 1/2}, t_2^{\pm 1/2}, X_1^{\pm 1}, X_2^{\pm 1}, X_3^{\pm 1}].$$

\subsection{Toroidal Laurent phenomenon and positivity}

We have the following toroidal Laurent phenomenon. The generalization of the proof of the commutative Laurent phenomenon 
in \cite{bfz} to the quantum case in \cite[Section 5]{bz} gives also the result in the toroidal case as explained by 
Goodearl-Yakimov in \cite{GYa}.

\begin{thm}\cite[Theorem 2.15]{GYa}\label{thm:laurentphen} Let $\mathcal{S} = (Y_1,\cdots, Y_n;\widetilde{B})$ be a toroidal seed in $\mathcal{F}$. The associated toroidal cluster algebra is 
contained in the $\ZZ[t_1^{\pm \frac{1}{2}},\cdots, t_r^{\pm \frac{1}{2}}]$-subalgebra of $\mathcal{F}$ generated by the $Y_1^{\pm 1}, \cdots, Y_n^{\pm 1}$.
\end{thm}

Hence every toroidal cluster variable can be written in a unique way as a linear combination of commutative Laurent monomials $M$ 
in the $Y_1^{\pm 1}, \cdots, Y_n^{\pm 1}$ with coefficients 
$$P_M(t_1,\cdots, t_r)\in \ZZ[t_1^{\pm \frac{1}{2}},\cdots, t_r^{\pm \frac{1}{2}}].$$ 
These are called the coefficients of the Laurent expansion of the cluster variable with respect to the toroidal seed $\mathcal{S}$.

Let us state the positivity for toroidal cluster algebras. We explain it is a consequence of positivity of
quantum cluster algebras proved by Davison \cite{d}.

\begin{thm}\label{pos} The coefficients of the Laurent expansion of every toroidal cluster variable with respect to any toroidal seed in $\mathcal{A}_{tor}(\mathcal{S})$ are positive.
\end{thm}

\begin{proof}
It follows from \cite{d} that for any $(a_1,\cdots, a_r)\in (\ZZ_{>0})^r$, the
$$P_M(t^{a_1},\cdots, t^{a_r}) \in \ZZ[t^{\pm \frac{1}{2}}]$$
are positive Laurent polynomials in $t^{\frac12}$ (the fact that $P_M(1,\cdots, 1)\geq 0$ is proved in \cite{ls}).
Indeed, the corresponding specialization morphism 
$$(t_1,\cdots,t_r)\mapsto (t^{a_1},\cdots, t^{a_r})$$
defines a quantum seed from $\mathcal{S}$, see (iv) in Remark \ref{prere}.
For any Laurent monomial in $P_M$
$$m = t_1^{b_1}\cdots t_r^{b_r},$$ then there is 
a choice $(a_1,\cdots, a_r)$ so that $m$ is the only monomial contributing to $t^{\sum_i b_ia_i}$ in 
$$P(t) = P_M(t^{a_1},\cdots, t^{a_r}).$$ 
For example, for $A$ the maximum of the absolute values of the powers of the $t_i^{\frac{1}{2}}$ occurring in $P_M$, we may set $a_i = 10^{A \,i }.$
Indeed, the application
$$[-A,A]^r\rightarrow \ZZ \text{ so that } (\mu_1,\cdots, \mu_r)\mapsto \sum_{1\leq i\leq r} \mu_i 10^{A\, i}$$ 
is injective. As $P(t)$ is positive, then $ m $ must have a positive coefficient in $P_{M}$.
\end{proof}

The exchange graph of a cluster algebra is defined as the graph with 
the seeds as vertices, and the edges corresponding to seed mutations.
The toroidal exchange graph of a toroidal cluster algebra is defined in exactly the same way. To a toroidal cluster algebra is naturally associated its classical specialization 
which is the cluster algebra associated to the exchange matrix of one of the 
toroidal seeds.

\begin{thm}\label{thm:exchgraph}
The toroidal exchange graph of a toroidal algebra identifies with 
the exchange graph of its classical specialization.
\end{thm}

The proof relies on the result established in \cite{bz} for quantum cluster algebras.

\begin{proof}
There is a natural surjective map $\Psi$ from the toroidal exchange graph
to the exchange graph of the specialization. To prove it is a bijection, is suffices
to prove that two cluster variables $\chi$, $\chi'$ whose images in the classical specialization are 
equal coincide. We use the same strategy as in the proof of Theorem \ref{pos} above, that is we 
specialize to various quantum cluster algebras. We fix an initial seed and we consider the
Laurent developments of the various cluster variables : for a Laurent monomial $M$ of the 
cluster variables of the initial seed, we denote by $P_M(t_1,\cdots, t_r)$ and $P_M'(t_1,\cdots, t_r)$ the coefficient of $M$ in the respective developments. As above, 
it follows from the result \cite{bz} in the quantum case that for any $(a_1,\cdots, a_r)\in(\mathbb{Z}_{>0})^r$, 
$$P_M(t^{a_1},\cdots, t^{a_r}) = P_M'(t^{a_1},\cdots, t^{a_r}).$$
By an analog argument as in the proof of Theorem \ref{pos}, this implies that $P_M = P_M'$.
\end{proof}

As in the classical case, we say that a toroidal cluster algebra is of finite type if is has a finite number of toroidal cluster variables. 
Fomin and Zelevinsky \cite{fz2} gave a complete classification of cluster algebras of finite type, which turns out to mirror the Cartan-Killing 
classification of simple Lie algebras and finite root systems.
We say that two quivers associated to different seeds in a (toroidal) cluster algebra are mutation equivalent if one can be obtained from the 
other by performing a sequence of mutations of quivers.
Then, a cluster algebra is of (finite) type $X_n$ if its underlying quiver is mutation equivalent to an orientation of the Dynkin diagram of type $X_n$ 
 (the vertices corresponding to frozen variables are disregarded, that is we only consider the principal part of the quiver). 
Clearly, a cluster algebra has a finite number of cluster variables if and only if its exchange graph is finite. Thus, as a direct consequence of Theorem 
\ref{thm:exchgraph} we have the following : 
\begin{cor}
A toroidal cluster algebra $\mathcal{A}_{tor}(\mathcal{S})$ is of finite type $X_n$, if the principal part of the quiver associated to the initial seed $\mathcal{S}$ 
is mutation equivalent to an orientation of the Dynkin diagram of type $ X_n$.
\end{cor}

\section{Multi-parameter quantum tori}\label{deux}

We consider multi-parameter quantum tori which appear naturally in the study of the representation theory 
of quantum affine algebras and of certain formal power series with coefficients in Heisenberg algebras studied in \cite{H1}. 
These multi-parameter quantum tori will allow us to construct toroidal Grothendieck rings and examples of 
toroidal cluster algebras in the next Sections.

\subsection{Quantized Cartan matrix}
Let $\Glie$ be a simply-laced untwisted affine Kac--Moody
Lie algebra with underlying finite-dimensional simple Lie algebra $\overline{\Glie}$ of rank $n$. Set $I=\{1,\ldots, n\}$ and 
$C = (C_{i,j})_{i,j\in I}$ the Cartan matrix of $\overline{\Glie}$. Let 
$$\{\alpha_i\}_{i\in I}\text{ , }\{\alpha_i^\vee\}_{i\in I}\text{ , }\{\omega_i\}_{i\in I}\text{ , }\{\omega_i^\vee\}_{i\in I},$$ 
and $\overline{\mathfrak{h}}$ be the simple
roots, the simple coroots, the fundamental weights, the fundamental
coweights, and the Cartan subalgebra of $\overline{\Glie}$, respectively.  
We denote by $ \Delta$ be the root system of $ \mathfrak{g}$, and by $ \Delta_+ \subset \Delta $ the subset of positive roots.
We use the numbering of the Dynkin diagram as in \cite{kac}.  

Let $z$ be an indeterminate, and let $C(z)$ be the $n\times n$-matrix with entries
$$C_{i,j}(z) = [C_{i,j}]_z$$
Here for an integer $m$, $[m]_z = \frac{z^m - z^{-m}}{z - z^{-1}} = \sum_{h=0}^{m-1} z^{m-2h-1}$ is the standard quantum number.

Thus $C(1)$ is just the Cartan matrix $C$ of $\overline{\Glie}$.
Since $\det(C) \not = 0$, $\det(C(z))\not = 0$.  We denote by $\widetilde{C}(z)$ the inverse of the matrix $C(z)$.
This is a matrix with entries $\widetilde{C}_{ij}(z)\in\QQ(z)$. Explicit formulas can be found in \cite[Appendix A]{gtl}.
The entries of $\widetilde{C}(z)$ have power series expansions in $z$ of the form (see \cite{HL:qGro}) :
\begin{equation}\label{20190726:eq1}
\widetilde{C}_{i,j}(z) = \sum_{m \ge 1} \widetilde{C}_{ij}(m)\, z^{m}\text{, where $\widetilde{C}_{i,j}(m)\in\ZZ$.}
\end{equation}

We have the following periodicity property established in \cite{HL:qGro}, for $i,j\in I$ and $m \ge 1$ :
$$\widetilde{C}_{i,j}(m+2h) = \widetilde{C}_{i,j}(m),$$
where $h $ ($ = h^\vee $) is the (dual) Coxeter number of $\overline{\Glie}$.
Moreover, by \cite[Prop 2.1]{HL:qGro}, the following properties hold :
\begin{equation}\label{20190801:eq4}
\begin{array}{l}
\displaystyle{
\widetilde{C}_{i,j}(1) = \delta_{ij}
\,}\\
\displaystyle{
\widetilde{C}_{i,j}(m+1) +  \widetilde{C}_{i,j}(m-1) - \sum_{k \sim i} \widetilde{C}_{k,j}(m) = 0\,, \quad m \geq 1
\,.}
\end{array}
\end{equation}

\begin{rem}
When $ \overline{\mathfrak{g}}$ of Dynkin type $A_n$, we can give a very explicit description of the entries of the inverse of the quantum Cartan matrix $\widetilde{C}(z)$, which is easily derived by the formulas in \cite[Appendix A]{gtl}: For $1 \leq i \leq j \leq n$ we have
\begin{equation}\label{20190221:eq2}
\begin{array}{l}
\displaystyle{
\vphantom{\Big(}
\widetilde{C}_{i,j}(z)
=
\Big(
\sum_{a=0}^{i-1} z^{i+j-1-2a} - \sum_{a=-n+j-1}^{-n+i+j-2} z^{i+j-1-2a}
\Big)
\sum_{b\geq 0} z^{2(n+1)b}
\,,}
\end{array}
\end{equation}
while for $ i > j $ we use the fact that the inverse of the quantum Cartan matrix is symmetric: $ \widetilde{C}_{i,j}(z) = \widetilde{C}_{j,i}(z)$.
\end{rem}

\subsection{Heisenberg Lie-algebra and Frenkel-Reshetikhin currents}
Let $ q $ be a non-zero complex number which is not a root of unity.
Following \cite{H1}, we consider the Heisenberg algebra $\mathscr{H} $ as the $\CC$-algebra with generators $ a_i[m]$ ($i \in I, m \in \mathbb{Z} \setminus \{0\}$),
the central elements $ c_r $ ($r>0$), and relations ($i,j\in I, m, r \in \mathbb{Z}\setminus \{0\}$):
$$
[a_i[m], a_j[r]] = \delta_{m,-r}(q^m-q^{-m}) C_{i,j}(q^m)c_{|m|}.$$

For $ j \in I$, $ m \in \mathbb{Z}$, let moreover $ y_j[m] = \sum_{i \in I} \widetilde{C}_{i,j}(q^m)a_i[m] \in \mathscr{H}$.
We have
\begin{equation}\label{20190320:eq1}
\begin{array}{l}
\displaystyle{
\vphantom{\Big(}
[a_i[m], y_j[r]] = (q^{m } - q^{-m })\delta_{m,-r}\delta_{ij}c_{|m|}
,}\\
\displaystyle{
\vphantom{\Big(}
[y_i[m], y_j[r]] = \delta_{m,-r} \widetilde{C}_{j,i}(q^m)(q^{m} - q^{-m })c_{|m|}
.}
\end{array}
\end{equation}

Next, consider the $\CC$-algebra $ \mathscr{H}_h := \mathscr{H}[[h]]$. Following \cite{Fre, H1}, we consider certain invertible elements in $\mathscr{H}_h$ 
for $ (i, r) \in \hat{I}:= I \times \mathbb{Z} $  :
$$A_{i,r} = \text{exp} \Big( \sum_{m > 0} h^m a_i[m]q^{r m}\Big)  \text{exp} \Big(\sum_{m > 0} h^m a_{i}[-m]q^{-r m}\Big),$$
$$Y_{i,r} = \text{exp} \Big( \sum_{m > 0} h^m y_i[m]q^{r m}\Big)  \text{exp} \Big(\sum_{m > 0} h^m y_{i}[-m]q^{-r m}\Big).$$
Let $\mathfrak{U} \subset \mathbb{Q}(z)$ be the group of rational fractions of the form $\frac{P(z)}{Q(z)}$ where $ P(z) \in \frac{1}{2}\mathbb{Z}[z^{\pm 1}]$, $Q(z)\in \mathbb{Z}[z]$, $Q(0) = 1$ 
and the zeros of $Q(z)$ are roots of unity. By expanding $ Q(z)^{-1} $ in $ \mathbb{Z}[[z]] $ we have an embedding $ \mathfrak{U} \subset \frac{1}{2}\mathbb{Z}((z))$. For $R \in \mathfrak{U} $ we denote
$$ t_R = \text{exp} \Big(\sum_{m > 0} h^{2m} R(q^m)c_m\Big)\in  \mathscr{H}_h.$$

We denote by $ \widetilde{\mathscr{Y}}_\infty$ the $\mathbb{Z}$-subalgebra of $\mathscr{H}_h$ generated by the elements
$Y_{i,r}^{\pm 1}$, $A_{i,r}^{\pm 1}$, $t_R $
($(i,r)\in \hat{I}$, $R\in \mathfrak{U}$).
We denote by $*$ the product in this algebra. By \eqref{20190320:eq1}, the following quasi-commutation relations hold in $\widetilde{\mathscr{Y}}_\infty$ :
\begin{equation}\label{20190320:eq2}
\begin{array}{l}
\displaystyle{
\vphantom{\Big(}
A_{i,p}* Y_{j,s}* A_{i,p}^{-1}* Y_{j,s}^{-1} = t_{\delta_{ij}(z-z^{-1})(z^{(p-s)}-z^{(s-p)})}
,}\\
\displaystyle{
\vphantom{\Big(}
Y_{i,p} * Y_{j,s} * Y_{i,p}^{-1} * Y_{j,s}^{-1} = t_{\widetilde{C}_{j,i}(z)(z - z^{- 1})(z^{(p-s)}-z^{(s-p)})}
.}
\end{array}
\end{equation}

\subsection{Multi-parameter quantum tori}\label{sec:quantumtori}
Let $ \mathscr{Y} = \mathbb{Z}[Y_{i,r}^{\pm 1}  \mid (i, r) \in \hat{I}] $,
be the Laurent polynomial ring generated by a collection of commutative variables $Y_{i,r}$.
In  \cite{H1} the second author constructed, starting from $\widetilde{\mathscr{Y}}_\infty$, a one-parameter deformation of the ring $ \mathscr{Y}$. We consider here deformations with an arbitrary number of parameters.

Given $R\in \mathfrak{U}$ we can rewrite (cfr. \cite[Lemma 3.7]{H1}) 
$$ t_R = t_{\Big(\sum_{m \geq - M_R} R_m z^m\Big)} = \prod_{m\geq -M_R} \big(t_{z^m}\big)^{R_m}.$$
 For any $ r \in \mathbb{Z} $ we have the coefficient map $\pi_r: \mathfrak{U} \longrightarrow \frac{1}{2}\mathbb{Z}$, so that
$$ P = \sum_{r \geq - R } \pi_r(P) z^r\text{ for any $P\in\mathfrak{U}$.}$$
By definition, $ \pi_r(\widetilde{C}_{i,j}(z)) = \widetilde{C}_{i,j}(r)$.

Let us use the shorthand $ t_m$ for $ t_{z^m}$, $ m \in \mathbb{Z}$.
As the $Y_{i,r}$ are algebraically independent, 
 $\widetilde{\mathscr{Y}}_\infty$ can be presented (\cite[Lemma 3.9]{H1}) as
the $\mathbb{Z}[t_R \mid R \in \mathfrak{U}]$-algebra with the generators $Y_{i, r}^{\pm 1} $, $(i,r) \in \hat{I}$, and quasi-commutation relations
\begin{equation}\label{20190409:eq1}
Y_{i, p}\ast  Y_{j, s} = \left( \prod_{a\in \mathbb{Z}} {t_a}^{\mathcal{N}_a(i, p; j, s)}\right) Y_{j, s} \ast Y_{i, p},
\end{equation}
where the map $ \mathcal{N}_a: \hat{I} \times \hat{I} \longrightarrow \mathbb{Z}$ is given by
\footnote{This is obtained from the quasi-commutation relation in \eqref{20190320:eq2}, but with a sign change, so that it agrees with the formula in \cite{HL:qGro} corresponding to the case $ a = 0$.}
$$\mathcal{N}_a(i,p; j,s) =  \pi_a\Big( \widetilde{C}_{j,i}(z) (z - z^{- 1})(z^{s-p}-z^{p-s})\Big).$$
This is a quantum torus of infinite rank.

We can compute the value of $\mathcal{N}_a(i,p; j,s)$ explicitly and obtain
\begin{equation}\label{20190801:eq3}
\begin{array}{l}
\displaystyle{
\vphantom{\Big(}
\mathcal{N}_a(i,p; j,s)
=
\sum_{r \in \mathbb{Z}} \widetilde{C}_{j,i}(r)\big(\delta_{s-p+r+ 1 ,a} -\delta_{p-s+r+ 1,a} -\delta_{s-p+r- 1,a}+ \delta_{p-s+r-1,a}\big)
\,}\\
\displaystyle{
\vphantom{\Big(}
=
\widetilde{C}_{j,i}(p-s- 1+a) -\widetilde{C}_{j,i}(s-p- 1+a) -\widetilde{C}_{j,i}(p-s+ 1+a) + \widetilde{C}_{j,i}(s-p+1+a) 
.}
\end{array}
\end{equation}

Note that the map $ \mathcal{N}_a: \hat{I} \times \hat{I} \longrightarrow \mathbb{Z}$
is clearly skew-symmetric, namely 
$$\mathcal{N}_a(i,p; j,s) = - \mathcal{N}_a(j,s; i,p).$$
Moreover, it only depends on the difference $ s-p$. Thus, $ \mathcal{N}_a(i,p; j,s) = \mathcal{N}_a(i,0; j,s-p)$ and $ \mathcal{N}_a(i,p; j,p) = 0$ for any $ p \in \mathbb{Z}$, hence the variables $ Y_{i, p} $ and $ Y_{j, p} $ commute for any $ i,j \in I$.

\begin{rem}\label{20190807:rem1}
\begin{enumerate}
\item
It follows from \eqref{20190801:eq3} that $\mathcal{N}_{-a}(i,0; j,a) = \delta_{i,j} $ ($a >0$), and $\mathcal{N}_a(i,p; j,s) = 0 $ for $ a < - \mid s-p \mid$. 
Hence, for any pair $ Y_{i,p}$, $ Y_{j,s}$, in the RHS of \eqref{20190409:eq1} there is 
a minimum $a$ so that the parameter $t_a$ occurs, while such a maximum $a$ does not exist in general.
\item
We want to remark that even though in the RHS of Equation \eqref{20190409:eq1} we allow an infinite product of $ t_{R}$'s, in practice we will always consider only a finite number of parameters. In fact, thanks to \eqref{20190801:eq3} and the periodicity condition $\widetilde{C}_{i,j}(m+2h) = \widetilde{C}_{i,j}(m)$ for $ m \geq 1$, we obtain that $ \mathcal{N}_{a+2h}(i,0;j,s) = \mathcal{N}_{a}(i,0;j,s) $ for $ a \geq s+2$.
\end{enumerate}
\end{rem}

The following constructions are analogous to the corresponding ones for the one-parameter deformation.

Given a family of integers $u_{i,p}$, $(i,p) \in \hat{I}$, with finitely many nonzero components we have the commutative monomial
$$
\prod_{(i,p) \in \hat{I}} Y_{i,p}^{u_{i,p}}
=
\prod_{a \in \mathbb{Z}}
{t_a}^{\frac12\sum_{(i,p)<(j,s)} u_{i,p}u_{j,s}\mathcal{N}_a(j,s;i,p)} \overset{\longrightarrow}{*}_{(i,p) \in \hat{I}} Y_{i,p}^{u_{i,p}}
\,,
$$
where the arrow means that the product is ordered according to a certain ordering of $ \hat{I}$, arbitrarily chosen.
It follows that given $ m_1 = \prod_{(i,p)\in\hat{I}} Y_{i,p}^{u_{i,p}(m_1)} $ and $m_2 = \prod_{(j,s)\in\hat{I}} Y_{j,s}^{u_{j,s}(m_2)}$ commutative monomials,
their non-commutative $\ast$-product in $\widetilde{\mathscr{Y}}_{\infty}$ is
\begin{equation}\label{20190429:eq1}
m_1 \ast m_2 = \prod_{a \in \mathbb{Z}}
{t_a}^{\frac12 D_a(m_1,m_2)}m_1m_2
\,,
\end{equation}
$$
\text{where }D_a(m_1,m_2) = \sum_{(i,p),(j,s)\in\hat{I}} u_{i,p}(m_1) u_{j,s}(m_2) \mathcal{N}_a(i,p;j,s)
\,.
$$
We say that a commutative $m = \prod_{(i,p)\in\hat{I}} Y_{i,p}^{u_{i,p}(m)} $
is dominant if $ u_{i,p}(m) \geq 0 $ for all $ (i,p) \in \hat{I} $. 

For $(i,r) \in \hat{I}$, let us denote by the same symbol $A_{i,r}$ the commutative monomial associated to 
the formal power series $A_{i,r}$ defined above : 
\begin{equation}\label{20190624:eq1}
A_{i,r} = Y_{i,{r-1}} Y_{i,{r+1}}
\Big(\!\prod_{j | C_{ji}=-1} Y_{j,r}^{-1}\Big).
\end{equation}

\begin{rem}\label{20190726:rem1}
The one-parameter deformation of the Laurent ring $ \mathscr{Y}$ described in \cite{H1} is obtained as a particular case, when we consider the quotient $\mathscr{Y}_t$ of $\widetilde{\mathscr{Y}}_\infty$
 obtained by the relations $t_a = 1$ if $a\neq 0$. We set $t = t_0$ and by \cite[Theorem 3.11]{H1}, $ \mathscr{Y}_t $ is the $ \mathbb{Z}[t^{\pm 1}]$-algebra with generators the variables $ Y_{i, r}^{\pm 1} $, $(i,r) \in \hat{I}$, and commutation relations
\begin{equation}\label{20190430:eq2}
Y_{i, p} \ast Y_{j, s} = t^{\mathcal{N}(i, p; j, s)} Y_{j, s} \ast Y_{i, p},
\end{equation}
where $ \mathcal{N}(i,p;j,s)$ coincides
(up to a sign\footnote{Here, we are using the same product as in \cite{HL:qGro}, which differs from the one in \cite{H1} by replacing $ t $ with $ t^{-1}$. This change amounts to a sign change in the RHS of the formulas for $ \mathcal{N}_a(i,p; j,s)$.})
with the exponent $ \mathcal{N}_0(i,p;j,s) $ defined above.
\end{rem}

\subsection{Finite rank multi-parameter quantum tori}

Let us give several examples of finite rank multi-parameter quantum tori which will appear in the following 
when we will study monoidal subcategories of finite-dimensional representations of quantum affine algebras.

 Let $\mathcal{Q} $ be an orientation of the Dynkin diagram of the Lie algebra $ \overline{\mathfrak{g}}$.
A height function $ \xi: I \longrightarrow \mathbb{Z} $ on $\mathcal{Q}$ is a function satisfying
$ \xi_j = \xi_i -1 $
whenever there is an arrow $ i \rightarrow j $ between the nodes $ i, j \in \mathcal{Q}$.
As $ \mathcal{Q} $ is connected, we can fix a height function $ \xi $ and any two height functions would differ by a constant.

{\bf Example 1} : we suppose the orientation is bipartite, that is every vertex of $\mathcal{Q}$ is a sink or a source. We assume $ \xi_i = 1 $ for any $ i \in I $ source.
Then $\widetilde{\mathscr{Y}}_{\infty,1}$ is the $\mathbb{Z}[t_R \mid R \in \mathfrak{U}]$-subalgebra of $\widetilde{\mathscr{Y}}_\infty$  generated by the variables $Y_{i,p}^{\pm 1}$,
where $i\in I$, $p = \xi_i , \xi_i + 2$. This is a quantum torus of rank $ 2n$. $\widetilde{\mathscr{Y}}_{\infty,1}$ 
is the corresponding extended torus.

{\bf Example 2} : we suppose $\overline{\mathfrak{g}}$ be of type $A_n$ and the orientation of $ \mathcal{Q} $ is linear with $\xi_i = i$.
Then $ \widetilde{\mathscr{Y}}_{\infty,ob} $ is the $\mathbb{Z}[t_R \mid R \in \mathfrak{U}]$-subalgebra of $\widetilde{\mathscr{Y}}_\infty$ 
generated by the variables $Y_{i,p}^{\pm 1} $,
where $ 1 \leq i \leq n $ and $ p = i-1, i+1$. This is a quantum torus of rank $ 2n$. 
$ \widetilde{\mathscr{Y}}_{\infty, ob}$ is the corresponding extended torus.

{\bf Example 3} : for an arbitrary orientation of $\mathcal{Q}$, we set
$$ \hat{I}_\xi := \{ (i,p) \in I \times \mathbb{Z} \mid p - \xi_i \in 2\mathbb{Z}\}, $$
and $ \hat{\Delta} := \Delta \times \mathbb{Z}$, which gives a labeling for the infinite repetition quiver
$ \widehat{\mathcal{Q}} $ attached to $ \mathcal{Q}$.
There is a bijection $\varphi: \hat{I}_\xi \rightarrow \hat{\Delta}$ (see \cite{HL:qGro} for details).
Define
$$ \hat{I}_{\mathcal{Q}} := \varphi^{-1}(\Delta_+ \times \{0\}) \subset \hat{I}_\xi,$$
and let $ \widetilde{\mathscr{Y}}_{\infty,\mathcal{Q}} $ be the $\mathbb{Z}[t_R \mid R \in \mathfrak{U}]$-subalgebra of $\widetilde{\mathscr{Y}}_\infty$
generated by the $Y_{i,p}^{\pm 1} $ for $(i,p)\in \hat{I}_{\mathcal{Q}}$.
This is a quantum torus of rank $ |\Delta_+ |$. $ \widetilde{\mathscr{Y}}_{\infty,\mathcal{Q}}$ is the corresponding extended torus.

\section{Finite-dimensional representations of quantum affine algebras}\label{trois}

We give a brief review on finite-dimensional representations of quantum affine algebras (the reader may refer to \cite{Cha2, CH} and references therein for more details).

Recall that $ q $ is a non-zero complex number which is not a root of unity. Let $U_q(\mathfrak{g})$ be the quantum affine algebra associated with the affine Kac-Moody algebra $\Glie$, which is a $q$-deformation of the universal enveloping algebra of $ \mathfrak{g}$. Let $ \mathscr{C} $ be the category of finite-dimensional $ U_q(\mathfrak{g})$-modules of type $ 1 $, namely the category of modules whose eigenvalues for the elements $k_i$ ($i \in I $) are of the form $ q^m$ for some $ m \in \mathbb{Z}$. Since $U_q(\mathfrak{g})$ is a Hopf algebra the category $ \mathscr{C} $ has a tensor structure. However it is not semisimple and not braided.

Chari and Pressley \cite{Cha2} proved that the simple objects $ L $ of $ \mathscr{C}$ are parametrized by $I$-tuples of polynomials in one indeterminate,
with coefficients in $ \mathbb{C}$ and constant term $ 1 $, the so-called Drinfeld polynomials $ P_L = (P_{i,L}(u)\,, i \in I )$.  
Some distinguished objects of $\mathscr{C}$ are the \textit{fundamental modules} $ V_{i}(a)$
($i \in I $, $ a \in \mathbb{C}^\ast$) whose Drinfeld polynomials are of the form
$$ P_{j,V_{i}(a)}(u) = \begin{cases} 1-au, &\text{if } i = j \\ 1, &\text{otherwise}.\end{cases}$$

Another family of distinguished (simple) objects of $ \mathscr{C} $ is given by the
\textit{Kirillov-Reshetikhin modules} $ W_{k,a}^{(i)}$
($i \in I $, $k \in \mathbb{N}^\ast$ and $ a \in \mathbb{C}^\ast$)
whose Drinfeld polynomials are of the form
$$ P_{j,W^{(i)}_{k,a}}(u) = \begin{cases} (1-au)(1-aq^2u)\cdots (1-aq^{2k-2}u), &\text{if } i = j \\ 1, &\text{otherwise}.\end{cases}$$
Clearly, $ W^{(i)}_{1,a}$ coincides with the fundamental module $ V_{i}(a) $
while $ W^{(i)}_{0,a} $ is by convention the trivial representation, for every $ i $ and every $ a $.

The affine analogue of the usual weights for $\overline{\mathfrak{g}}$-modules are called $ \ell$-weights,
and every simple $U_q(\mathfrak{g})$-module $ L $ is uniquely characterized by a highest $\ell$-weight $ \gamma$ similarly to what happens for simple $\overline{\mathfrak{g}}$-modules. 

Given an element $ V \in \text{Ob}(\mathscr{C})$, it can be decomposed as direct sum of its  $\ell$-weight spaces, namely the affine analogues of the weight spaces.
Frenkel and Reshetikhin \cite{Fre}, have attached to $V$ a certain element of $\mathbb{Z}[Y_{i,a}^{\pm 1} \mid i \in I, a \in \mathbb{C}^\ast]$ with positive coefficients, which we call its $q$-character $ \chi_q(V)$.
This is the generating series of the $\ell$-weight spaces of $ V $.
If $V$ is a simple module, it is determined up to isomorphism by its $q$-character.

To any $ \ell$-weight $\gamma $ of a $U_q(\mathfrak{g})$-module $ V $ we can attach a certain Laurent monomial $ m_\gamma \in \mathbb{Z}[Y_{i,a}^{\pm 1} \mid i \in I, a \in \mathbb{C}^\ast]$.
In particular, for any simple $U_q(\mathfrak{g})$-module $ L$ with highest $\ell$-weight $\gamma$,
the corresponding monomial $ m_\gamma $ is \textit{dominant}, i.e. it does only contain positive powers of the variables $ Y_{i,a}$. In this case we denote $ L = L(m_{\gamma})$.
Dominant monomials give an equivalent classification of the simple objects in $ \mathscr{C}$, up to isomorphism.

Let $\mathscr{K}(\mathscr{C})$ denote the Grothendieck ring of $\mathscr{C}$.
Although $\mathscr{C}$ is not braided, its Grothendieck ring  $\mathscr{K}(\mathscr{C})$ is commutative.
It is known \cite[Cor. $2$]{Fre} that the classes of the fundamental modules are algebraically independent 
and that $\mathscr{K}(\mathscr{C})$ is isomorphic to the polynomial ring in these classes.

Let us focus on certain monoidal subcategories.

Following \cite{hl}, let $\mathscr{C}_{\mathbb{Z}}$ the subcategory of the category of finite-dimensional $ U_q(\mathfrak{g})$-modules (of type 1) 
whose simple constituents have highest monomial in $\mathscr{Y}_\xi : = \mathbb{Z}[Y_{i,q^{r}}^{\pm 1} \mid (i,r) \in \hat{I}_\xi]$, 
where $ \xi$ is a height function. Then $\mathscr{C}_{\mathbb{Z}}$ is a tensor subcategory of
$\mathscr{C}$ and its Grothendieck ring is the subring of $\mathscr{K}(\mathscr{C})$ generated by the classes of the fundamental modules
of the form $V_i(q^{r})$, $(i,r) \in \hat{I}_\xi$. The $q$-character of an object in $\mathscr{C}_{\mathbb{Z}}$ is 
a Laurent polynomial in $\mathscr{Y} = \mathbb{Z}[Y_{i,q^r}^{\pm 1} \mid (i,r) \in \hat{I}_\xi]$. %
In particular, the fundamental module $V_i(q^r) $ is associated to the dominant monomial $ Y_{i,q^r}$, and 
the KR-module $W^{(i)}_{k,q^r}$ to the dominant monomial $  m^{(i)}_{k,r}:= Y_{i,q^r}Y_{i,q^{r+2}}\cdots Y_{i,q^{r + 2k-2}}$.
For  $i\in I$ and $r\in\mathbb{Z}$ we use a simplification of notation : $Y_{i,r} = Y_{i,q^r}$, and $W^{(i)}_{k,r} := W^{(i)}_{k,q^r} $.

This category has interesting 
monoidal subcategories corresponding to the sub-tori discussed in the last section.

{\bf Example 1} Let $ \mathscr{C}_1$ be the full subcategory of $\mathscr{C}_{\mathbb{Z}}$ of objects whose simple constituents are indexed by dominant commutative monomials in $\widetilde{\mathscr{Y}}_{\infty,1}$. 

{\bf Example 2} In type $A$, let $\mathscr{C}_1^{ob}$ be the full subcategory of $\mathscr{C}_{\mathbb{Z}}$ of objects whose simple constituents are indexed by dominant commutative monomials in $\widetilde{\mathscr{Y}}_{\infty,ob}$. 

{\bf Example 3} Let $\mathscr{C}_{\mathcal{Q}}$ be the full subcategory of $\mathscr{C}_{\mathbb{Z}}$ of objects whose simple 
constituents are indexed by dominant commutative monomials in $\widetilde{\mathscr{Y}}_{\infty,\mathcal{Q}}$.

These categories $\mathscr{C}_1$, $\mathscr{C}_1^{ob}$, $\mathscr{C}_{\mathcal{Q}}$ are monoidal (see \cite{hl, hlad, HL:qGro}, respectively). 
Note that the choice of an arbitrary sub-torus does not lead necessarily to a monoidal category, see comments in the proof of \cite[Lemma 5.8]{HL:qGro}.

\section{Toroidal Grothendieck rings}\label{quatre}

Let $\mathscr{C}' \subset \mathscr{C}$ be one of the subcategories considered above, that is $\mathscr{C}_{\mathbb{Z}}$, $\mathscr{C}_1$, 
$\mathscr{C}_1^{ob}$ or $\mathscr{C}_{\mathcal{Q}}$.
We introduce multi-parameter deformations of the Grothendieck ring $ \mathscr{K}(\mathscr{C}')$ of the category $\mathscr{C}'$  (see Definition \ref{quotientorus}).

One parameter quantum deformations of the quantum Grothendieck ring appeared in the work of Nakajima \cite{Nak:quiver} and Varagnolo-Vasserot  \cite{VV:qGro} in type ADE, with a geometric construction based on categories of perverse sheaves on quiver varieties. An alternative algebraic construction was given by the second author \cite{H1} for all types. We will follow this approach, but in addition to the technical points addressed in \cite{H1}, we have to overcome new difficulties related to the flatness of the deformation. We introduce a specific quotient of a multi-parameter quantum torus (Definition \ref{quotientorus}) in which we construct the
toroidal Grothendieck rings (Definition \ref{torgr}). The flatness is proved in Theorem \ref{mainflat}. 

We also define classes of fundamental representations 
which provide a generating family of the toroidal Grothendieck ring (Proposition \ref{fundgen}).

\subsection{Quantum Grothendieck rings}\label{sec:qgr}

The constructions of quantum Grothendieck rings $\mathscr{K}_t(\mathscr{C})$ with one parameter are based on $t$-deformations of Frenkel and Reshetikhin $q$-character, the $ (q,t)$-characters. 
They belong to a $t$-deformed version $\mathscr{Y}_t $ of the quantum torus $\mathscr{Y}$ (cfr. Remark \ref{20190726:rem1}). 
Let us recall the following main properties (see \cite{HL:qGro} for a complete review): 

(1) $\mathscr{K}_t(\mathscr{C})$ is defined as the intersection of subrings $\mathscr{K}_{i,t}$, $i\in I$, of $\mathscr{Y}_t$. 
The definition of these subrings mimics what should be the definition in the $sl_2$-case for each node $i$. This is a 
reminiscence of the Weyl group invariance
of usual characters.

(2) $\mathscr{K}_t(\mathscr{C})$ has a $\mathbb{Z}[t^{\pm \frac{1}{2}}]$-basis of elements denoted by $F_t(m)$. Each $F_t(m)$ has a unique dominant monomial $m$ and its multiplicity is $1$. In particular, each non-zero element
in $\mathscr{K}_t(\mathscr{C})$ has at least a dominant monomial and is characterized by the multiplicity of its dominant monomials.
The $F_t(m)$ are obtained by an explicit algorithm and are deformations of analogs $F(m)$ which form a basis of $\mathscr{K}(\mathscr{C})$. 
We get a flat deformation of the classical Grothendieck ring $\mathscr{K}(\mathscr{C})$.
By a flat deformation of a commutative algebra $\overline{A}$ we mean an algebra $ A$ that is a free module over a Laurent polynomial ring $ \mathbb{Z}[t_\lambda^{\pm \frac12} \mid \lambda \in \Lambda] $ such that  $A / \sum_{\lambda \in \Lambda} (t_\lambda^{\frac12} -1)A $ is isomorphic to $\overline{A}$.
The quantum Grothendieck ring $\mathscr{K}_t(\mathscr{C})$ satisfies this property since the ordered products of classes of fundamental modules provide a basis both of the classical and of the quantum Grothendieck ring.

(3) The $(q,t)$-character of the fundamental module $V_i(a)$ is $F_t(Y_{i,a})$. The $(q,t)$-characters of fundamental modules 
generate $\mathscr{K}_t(\mathscr{C})$ as a $\mathbb{Z}[t^{\pm \frac{1}{2}}]$-algebra.

(4) Each category $\mathscr{C}'$ has a quantum Grothendieck ring $\mathscr{K}_t(\mathscr{C}')$ generated by the $(q,t)$-characters of 
the fundamental modules in this category.

(5) For the subcategories $\mathscr{C}_1$, $\mathscr{C}_1^{ob}$, $\mathscr{C}_Q$, the $ (q,t)$-character of fundamental modules may contain monomials which do not belong to the underlying quantum subtorus. This brings to the introduction of \textit{truncated} $(q,t)$-characters, obtained by discarding all such monomials; we denote it by $ \widetilde{\chi_{q,t}(L)}$ (the specialization at $t = 1$ is the truncated $q$-character $ \widetilde{\chi_q(L)}$). They generate a different subring of the quantum torus $\mathscr{Y}_t$ isomorphic to $\mathscr{K}_t(\mathscr{C}')$ (see for instance \cite[Prop 6.1]{hl}, \cite[Prop 3.10]{hljems}). The properties (2) above are also satisfied as for a simple object $ L $ in $\mathscr{C}'$, all dominant monomials occurring in $ \chi_{q,t}(L)$ also occur in $ \widetilde{\chi_{q,t}(L)}$.

\subsection{Naive construction in the toroidal case} We first highlight the issues concerning the flatness of deformations. 
For simplicity, let us consider the case when all the fundamental representations in $\mathscr{C}' $ are \textit{thin}, namely their $q$-characters are multiplicity free (we discuss the general case later). Note that all fundamental representations are thin for types $A$, $B$, $C$ and $G_2$ 
(see \cite{H2}).

In fact, when fundamental representations are thin, we know (see for instance \cite{HL:qGro})
that their $(q,t)$-characters coincide with the usual $q$-characters. We follow this approach  to define what we call the $(q,\infty)$-character of any fundamental module $ V_{i}(a) \in \text{Ob}(\mathscr{C}')$ :
$$
[V_{i}(a)]_{q,\infty} := \chi_{q,t}(V_{i}(a)) = \chi_q(V_{i}(a)) = F(Y_{i,a}),
$$
where we identify commutative monomials, inside the quantum torus $\widetilde{\mathscr{Y}}_{\infty}$.

As a first naive definition, we can define the (generalized) toroidal Grothendieck ring of the category $\mathscr{C}'$ as the subring of the ring of Laurent polynomials $\widetilde{\mathscr{Y}}_{\infty}$ generated by the $(q,\infty)$-characters of the fundamental modules. We denote this 
by $ \widetilde{\mathscr{K}}_{\infty}(\mathscr{C}') $.

The issue with this definition of toroidal Grothendieck ring is that it might be too big, and therefore fail to be a \textit{flat deformation} of the Grothendieck ring $ \mathscr{K}(\mathscr{C}')$ (see (2) in Section \ref{sec:qgr}).
Namely, since we know that the ordered products of classes of fundamental modules provide a basis of the Grothendieck ring $\mathscr{K}(\mathscr{C}')$, we would like that (ordered) product of the same classes, when considered in the quantum torus $\widetilde{\mathscr{Y}}_{\infty}$, to provide a basis of $\widetilde{\mathscr{K}}_{\infty}(\mathscr{C}')$.  We illustrate with an example what happens when we try to do so.

\subsection{An example}\label{ex:multidefquanumgroth}
Let $ \overline{\mathfrak{g}} = sl_3$ and $ \mathscr{C'} = \mathscr{C}_{\mathcal{Q}}$.
The category $ \mathscr{C}_{\mathcal{Q}}$ has a basis given by the fundamental modules
$ V_1(1)$, $V_1(q^2) $ and $ V_2(q)$
(with height function $ \xi_1 = 0$, $ \xi_2 = 1 $).  
By definition, the $(q,\infty)$-characters of these fundamental representations are
\begin{equation}
\begin{array}{l}
\displaystyle{
\vphantom{\Big(}
[V_1(1)]_{q,\infty} = \chi_{q,t}(V_{1}(1)) = Y_{1,0} + Y_{1,2}^{-1}Y_{2,1} + Y_{2,3}^{-1}
\,}\\
\displaystyle{
\vphantom{\Big(}
[V_1(q^2)]_{q,\infty} = \chi_{q,t}(V_{1}(q^2)) = Y_{1,2} + Y_{1,4}^{-1}Y_{2,3} + Y_{2,5}^{-1}
\,}\\
\displaystyle{
\vphantom{\Big(}
[V_2(q)]_{q,\infty} = \chi_{q,t}(V_{2}(q)) = Y_{2,1} + Y_{1,2}Y^{-1}_{2,3} + Y_{1,4}^{-1}
\,.}
\end{array}
\end{equation}

The quantum Cartan matrix for $sl_3 $ is
$$
C(z)
= 
\begin{pmatrix}
z + z^{-1} & -1 \\ -1 & z + z^{-1} 
\end{pmatrix},
$$
and we can compute its inverse $ \widetilde{C}(z) = ( \widetilde{C}_{ij}(z) )_{ij\in I}$,
whose entries have the form \eqref{20190726:eq1}.
By \cite[Cor. $2.3$]{HL:qGro} we have
$ \widetilde{C}_{ij}(m) = \widetilde{C}_{ij}(m+6)$ for $ i,j \in I $ and $ m \geq 1$,
and in particular
\begin{equation}
\begin{array}{l}
\displaystyle{
\vphantom{\Big(}
\widetilde{C}_{11}(z) = \widetilde{C}_{22}(z)= z - z^5 + z^7 - z^{11} + \ldots
\,}\\
\displaystyle{
\vphantom{\Big(}
\widetilde{C}_{12}(z) = \widetilde{C}_{21}(z)= z^2 - z^4 + z^8 - z^{10} + \ldots
\,.}
\end{array}
\end{equation}

According to \eqref{20190409:eq1}, we can compute
\begin{equation}\label{20181130:eq1b}
\begin{array}{l}
\displaystyle{
\vphantom{\Big(}
[V_1(1)]_{q,\infty}
\ast
[V_1(q^2)]_{q,\infty}
=
\prod_{a \in \ZZ} t_a^{\frac{\mathcal{N}_a(1,0,1,2)}{2}}
\Big(
Y_{1,0}Y_{1,2}
+
Y^{-1}_{2,3}Y^{-1}_{2,5}
+
Y^{-1}_{1,2}Y^{-1}_{1,4}Y_{2,1}Y_{2,3}
\,}\\
\displaystyle{
\vphantom{\Big(}
+
t_{-4}^{-\frac12}t_{-2}^{\frac12}t_2^{\frac12}t_4^{-\frac12}
\big(
Y_{1,0}Y^{-1}_{1,4}Y_{2,3}
+
Y_{1,0}Y^{-1}_{2,5}
+
Y^{-1}_{1,2}Y_{2,1}Y^{-1}_{2,5}
\big)
+
t_{-2}^{-\frac12}t_0t_2^{-\frac12}
[V_2(q)]_{q,\infty}
\Big)
\,,}
\end{array}
\end{equation}
\begin{equation}\label{20181130:eq7b}
\begin{array}{l}
\displaystyle{
\vphantom{\Big(}
[V_1(1)]_{q,\infty}
\ast
[V_2(q)]_{q,\infty}
=
\prod_{a \in \ZZ} t_a^{\frac{\mathcal{N}_a(1,0,2,1)}{2}}
\Big(
Y_{1,0}Y_{2,1}
+
Y_{1,0}Y_{1,2}Y^{-1}_{2,3}
+
Y^{-1}_{1,2}Y^2_{2,1}
+
Y_{1,2}Y^{-2}_{2,3}
\,}\\
\displaystyle{
\vphantom{\Big(}
+
Y^{-1}_{1,2}Y^{-1}_{1,4}Y_{2,1}
+
Y^{-1}_{1,4}Y^{-1}_{2,3}
+
t_{-4}^{-\frac12}t_{-2}^{\frac12}t_2^{\frac12}t_4^{-\frac12}
Y_{1,1}Y^{-1}_{1,q^4}
+
\big(
t_{-2}^{-\frac12}t_0t_2^{-\frac12}
+
t_{-2}^{\frac12}t_0^{-1}t_2^{\frac12}
\big)
Y_{2,1}Y^{-1}_{2,3}
\Big)
\,,}
\end{array}
\end{equation}
\begin{equation}\label{20181130:eq14b}
\begin{array}{l}
\displaystyle{
\vphantom{\Big(}
[V_1(q^2)]_{q,\infty}
\ast
[V_2(q)]_{q,\infty}
=
\prod_{a \in \ZZ} t_a^{\frac{-\mathcal{N}_a(1,0,2,1)}{2}}
\Big(
Y_{1,2}Y_{2,1}
+
Y^2_{1,2}Y^{-1}_{2,3}
+
Y^{-1}_{1,4}Y_{2,1}Y_{2,3}
+
Y^{-2}_{1,4}Y_{2,3}
\,}\\
\displaystyle{
\vphantom{\Big(}
+
Y_{1,2}Y^{-1}_{2,3}Y^{-1}_{2,5}
+
Y^{-1}_{1,4}Y^{-1}_{2,5}
+
t_{-4}^{\frac12}t_{-2}^{-\frac12}t_2^{-\frac12}t_4^{\frac12}
Y_{2,1}Y^{-1}_{2,5}
+
\big(
t_{-2}^{-\frac12}t_0t_2^{-\frac12}
+
t_{-2}^{\frac12}t_0^{-1}t_2^{\frac12}
\big)
Y_{1,2}Y^{-1}_{1,4}
\Big)
\,,}
\end{array}
\end{equation}
$$
\begin{array}{l}
\displaystyle{
\vphantom{\Big(}
\prod_{a \in \ZZ} t_a^{\mathcal{N}_a(1,0,1,2)}
=
t_{-2}t_0^{-1}t_2^{-2}t_4^3\prod_{k\geq 1} t_{6k+2}^{-3} t_{6k+4}^3
\,,}\\
\displaystyle{
\vphantom{\Big(}
\prod_{a \in \ZZ} t_a^{\mathcal{N}_a(1,0,2,1)}
=
t_0t_2^{-3}t_4^3\prod_{k\geq 1} t_{6k+2}^{-3} t_{6k+4}^3
\,.}
\end{array}
$$

Combining the corresponding products in the opposite order,
we obtain
\begin{equation}\label{20181130:eq3b}
\begin{array}{l}
\displaystyle{
\vphantom{\Big(}
[V_1(1)]_{q,\infty}
\ast
[V_1(q^2)]_{q,\infty}
-
\prod_{a \in \ZZ} t_a^{\mathcal{N}_a(1,0,1,2)}
[V_1(q^2)]_{q,\infty}
\ast 
[V_1(1)]_{q,\infty}
\,}\\
\displaystyle{
\vphantom{\Big(}
=
\big(
1
-
t_{-4}t_{-2}^{-1}t_2^{-1}t_4
\big)
\big(
Y_{1,0}
\ast 
Y^{-1}_{1,4}Y_{2,3}
+
Y_{1,0}
\ast 
Y^{-1}_{2,5}
+
Y^{-1}_{1,2}Y_{2,1}
\ast 
Y^{-1}_{2,5}
\big)
\,}\\
\displaystyle{
\vphantom{\Big(}
+
\big(
1
-
t_{-2}t_0^{-2}t_2
\big)
\prod_{a \in \ZZ} t_a^{\frac{\mathcal{N}_a(1,0,2,1)}{2}}
[V_2(q)]_{q,\infty}
\,,}
\end{array}
\end{equation}
\begin{equation}\label{20181130:eq9b}
\begin{array}{l}
\displaystyle{
\vphantom{\Big(}
[V_1(1)]_{q,\infty}
\ast
[V_2(q)]_{q,\infty}
-
\prod_{a \in \ZZ} t_a^{\mathcal{N}_a(1,0,2,1)}
[V_2(q)]_{q,\infty}
\ast
[V_1(1)]_{q,\infty}
\,}\\
\displaystyle{
\vphantom{\Big(}
=
\big(
1
-
t_{-4}t_{-2}^{-1}t_2^{-1}t_4
\big)
Y_{1,0}
\ast
Y^{-1}_{1,4}
\,,}
\end{array}
\end{equation}
\begin{equation}\label{20181130:eq16b}
\begin{array}{l}
\displaystyle{
\vphantom{\Big(}
[V_1(q^2)]_{q,\infty}
\ast
[V_2(q)]_{q,\infty}
-
\prod_{a \in \ZZ} t_a^{-\mathcal{N}_a(1,0,2,1)}
[V_2(q)]_{q,\infty}
\ast
[V_1(q^2)]_{q,\infty}
\,}\\
\displaystyle{
\vphantom{\Big(}
=
\big(
1
-
t_{-4}^{-1}t_{-2}t_2t_4^{-1}
\big)
Y^{-1}_{2,5}
\ast
Y_{2,1}
\,.}
\end{array}
\end{equation}
The RHS of equations \eqref{20181130:eq3b}, \eqref{20181130:eq9b} and \eqref{20181130:eq16b},
should provide elements of $ \widetilde{\mathscr{K}}_{\infty}(\mathscr{C}_{\mathcal{Q}})$.
However, no dominant monomial occur in these elements (except for $ [V_2(q)]_{q,\infty}$ in the RHS of \eqref{20181130:eq3b}), 
as we should expect if we had a basis of ordered product of fundamental classes (cfr. Section \ref{sec:qgr}, property ($2$)).
Therefore, we need to impose some additional relations between the parameters $ t_a^{\pm \frac12}$
in order for $\widetilde{\mathscr{K}}_{\infty}(\mathscr{C}_{\mathcal{Q}})$ to be a flat deformation of $\mathscr{K}(\mathscr{C}_{\mathcal{Q}})$.

If we quotient the quantum torus $ \widetilde{\mathscr{Y}}_\infty$ by the relation
$$ 1 = t_{-4}^{\frac12}t_{-2}^{-\frac12}t_2^{-\frac12}t_4^{\frac12}$$
and we take $\mathscr{K}_{\infty}(\mathscr{C}_{\mathcal{Q}}) $ to be the subring of this quotient quantum torus 
$ \mathscr{Y}_{\infty}$ generated by the (images of the) classes of the fundamental modules
$ V_1(1)$, $V_1(q^2)$ and $ V_2(q)$, it descends from the relations above that the result is a genuine flat deformation of $ \mathscr{K}(\mathscr{C}_\mathcal{Q})$.
As a consequence, we can also uniquely define the following classes of simple modules:
\begin{equation}\label{20181211:eq1}
\begin{array}{l}
\displaystyle{
\vphantom{\Big(}
[L(Y_{1,0}Y_{1,2})]_{q,\infty}
:=
\chi_{q,t}(L(Y_{1,0}Y_{1,2}))
\,,}\\[10pt]
\displaystyle{
\vphantom{\Big(}
[L(Y_{1,0}Y_{2,1})]_{q,\infty}
:=
Y_{1,0}Y_{2,1}
+
Y_{1,0}Y_{1,2}Y^{-1}_{2,3}
+
Y^{-1}_{1,2}Y^2_{2,1}
+
Y^{-1}_{1,2}Y^{-1}_{1,4}Y_{2,1}
\,}\\
\displaystyle{
\vphantom{\Big(}
\phantom{A}
+
Y_{1,2}Y^{-2}_{2,3}
+
Y^{-1}_{1,4}Y^{-1}_{2,3}
+
Y_{1,0}Y^{-1}_{1,4}
+
\big(
t_{-2}^{-\frac12}t_0t_2^{-\frac12}
+
t_{-2}^{\frac12}t_0^{-1}t_2^{\frac12}
\big)
Y_{2,1}Y^{-1}_{2,3}
\,,}\\[10pt]
\displaystyle{
\vphantom{\Big(}
[L(Y_{1,2}Y_{2,1})]_{q,\infty}
:=
Y_{1,2}Y_{2,1}
+
Y^2_{1,2}Y^{-1}_{2,3}
+
Y^{-1}_{1,4}Y_{2,1}Y_{2,3}
+
Y^{-2}_{1,4}Y_{2,3}
\,}\\
\displaystyle{
\vphantom{\Big(}
\phantom{A}
+
Y_{1,2}Y^{-1}_{2,3}Y^{-1}_{2,5}
+
Y^{-1}_{1,4}Y^{-1}_{2,5}
+
Y_{2,1}Y^{-1}_{2,5}
+
\big(
t_{-2}^{-\frac12}t_0t_2^{-\frac12}
+
t_{-2}^{\frac12}t_0^{-1}t_2^{\frac12}
\big)
Y_{1,2}Y^{-1}_{1,4}
\,.}
\end{array}
\end{equation}

\subsection{Idea of the general construction}

In general, our strategy is to define a new quantum torus $ \mathscr{Y}_{\infty} $
as the quotient of $ \widetilde{\mathscr{Y}}_{\infty}$ by all the relations 
which appear as coefficients leading to elements in the toroidal Grothendieck ring 
without dominant monomials.

Let us now give the precise construction.

\subsection{Toroidal Grothendieck ring for the category $\mathscr{C}_{\mathbb Z}$}
We follow the idea of one parameter quantum Grothendieck rings explained at the beginning of this section. We introduce for each $i\in I$ a subring mimicking the construction in the $sl_2$-case and then the toroidal Grothendieck ring will be defined as the intersection of these subrings.

For $ i \in I $, let  $\widetilde{\mathscr{K}}_{i,\infty}  $ 
be the $\mathbb{Z}[t_R \mid R \in \mathfrak{U}]$-subalgebra of $\widetilde{\mathscr{Y}}_{\infty} $
generated by the elements
\begin{equation}\label{20190919:eq1}
\begin{cases}
Y_{i,r} + Y_{i,r}A^{-1}_{i,r+1},&\\
Y_{j,r},\quad &j \neq i,
\end{cases}
\end{equation}
where for $ i \in I$, $ r \in \mathbb{Z}$.
Note that $Y_{i,r}A^{-1}_{i,r+1} = Y_{i,r+2}^{-1}\prod_{j \sim i} Y_{j,r+1}$.

Let $i\in I$. A monomial $m = \prod_{(j,r)  \in \hat{I}} Y_{j,r}^{u_{j,r}(m)}$ is said to be $i$-dominant if the powers $u_{j,r}(m)$ of the $Y_{j,r}$ are all positive for $j = i$. Then one can define 
$$
E_{i,\infty}(m)
:=
\overset{\rightarrow}{*}_{r \in \mathbb{Z}}
\Big( \big(Y_{i,r} + Y_{i,r}A^{-1}_{i,r+1}\big)^{u_{i,r}(m)} *_{j\neq i} Y_{j,r}^{u_{j,r}(m)}\Big) \in \widetilde{\mathscr{K}}_{i,\infty},
$$
where the arrow above the product sign means that the product is ordered increasingly in the index
$ r $ (i.e. $\overset{\rightarrow}{\prod}_{r \in \mathbb{Z}} U_r = \ldots U_{-1} U_0 U_1 U_2 \ldots$).

Analogously to the one-parameter case in \cite{H1}, we would like to prove that for every $ i \in I $ the elements $E_{i,\infty}(m) $ provide a $\mathbb{Z}[t_R  \mid R \in \mathfrak{U}]$-basis of $\widetilde{\mathscr{K}}_{i,\infty}$. This family is linearly free. So, it would be sufficient to show that every non-ordered product in 
$\widetilde{\mathscr{K}}_{i,\infty}$ can be written as a linear combination of
the elements $E_{i,\infty}(m) $.

By \eqref{20190430:eq2} the variables $ Y_{j,r} $ for a fixed $ r \in \mathbb{Z}$ mutually commute,
and moreover it is possible to show (analogously to the one-parameter case in \cite{H1})
that for any $ j \neq i $ and fixed $ r \in \mathbb{Z}$ the two generators
$ Y_{i,r} + Y_{i,r}A^{-1}_{i,r+1} $ and $ Y_{j,r} $ commute.
Therefore it suffices to consider the product of two generators as $ Y_{i,r} + Y_{i,r}A^{-1}_{i,r+1}$ and $Y_{i,r'} + Y_{i,r'}A^{-1}_{i,r'+1}$ ($r \neq r'$) in both orders. The following Proposition thus shows how the $E_{i,\infty}(m) $ fail to generate the whole $\widetilde{\mathscr{K}}_{i,\infty}$.

\begin{prop}\label{prop:qgroth2}
Let $ k \geq 1 $, and consider the family $ \alpha_s(k) \in \frac12 \mathbb{Z}$ defined by
$$ Y_{i,r} \ast Y_{i,r+2k} = \Big(\prod_{s \in \mathbb{Z}} t_s^{\alpha_s(k)}\Big) Y_{i,r+2k} \ast Y_{i,r}.$$
Then in the quantum torus $ \widetilde{\mathscr{Y}}_{\infty}$ the following holds
\begin{equation}\label{20190430:eq1}
\begin{array}{l}
\displaystyle{
\vphantom{\Big(}
\Big(Y_{i,r} + Y_{i,r}A^{-1}_{i,r+1}\Big)
\ast
\Big(Y_{i,r+2k} + Y_{i,r+2k}A^{-1}_{i,r+2k+1}\Big)
}\\
\displaystyle{
\vphantom{\Big(}
-
\Big(\prod_{s \in \mathbb{Z}} t_s^{\alpha_s(k)}\Big)
\Big(Y_{i,r+2k} + Y_{i,r+2k}A^{-1}_{i,r+2k+1}\Big)
\ast
\Big(Y_{i,r} + Y_{i,r}A^{-1}_{i,r+1}\Big)
}\\
\displaystyle{
\vphantom{\Big(}
=
\Big(1 - t_{-2k-2}t_{-2k}^{-1}t_{2k}^{-1}t_{2k+2}\Big)
\Big(\prod_{s \in \mathbb{Z}} t_s^{\beta_s(k)}\Big)
Y_{i,r}Y_{i,r+2k}A_{i,r+2k+1}^{-1}
}\\
\displaystyle{
\vphantom{\Big(}
+
\Big(1 - t_{-2k}t_{-2k+2}^{-1}t_{2k-2}^{-1}t_{2k}\Big)
\Big(\prod_{s \in \mathbb{Z}} t_s^{\gamma_s(k)}\Big)
Y_{i,r}A_{i,r+1}^{-1}Y_{i,r+2k}
\,,
}
\end{array}
\end{equation}
for some $ \beta_s(k), \gamma_s(k) \in \frac12 \mathbb{Z}$.
\end{prop}

\begin{proof}
For $ s \in \mathbb{Z} $, let $ \beta_s(k)$, $ \gamma_s(k)$, $ \delta_s(k) \in \frac12 \mathbb{Z}$
such that
$$
\begin{array}{l}
\displaystyle{
\vphantom{\Big(}
Y_{i,r} \ast Y_{i,r+2k}A^{-1}_{i,r+2k+1}
=
\Big(\prod_{s \in \mathbb{Z}} t_s^{2\beta_s(k)}\Big)
Y_{i,r+2k}A^{-1}_{i,r+2k+1} \ast Y_{i,r}
\,,}\\
\displaystyle{
\vphantom{\Big(}
Y_{i,r} A^{-1}_{i,r+1} \ast Y_{i,r+2k}
=
\Big(\prod_{s \in \mathbb{Z}} t_s^{2\gamma_s(k)}\Big)
Y_{i,r+2k} \ast Y_{i,r}A^{-1}_{i,r+1}
\,,}\\
\displaystyle{
\vphantom{\Big(}
Y_{i,r}A^{-1}_{i,r+1}\ast Y_{i,r+2k}A^{-1}_{i,r+2k+1}
=
\Big(\prod_{s \in \mathbb{Z}} t_s^{2\delta_s(k)}\Big)
Y_{i,r+2k}A^{-1}_{i,r+2k+1} \ast Y_{i,r}A^{-1}_{i,r+1}
\,.}
\end{array}
$$
By substituting these commutation relations, the LHS of \eqref{20190430:eq1}
equals
\begin{equation}\label{20190624:eq3}
\begin{array}{c}
\displaystyle{
\vphantom{\Big(}
\Big(
1
-
\prod_{s \in \mathbb{Z}} t_s^{\alpha_s(k)-2\beta_s(k)}
\Big)
Y_{i,r} \ast Y_{i,r+2k}A^{-1}_{i,r+2k+1}
+
\Big(
1
-
\prod_{s \in \mathbb{Z}} t_s^{\alpha_s(k)-2\gamma_s(k)}
\Big)
Y_{i,r}A^{-1}_{i,r+1} \ast Y_{i,r+2k}
}\\
\displaystyle{
\vphantom{\Big(}
+
\Big(
1
-
\prod_{s \in \mathbb{Z}} t_s^{\alpha_s(k)-2\delta_s(k)}
\Big)
Y_{i,r}A^{-1}_{i,r+1} \ast Y_{i,r+2k}A^{-1}_{i,r+2k+1}
\,.
}
\end{array}
\end{equation}

By \eqref{20190429:eq1}, and by the properties of the inverse of the quantum Cartan matrix
we can compute explicitly the exponents $ \alpha_s(k), \beta_s(k), \gamma_s(k), \delta_s(k) $ ($ s \in \mathbb{Z}$) and obtain:
\[
\begin{array}{l}
\displaystyle{
\alpha_s(k) \,= \mathcal{N}_s(i,0;i,2k)
= \widetilde{C}_{ii}(-2k-1+s) - \widetilde{C}_{ii}(2k-1+s) 
- \widetilde{C}_{ii}(-2k+1+s) + \widetilde{C}_{ii}(2k+1+s)
\,,}\\
\displaystyle{
\vphantom{\Big(}
2\beta_s(k) = - \mathcal{N}_s(i,0;i,2k+2) + \sum_{j \sim i} \mathcal{N}_s(i,0;j,2k+1)
=  - \delta_{s,-2k-2} + \delta_{s,-2k} + \delta_{s,2k} - \delta_{s,2k+2}
\,}\\
\displaystyle{
\quad + \widetilde{C}_{ii}(-2k-1+s) - \widetilde{C}_{ii}(-2k+1+s) - \widetilde{C}_{ii}(2k-1+s) + \widetilde{C}_{ii}(2k+1+s)
\,,}\\
\displaystyle{
\vphantom{\Big(}
2\gamma_s(k) = - \mathcal{N}_s(i,0;i,2k-2) + \sum_{j\sim i}\mathcal{N}_s(i,0;j,2k-1)
=
-\delta_{s,-2k} + \delta_{s,-2k+2} + \delta_{s,2k-2} -\delta_{s,2k}
\,}\\
\displaystyle{
\quad
+\widetilde{C}_{ii}(-2k-1+s) - \widetilde{C}_{ii}(-2k+1+s) - \widetilde{C}_{ii}(2k-1+s)+ \widetilde{C}_{ii}(2k+1+s)
\,,}\\
\displaystyle{
\vphantom{\Big(}
2\delta_s(k) = \mathcal{N}_s(i,0;i,2k)\! - \!\sum_{j\sim i} \mathcal{N}_s(i,0;j,2k\!-\!1)  
\!-\! \sum_{j\sim i}\mathcal{N}_s(i,0;j,2k\!+\!1) \!+\!\! \sum_{j,h \sim i} \mathcal{N}_s(j,0;h,2k)
\,}\\
\displaystyle{
\quad = \mathcal{N}_s(i,0;i,2k)
\,.}
\end{array}
\]
As a consequence, for any $ k\geq 1 $, we have $ \alpha_s(k) - 2\delta_s(k) = 0 $, whereas
$$
\alpha_s(k) - 2\beta_s(k) = \delta_{s,-2k-2} - \delta_{s,-2k} - \delta_{s,2k} + \delta_{s,2k+2} 
\,,
$$
$$
\alpha_s(k) - 2\gamma_s(k) = \delta_{s,-2k} - \delta_{s,-2k+2} - \delta_{s,2k-2} + \delta_{s,2k} 
\,.
$$
To conclude, substituting the values of $ \alpha_s(k) - 2\beta_s(k)$, $\alpha_s(k) - 2\gamma_s(k)$
and $\alpha_s(k) - 2\delta_s(k)$ in \eqref{20190624:eq3},
we obtain exactly equation \eqref{20190430:eq1}.
\end{proof}

\begin{rem}
For $s = 0 $ the exponents $ \alpha_0(k)$, $\beta_0(k)$, $ \gamma_0(k)$, $ \delta_0(k)$ in Proposition \ref{prop:qgroth2} coincide with the corresponding powers of $ t_0 = t $ in the one-parameter case (cfr. \cite[Cor. 4.11]{H1}).
\end{rem}

By definition $Y_{i,r}A_{i,r+1}^{-1}Y_{i,r+2} \in \mathbb{Z}[Y_{j,s}]_{j \neq i, s \in \mathbb{Z}}$,
thus the monomial $Y_{i,r}A_{i,r+1}^{-1}Y_{i,r+2k}\in \widetilde{\mathscr{K}}_{i,\infty}$ if and only if $ k = 1$. 
The monomial $Y_{i,r}Y_{i,r+2k}A_{i,r+2k+1}^{-1}\notin \widetilde{\mathscr{K}}_{i,\infty}$, for any $k \geq 1 $.

Therefore, equation \eqref{20190430:eq1} shows that the elements $E_{i,\infty}(m)$ do not form a basis of $  \widetilde{\mathscr{K}}_{i,\infty}$.
However, as explained above, we will instead consider a particular quotient of the quantum torus
$  \widetilde{\mathscr{Y}}_{\infty} $, and define particular subalgebras 
$\mathscr{K}_{i,\infty} $  therein such that (the images of) the elements $E_{i,\infty}(m)$
now have the desired properties. The relations to define the new quantum tori naturally appeared in
the Proposition \ref{prop:qgroth2} above.

\begin{defi}\label{quotientorus}  Let $ \mathscr{Y}_{\infty} $ be the quotient of the quantum torus $  \widetilde{\mathscr{Y}}_{\infty} $
by the relations
$$
\mathscr{R}_k: \quad 1 = t_{-2k-2}^{-\frac12} t_{-2k}^{\frac12} t_{2k}^{\frac12} t_{2k+2}^{-\frac12}\text{, $k\geq 1$.}
$$
\end{defi}
Let moreover $\mathscr{K}_{i,\infty} $ be the image of $\widetilde{\mathscr{K}}_{i,\infty}$ in $ \mathscr{Y}_{\infty} $.

In the following, by \textit{basis} of a submodules of $ \mathscr{Y}_{\infty} $ we mean a basis over the ring 
$$ \mathbb{Z}[t_R \mid  R \in \mathfrak{U}]/ \big( \mathscr{R}_k \big)_{k\geq 1} .$$

\begin{prop} The images of the $E_{i,\infty}(m)$
in $ \mathscr{K}_{i,\infty}$ form a basis.
\end{prop}

We will still denote $E_{i,\infty}(m)$ its image in $\mathscr{K}_{i,\infty}$ and the product by $ \ast$.

\begin{proof}
From Proposition \ref{prop:qgroth2}, inside the quantum torus $\mathscr{Y}_{\infty} $ we have (for $k > 1$) :
\begin{equation}\label{20190624:eq4}
\begin{array}{l}
\displaystyle{
\vphantom{\Big(}
\Big(Y_{i,r} + Y_{i,r}A^{-1}_{i,r+1}\Big)
\ast
\Big(Y_{i,r+2} + Y_{i,r+2}A^{-1}_{i,r+3}\Big)
}\\
\displaystyle{
\vphantom{\Big(}
-
\Big(\prod_{s \in \mathbb{Z}} t_s^{\alpha_s(1)}\Big)
\Big(Y_{i,r+2} + Y_{i,r+2}A^{-1}_{i,r+3}\Big)
\ast
\Big(Y_{i,r} + Y_{i,r}A^{-1}_{i,r+1}\Big)
}\\
\displaystyle{
\vphantom{\Big(}
\in \mathbb{Z}[t_R \mid  R \in \mathfrak{U}] [Y_{i,r}A_{i,r+1}^{-1}Y_{i,r+2}]
\,,
}
\end{array}
\end{equation}
\begin{equation}
\begin{array}{l}
\displaystyle{
\vphantom{\Big(}
\Big(Y_{i,r} + Y_{i,r}A^{-1}_{i,r+1}\Big)
\ast
\Big(Y_{i,r+2k} + Y_{i,r+2k}A^{-1}_{i,r+2k+1}\Big)
}\\
\displaystyle{
\vphantom{\Big(}
-
\Big(\prod_{s \in \mathbb{Z}} t_s^{\alpha_s(k)}\Big)
\Big(Y_{i,r+2k} + Y_{i,r+2k}A^{-1}_{i,r+2k+1}\Big)
\ast
\Big(Y_{i,r} + Y_{i,r}A^{-1}_{i,r+1}\Big)
= 0
\,.
}
\end{array}
\end{equation}
Hence the same arguments as in \cite{H1} for the quantum case also work in the toroidal case.
\end{proof}

\begin{rem} The $\mathscr{R}_k $ ($ k \geq 1$) is a minimal set of relations so that the last result holds.
\end{rem}

\begin{rem}\label{20191029:rem1}
Given variables $ Y_{i,2r + \xi_i} $, $ Y_{j, 2s + \xi_j}$, $\big((i,r), (j,s) \in \hat{I}\big)$, we have
$$ Y_{i,2r + \xi_i} \ast Y_{j, 2s + \xi_j} = \prod_{a \in \mathbb{Z}} t_a^{\mathcal{N}_a(i,2r + \xi_i;j,2s + \xi_j)} Y_{j, 2s + \xi_j} \ast Y_{i,2r + \xi_i}\,.$$
By \cite[Prop. 2.1]{hl}, for any $ m \geq 1$, the entry $ \widetilde{C}_{i,j}(m) $ of $ \widetilde{C}(z) $ vanishes unless $ m + \xi_i - \xi_j $ is odd.
We thus conclude that $\mathcal{N}_a(i,2r + \xi_i;j,2s + \xi_j) = 0 $ for all $ a \in 2\mathbb{Z} +1$.
As a consequence, only parameter with even indices, that is the parameters $ t_{-2k}$, may appear in the quasi-commutation relations of any two monomials of $ \widetilde{\mathscr{Y}_{\infty}}$ or of $\mathscr{Y}_{\infty}$.
\end{rem}

\begin{defi}\label{torgr}
We define the toroidal Grothendieck ring to be the intersection
$$\mathscr{K}_{\infty} \big(=   \mathscr{K}_{\infty}(\mathscr{C}_{\mathbb{Z}})  \big) := \bigcap_{i \in I} \mathscr{K}_{i,\infty}.$$
\end{defi}

One obtains that every non-zero element of $ \mathscr{K}_{\infty}$
has at least one dominant monomial.
The argument is the same as in the one-parameter case (cfr. \cite[Lemma 5.7]{H1}),
and is based on the classical result of Frenkel and Reshetikhin \cite{Fre}.
Let us sum up the argument. One first observes that to each monomial in the $ Y_{i,r}^{\pm 1} $
can be associated a weight, as follows:
$$
m = \prod_{(i,r)\in \hat{I}} Y_{i,r}^{u_{i,r}(m)} \mapsto \sum_i \Big(\prod_r u_{i,r}(m) \Big) \omega_i
\,.
$$
For instance, the weight of the monomial $ Y_{i,r} $ is the fundamental weight $ \omega_i $
and the weight of the element $ A_{i,r}$ is the simple root $ \alpha_i$.
 The Nakajima ordering is a partial ordering on monomials in the variables $Y_{i,r}^{\pm 1}$ which is a refinement of the usual ordering on weights :
$$
m \leq m' \quad \text{if and only if } m'm^{-1} \text{ is a product of } A_{i,r}
\,.
$$
 
Then, one argues that an element $\chi$ in $\mathscr{K}_{\infty}$ contains a monomial $ M $ maximal with respect to the order $ \leq $. For each $i\in I$, $\chi$ is a linear combination of various $E_{i,\infty}(m)$ with the $m$
$i$-dominant. This implies that $M$ is equal to one of these $m$ for each $i$. So $M$ is dominant.

Next, we need to show that $\mathscr{K}_{\infty} $ is non-zero and has the correct size. As explained above,
it is easier if the fundamental representations are thin. The general statement is also true.

\begin{thm}\label{mainflat} For any dominant monomial $m$, there is a unique $F_\infty(m)\in \mathscr{K}_{\infty}$ so that 
$m$ is the unique dominant monomial occurring in $F_\infty(m)$ and its multiplicity is $1$. The $F_\infty(m)$ form a basis of 
$\mathscr{K}_\infty$ which is a flat deformation of $\mathscr{K}(\mathscr{C})$. \end{thm}

\begin{proof} Thanks to the discussion above, we can now follow \cite[Theorem 5.11]{H1} and the recursive algorithm therein.
The proof is the same, in particular the fact that the algorithm used to construct  $ F_{\infty}(m)$ never produces conflicting results.
For this, it is sufficient to check the cases of Lie algebras of rank $2$, which is handled as in the one-parameter case.
We also point out that, although in \cite{H1} some completion of both the quantum torus and the deformed Grothendieck ring are introduced in order to carry out the argument, this is not necessary, as it is pointed out in \cite{HO} (see also \cite[Subsection 7.3]{H2}). For the last point, the $\mathbb{Z}$-basis $F(m)$ of $\mathscr{K}(\mathscr{C})$ gets deformed into a basis of $\mathscr{K}_\infty$.
\end{proof}

\begin{defi}
For a fundamental module $ V_i(q^r) $, $ r \in \mathbb{Z}$, we define its $(q,\infty)$-character (or its class in $\mathscr{K}_{\infty}$) as
$$ [ V_i(q^r)]_{q,\infty} := F_{\infty}(Y_{i,r})\,.$$
\end{defi}

\begin{rem}\label{usesmall} The definition $ [ V_i(q^r)]_{q,\infty}$ is consistent with the definition $[ V_i(q^r)]_{q,\infty} = \chi_{q,t}( V_i(q^r))$ in the case when the fundamental module $ V_i(q^r)$ is thin. This is clear in this case as the three algorithms used to 
construct $ F(Y_{i,r})$, $ F_{t}(Y_{i,r})$ and $F_{\infty}(Y_{i,r})$ give the same result (this is analog to \cite[Corollary 5.3]{HL:qGro}).
\end{rem}

The following is obtained as in the one-parameter case.

\begin{prop}\label{fundgen} The $[ V_i(q^r)]_{q,\infty}$ generate $\mathscr{K}_\infty$ as a $\mathbb{Z}[ t_R \mid R \in \mathfrak{U}] $-algebra.\end{prop}

\subsection{Subcategories and truncations}\label{truncinf}
Let $ \mathscr{C}' $ be one of the monoidal subcategories $ \mathscr{C}_1\,, \mathscr{C}_1^{ob}$ or $ \mathscr{C}_{\mathcal{Q}}$.
The toroidal Grothendieck ring $\mathscr{K}_{\infty}(\mathscr{C}')$ is defined as the subalgebra of $\mathscr{K}_{\infty}$
generated by the $(q,\infty)$-characters of fundamental representations which are in $\mathscr{C}'$. 

Note that the number of parameters which actually play a role in the structure of $\mathscr{K}_\infty(\mathscr{C}')$, that is that
occur in the relations among the $(q,\infty)$-characters of fundamental representations, is lower than for the whole $\mathscr{K}_\infty$
in general (we will see several examples in the rest of this paper).

We obtain also a natural definition of truncated $(q,\infty)$-character of fundamental representations as a truncation of $F_{\infty}(Y_{i,r})$, as
for $(q,t)$-characters in Section \ref{sec:qgr}, point (5). It will be denoted by $\widetilde{[L]}_{q,\infty}$ for a fundamental representation $L$. 
As for $(q,t)$-characters, they generate a subring of the quantum torus $\mathscr{Y}_\infty$ different but isomorphic to $\mathscr{K}_\infty(\mathscr{C}')$.
Indeed, as we proved $\mathscr{K}_{\infty}$ is a flat deformation of the classical Grothendieck ring, its structure is governed by the multiplicities of
dominant monomials. All of them occur in the truncated $(q,\infty)$-characters by construction.

\begin{rem} As a generalization of Remark \ref{usesmall}, if the truncated $q$-character of a representation $L$ has a unique dominant monomial and all the monomials of the truncation 
have multiplicity $1$, then it lies in the subring of truncated $(q,\infty)$-characters. Then the truncated $(q,\infty)$-character of $L$ can 
be defined as a sum of commutative monomials.
\end{rem}

\subsection{The case of $sl_2$}\label{sec:casesl2}
In the case of $ \overline{\mathfrak{g}}= sl_2 $, the quantum Cartan matrix is $C(z) = (z + z^{-1}) $ and its inverse is (cfr. \eqref{20190221:eq2})
$$
\widetilde{C}(z) = (z - z^{3})\sum_{n\geq 0}z^{4n} = z - z^3 + z^5 - z^7 + z^9 + \ldots
\,.
$$
In particular, the periodicity property reads $ \widetilde{C}(m) = \widetilde{C}(m+4)$, and moreover $ \widetilde{C}(m) = - \widetilde{C}(m+2)$ for $ m \geq 1$.

By the results of the previous section, the toroidal Grothendieck ring $\mathscr{K}_{\infty}$ is the subalgebra of the quantum torus $ \mathscr{Y}_{\infty} $ generated by the (images of the) classes of the fundamental modules
$$[ V(q^{2r})]_{q,\infty}: = \chi_{q,t}(V(q^{2r})) = Y_{2r} + Y_{2r+2}^{-1} \,, r \in \mathbb{Z}\,. $$
In particular, the simple modules are indexed by dominant commutative monomials in the variables $ Y_{2r}$, $ r \in \mathbb{Z}$, and it is possible to check that (assuming $ p < s $) 
$$ \mathcal{N}_a(1,2p;1,2s) = 0 \text{ for } a > 2(s - p)\,, \quad \text{and} \quad \mathcal{N}_{2(s-p)}(1,2p;1,2s) = 1\,. $$
The relations $ \mathscr{R}_k $ give, for each $ k $, $t_{-2k-2}^\frac12 t_{2k+2}^\frac12 = t_{-2k}^\frac12 t_{2k}^\frac12 = \ldots = t_{-2}^\frac12 t_2^\frac12$.
Let us denote this quantity by $s$. For $ h \geq 1$, we can compute,
\begin{equation}\label{20190227:eq9}
\begin{array}{l}
\displaystyle{
\vphantom{\Big(}
[ V(q^{2r}) ]_{q,\infty}
\ast
[ V(q^{2r+2h}) ]_{q,\infty}
=
\alpha(h)
 [L(Y_{2r}Y_{2r+2h} )]_{q,\infty}
 + \delta_{h,1}
\,,}
\end{array}
\end{equation}
where $ [L(Y_{2r}Y_{2r+2h} )]_{q,\infty} := Y_{2r} Y_{2r+2h} +
Y^{-1}_{2r+2} Y^{-1}_{2r+2h+2} +Y_{2r} Y^{-1}_{2r+2h+2} + \delta_{h>1} Y^{-1}_{2r+2} Y_{2r+2h}  $
and $ \alpha(h) =  {t_0}^{(-1)^h}s^{-(-1)^h}$.
We obtain the commutation relations
\begin{equation}\label{20190829:eq1aa}
\begin{array}{l}
\displaystyle{
\vphantom{\Big(}
[V(q^{2r}) ]_{q,\infty}
\ast
[V(q^{2r+2})]_{q,\infty}
-
t_0^{-2}s^2
[V(q^{2r+2})]_{q,\infty}
\ast
[V(q^{2r})]_{q,\infty}
=
1
-
t_0^{-2}s^2
\,,}\\
\displaystyle{
\vphantom{\Big(}
[V(q^{2r})]_{q,\infty}
\ast
[V(q^{2r+2h})]_{q,\infty}
=
\alpha(h)^2
[V(q^{2r+2h})]_{q,\infty}
\ast
[V(q^{2r})]_{q,\infty}
\,,\quad \text{for } h > 1
\,.}
\end{array}
\end{equation}

By \cite[Theorem 7.3, Example 7.4]{HL:qGro}, the $t$-deformed quantum Grothendieck ring for $ sl_2$ has a presentation given by generators $ \chi_{q,t}(V(2r))$, $ r \in \mathbb Z$, and relations
\begin{equation}
\begin{array}{l}
\displaystyle{
\vphantom{\Big(}
\chi_{q,t}(V(q^{2r})) \ast \chi_{q,t}(V(q^{2r+2}))
=
t^{-2}
\chi_{q,t}(V(q^{2r+2}))
\ast 
\chi_{q,t}(V(q^{2r}))
+
1 - t^{-2}
\,,}\\
\displaystyle{
\vphantom{\Big(}
\chi_{q,t}(V(q^{2r})) \ast \chi_{q,t}(V(q^{2r'}))
=
t^{2(-1)^{r'-r}}
\chi_{q,t}(V(q^{2r'}))
\ast
\chi_{q,t}(V(q^{2r}))
\,, \quad \text{if } r' > r+1
\,.}
\end{array}
\end{equation}

As $ \mathscr{K}_{\infty}$ is generated as a $ \mathbb{Z}[ t_R\mid  R \in \mathfrak{U}]$-algebra
by the classes of fundamental modules $[V(2r)]_{q,\infty}$, $ r \in \mathbb Z$,
equation \eqref{20190829:eq1aa} provides a surjective homomorphism 
$$
\begin{array}{c}
\displaystyle{
\vphantom{\Big(}
\mathscr{K}_{t}(\mathscr{C}_{\mathbb{Z}}) \longrightarrow  \mathscr{K}_{\infty}
\,,}\\
\displaystyle{
\vphantom{\Big(}
\chi_{q,t}(V(q^{2r}))  \mapsto [V(q^{2r})]_{q,\infty}
\,,\quad
t \mapsto t_0s^{-1} \,.
}
\end{array}
$$
Since $ \mathscr{K}_{\infty}$ is a flat deformation of $ \mathscr{K}_{t}(\mathscr{C}_{\mathbb{Z}})$ this map is in fact an isomorphism. 
Hence, in the $sl_2$-case, we just get the quantum deformed Grothendieck ring $\mathscr{K}_t(\mathscr{C}_{\mathbb{Z}})$ with an extension of scalars.
Even if not providing a genuine toroidal structure, this example works as a warmup for the general case.
It is moreover useful to get an idea for the issues and the general strategy needed in order to define toroidal $T$-systems, as shown below.

We recall that (\cite{N2}, \cite[Prop 5.6]{hl}), given $p \in 2 \mathbb{Z} $ and $ k \geq 1 $, the $ (q,t)$-characters of the Kirillov-Reshetikhin modules $ W_{k,p}:= L(m_{k,p})$ satisfy the following quantum $T$-system:
\begin{equation}\label{20190227:eq7}
\chi_{q,t}(W_{k,p}) \ast \chi_{q,t}(W_{k,p+2})
 = t^{-1}
 \chi_{q,t}(W_{k-1,p+2})
\ast
\chi_{q,t}(W_{k+1,p})
+ 1
\,.
\end{equation}
See \cite{KNS} for a general review on $T$-systems.

\begin{defi} For $p \in 2 \mathbb{Z},  k \geq 1 $, we define the $(q,\infty)$-character of the KR-module :
\begin{equation}\label{20190227:eq5}
[W_{k,p}]_{q,\infty} := F_\infty(Y_p \cdots Y_{p+2(k-1)}) 
=
\sum_{i=0}^k
Y_p \cdots Y_{p+2(i-1)}Y_{p+2(i+1)}^{-1}\cdots Y_{p+2k}^{-1}\in \mathscr{K}_{\infty}
\,.
\end{equation}
\end{defi}

We remark that the classes $[W_{k,p}]_{q,\infty}$ a priori live inside the quantum torus $\mathscr{Y}_{\infty}$, and we should prove that they belong to the toroidal Grothendieck ring $\mathscr{K}_{\infty}$. One way to do so is to give a recursive formula which expresses each $[W_{k,p}]_{q,\infty} $ as a polynomial in the classes of fundamental modules $ [V(q^p)]_{q,\infty} = [W_{1,p}]_{q,\infty} $, as in Proposition \ref{20190228:prop1} below.  

An argument analogous to that of \cite[Prop 5.6]{HL:qGro} proves the following deformed version of the quantum $T$-system \eqref{20190227:eq7} for the classes $[W_{k,p}]_{q,\infty}$.
\begin{prop}
Let $p \in 2 \mathbb{Z} $ and $ k \geq 1 $. Then, the classes $ [W_{k,p}]_{q,\infty}$ satisfy :
\begin{equation}\label{20190227:eq6}
[W_{k,p}]_{q,\infty} \ast [W_{k,p+2}]_{q,\infty}
=
t_0^{-1}
s
[W_{k-1,p+2}]_{q,\infty}
\ast
[W_{k+1,p}]_{q,\infty}
+
1
\,.
\end{equation}
\end{prop}

Note that equation \eqref{20190227:eq9} with $ h = 1 $ is an instance of \eqref{20190227:eq6} in the case of $ k = 1 $.

\begin{prop}\label{20190228:prop1}
Let $ p \in 2\mathbb{Z}$, $ \ell \geq 1 $. Set $ a(\ell) = \delta_{\ell \notin 2\mathbb{Z}}$. We have :
\begin{equation}\label{20190228:eq1}
[W_{\ell,p}]_{q,\infty} =
t_0s^{-1}
\Big(
\big(t_0s^{-1}\big)^{-a(\ell)}
[W_{1,p}]_{q,\infty} \ast [W_{\ell-1,p+2}]_{q,\infty} - [W_{\ell-2,p+4}]_{q,\infty}
\Big)
\,.
\end{equation}
\end{prop}
\begin{proof}
We work by induction on $ \ell $. For $ \ell = 2 $, it is the $T$-system \eqref{20190227:eq6} for $ k = \ell -1 = 1$. As we need a two-step induction, we shall also consider the case $ \ell = 3 $. The $T$-system \eqref{20190227:eq6} for $ k = \ell -1 = 2 $ can be rewritten as
\begin{equation}\label{20190228:eq3}
[W_{3,p}]_{q,\infty}
=
t_0s^{-1}
[W_{1,p+2}]_{q,\infty}^{-1}
\ast
\Big(
[W_{2,p}]_{q,\infty} \ast [W_{2,p+2}]_{q,\infty} - 1
\Big)
\,.
\end{equation}
By the bar-invariance of $[W_{k,p}]_{q,\infty}$, we also have
\begin{equation}\label{20190228:eq5}
[W_{2,p}]_{q,\infty}
=
t_0^{-1}s
\Big(
[W_{1,p+2}]_{q,\infty} \ast [W_{1,p}]_{q,\infty} - 1
\Big)
\,.
\end{equation}
Substituting, we obtain the claim
\begin{equation}\label{20190228:eq6}
\begin{array}{l}
\displaystyle{
\vphantom{\Big(}
[W_{3,p}]_{q,\infty}
=
t_0s^{-1}
\Big(
[W_{1,p}]_{q,\infty} 
\ast
t_0^{-1}s
[W_{2,p+2}]_{q,\infty} 
-
[W_{1,p+4}]_{q,\infty}
\Big)
\,.}
\end{array}
\end{equation}
Next, let us assume that $ \ell > 3 $ and that \eqref{20190228:eq1} holds for all values  $\ell' $ strictly smaller than $ \ell$.  By the bar-invariance of $[W_{k,p}]_{q,\infty}$ this amounts to assume that the identity
\begin{equation}\label{20190228:eq7}
[W_{\ell',p}]_{q,\infty} =
t_0^{-1}s
\Big(
\big(t_0s^{-1}\big)^{a(\ell')}
[W_{\ell'-1,p+2}]_{q,\infty}
\ast
[W_{1,p}]_{q,\infty} - [W_{\ell'-2,p+4}]_{q,\infty}
\Big)
\,,
\end{equation}
holds for all $ \ell' < \ell$ as well. Let us consider the toroidal $T$-system \eqref{20190227:eq6} for $ k = \ell -1$. 
By substituting the recursive formulas \eqref{20190228:eq7} for $ [W_{\ell-1,p}]_{q,\infty} $ and \eqref{20190228:eq1} for $ [W_{\ell-1,p+2}]_{q,\infty} $, the LHS of \eqref{20190227:eq6} becomes
\begin{equation}\label{20190228:eq8}
\begin{array}{l}
\displaystyle{
\vphantom{\Big(}
\Big(
\big(t_0s^{-1}\big)^{a(\ell-1)}
[W_{\ell-2,p+2}]_{q,\infty}
\ast
[W_{1,p}]_{q,\infty} - [W_{\ell-3,p+4}]_{q,\infty}
\Big)
\,}\\
\displaystyle{
\vphantom{\Big(}
\ast
\Big(
\big(t_0^{-1}s\big)^{a(\ell-1)}
[W_{1,p+2}]_{q,\infty} \ast [W_{\ell-2,p+4}]_{q,\infty} - [W_{\ell-3,p+6}]_{q,\infty}
\Big)
\,}\\
\displaystyle{
\vphantom{\Big(}
=
[W_{\ell-2,p+2}]_{q,\infty}
\ast
[W_{1,p}]_{q,\infty}
\ast
[W_{1,p+2}]_{q,\infty} \ast [W_{\ell-2,p+4}]_{q,\infty}
\,}\\
\displaystyle{
\vphantom{\Big(}
-
\big(t_0s^{-1}\big)^{a(\ell-1)}
[W_{\ell-2,p+2}]_{q,\infty}
\ast
[W_{1,p}]_{q,\infty}
\ast
[W_{\ell-3,p+6}]_{q,\infty}
\,}\\
\displaystyle{
\vphantom{\Big(}
-
\big(t_0^{-1}s\big)^{a(\ell-1)}
[W_{\ell-3,p+4}]_{q,\infty}
\ast
[W_{1,p+2}]_{q,\infty} \ast [W_{\ell-2,p+4}]_{q,\infty}
\,}\\
\displaystyle{
\vphantom{\Big(}
+
t_0^{-1}s
[W_{\ell-4,p+6}]_{q,\infty}
\ast
[W_{\ell-2,p+4}]_{q,\infty}
+ 1
\,}\\
\displaystyle{
\vphantom{\Big(}
=
[W_{\ell-2,p+2}]_{q,\infty}
\ast
\Big(
[W_{1,p}]_{q,\infty}
\ast
\Big(
[W_{1,p+2}]_{q,\infty} \ast [W_{\ell-2,p+4}]_{q,\infty}
-
\big(t_0s^{-1}\big)^{a(\ell-1)}
[W_{\ell-3,p+6}]_{q,\infty}
\Big)
\,}\\
\displaystyle{
\vphantom{\Big(}
-
[W_{\ell-2,p+4}]_{q,\infty}
\Big)
+ 1
\,}\\
\displaystyle{
\vphantom{\Big(}
=
t_0^{-1}s
[W_{\ell-2,p+2}]_{q,\infty}
\ast
\Big(
\big(t_0^{-1}s\big)^{a(\ell)-1}
[W_{1,p}]_{q,\infty}
\ast
[W_{\ell-1,p+2}]_{q,\infty}
-
t_0s^{-1}
[W_{\ell-2,p+4}]_{q,\infty}
\Big)
+ 1
\,.}
\end{array}
\end{equation}
where the first equality is given by the $T$-system \eqref{20190227:eq6} with $ k := \ell-3$ and $ p := p+4$, the second equality is given by \eqref{20190228:eq7} with $ \ell' := \ell-2 $ and $ p := p+2$ together with the fact that $a(k)+ a(k+1) = 1 $.
The last equality is given by \eqref{20190228:eq1} with $ \ell := \ell-1 $ and $ p := p+2$.
Comparing with the RHS of \eqref{20190227:eq9} we can conclude.\end{proof}

\begin{rem}
A version of a (classical) $T$-system in type $A$ is given in \cite[Sec. 2]{KNS}, and a closed formula for its solution $T^{(a)}_{m}(u)$ is given in terms of a certain determinant in  \cite[Theorem 6.2]{KNS}.
For quantum or toroidal $T$-systems we similarly conjecture that we can express the class of a Kirillov-Reshetikhin module  $[W_{\ell,p+2}]_{q,\infty}$ as a polynomial in the classes of the fundamental modules by computing a particular row-determinant. However its dependance on the quantum parameter is not clear yet.
\end{rem}


\section{Toroidal cluster algebra structure}\label{cinq}

In this section we see how toroidal Grothendieck rings provide examples of toroidal cluster algebras. 
The categories $\mathscr{C}_1$ in $ADE$-types discussed above are the first monoidal categories which were related
to cluster algebras in \cite{hl}. In the main result (Theorem \ref{thm:toroidaliso}) we establish that 
toroidal Grothendieck rings of these categories $\mathscr{C}_1$ are toroidal cluster algebras.

\subsection{Monoidal categorications}\label{cinqun}
We recall (see \cite{hl}) that a monoidal category $\mathscr{M}$ is said to be a monoidal categorification of a (classical) cluster algebra $ \mathcal{A} $
if there exists a ring isomorphism
$$\mathcal{A} \overset{\sim}{\longrightarrow} \mathscr{K}(\mathscr{M})
$$
so that cluster variables are certain classes of simple modules $L$ in the Grothendieck ring of $ \mathscr{M}$.
In particular, the cluster monomials of $\mathcal{A} $ (namely, monomials composed of cluster variables in the same seed) correspond to classes of real
(i.e. such that $ L \otimes L$ is still simple) simple objects of $\mathscr{K}(\mathscr{M})$. 

Various examples of monoidal categorification appeared as categories of finite-dimensional $U_q(\mathfrak{g})$-modules \cite{hl, hlad, hljems, bc}.
Other examples of monoidal categorification of cluster algebras can be given for instance by perverse sheaves on quiver varieties, representations of 
quiver-Hecke algebras, or equivariant perverse coherent sheaves on the affine Grassmannian (see \cite{Nak3, Q, kkko, CW} or \cite{H:bour} for a review)
In \cite{hl} Hernandez and Leclerc conjectured (Conjecture 4.6) that the category $\mathscr{C}_1$
is a monoidal categorification of a classical cluster algebra $\mathcal{A}$
with the same Dynkin type as $\overline{\mathfrak{g}}$, and they proved this conjecture
for $ \overline{\mathfrak{g}}$ of type $A_n$ and $D_4$.
In \cite{Nak3} the conjecture was later proved in type $ADE$, using the geometry of quiver varieties.
Again through a geometric approach \cite{Q} proved the conjecture in type ADE for more larger categories $\mathscr{C}_\ell$, $\ell\geq 1$, 
introduced by Hernandez-Leclerc.
In type $A_n$ and $D_n$ the same authors \cite{hlad} proved that the category $\mathscr{C}_1^{ob}$
also provides a monoidal categorification of cluster algebras of the same type.
The proof is similar to \cite{hl}, but the main calculations are in this case simpler.

For the category $\mathscr{C}_\mathcal{Q}$ it is proved in \cite[Theorem 6.1]{HL:qGro} that the $t$-deformed quantum Grothendieck ring $ \mathscr{K}_t(\mathscr{C}_{\mathcal{Q}})$
is isomorphic to $A_t(\mathfrak{n})$, the quantum coordinate ring of the unipotent group $N$ associated with the Lie subalgebra $\mathfrak{n}\subset \mathfrak{g}$, 
and it was proved in \cite{GLS} that the latter possesses a quantum cluster algebra structure (see also \cite{F1, F2} for very recent advances on the
structure of these categories). 

We study the possibility of producing examples of toroidal cluster algebras by using the toroidal Grothendieck rings $\mathscr{K}_{\infty}(\mathscr{C}')$ 
for $\mathscr{C}'$ one of the categories of the previous section. One of our main results is the proof for the categories $\mathscr{C}_1$. 
This gives an incarnation of a toroidal cluster algebra for each finite cluster type. 
Before proving the result, let us show an instance of this phenomenon for the categories $\mathscr{C}_{\mathcal{Q}}$ and $\mathscr{C}_1^{ob}$.

\subsection{First examples}
\begin{ex}\label{ex:quantumtoroidal1}
Let $ \overline{\mathfrak{g}} = sl_3 $ and $ \mathscr{C}' = \mathscr{C}_{\mathcal{Q}} $.
Let 
$$
\begin{array}{l l}
X_1 = [V_1(1)]_{q,\infty}\,,& X_2 = [V_2(q)]_{q,\infty}
\,\\
X_3 = [L(Y_{1,0}Y_{1,2})]_{q,\infty}\,, & X_1' = [V_1(q^2)]_{q,\infty}
\,.
\end{array}
$$
By comparing with Example \ref{ex:multidefquanumgroth}, it is easily seen that the toroidal cluster algebra structure of type $A_1$ with two parameters in the example of Section \ref{sec:afirstexample} is the structure obtained from $ \mathscr{K}_{\infty} (\mathscr{C}_{\mathcal{Q}})$ when we denote\footnote{Note that in the example of Section \ref{sec:afirstexample} the parameters are actually denoted by $ t_1$ and $ t_2$ respectively, following the notation introduced in that section.}
\begin{equation}\label{20190503:eq5}
t_{(1)} : = \prod_{a \in \ZZ} t_a^{\mathcal{N}_a(1,0,2,1)}\,,
\quad\text{and}\quad
t_{(2)} : = \prod_{a \in \ZZ} t_a^{-\mathcal{N}_a(1,0,1,2)-\mathcal{N}_a(1,0,2,1)}\,.
\end{equation}
In fact, in addition to the quasi-commutation relations \eqref{20181130:eq3b}, \eqref{20181130:eq9b} and \eqref{20181130:eq16b},
we can use the explicit expression of $[L(Y_{1,0}Y_{1,2})]_{q,\infty} $ at the end of Example \ref{ex:multidefquanumgroth}
and conclude that all quasi-commutation relations between the classes $ [V_1(1)]_{q,\infty}$, $ [V_2(q)]_{q,\infty} $, $ [L(Y_{1,0}Y_{1,2})]_{q,\infty}$ and $ [V_1(q^2)]_{q,\infty}$ coincide with the quasi-commutation relations between the toroidal cluster variables $X_1$, $X_2$, $X_3$ and $X_1'$, under the reparametrization \eqref{20190503:eq5}.
\end{ex}

\begin{ex}\label{ex:quantumtoroidal2}
Let $ \overline{\mathfrak{g}} = sl_3 $ and $ \mathscr{C}' = \mathscr{C}^{ob}_1 $.
Following \cite{hlad}, let
$$
\begin{array}{l l}
X_1 = \widetilde{[V_1(q^2)]}_{q,\infty} = Y_{1,2}
\,,&
X_2 = \widetilde{[V_2(q^3)]}_{q,\infty} = Y_{2,3}
\,,\\
X_3 = \widetilde{[L(Y_{1,0}Y_{1,2})]}_{q,\infty} = Y_{1,0}Y_{1,2}
\,,&
X_4 = \widetilde{[L(Y_{2,1}Y_{2,3})]}_{q,\infty} = Y_{2,1}Y_{2,3}
\,.
\end{array}
$$
If we let $ t_1 $ and $ t_2 $ be defined as in \eqref{20190503:eq5},
we obtain a toroidal cluster algebra structure on the toroidal Grothendieck ring $ \mathscr{K}_{\infty}(\mathscr{C}_1^{ob})$ with initial toroidal seed
$$
\mathcal{S} : = ( X_1, X_2, X_3, X_4, \widetilde{B}),
$$
with two coefficients ($X_3,X_4$), and with
$$
\widetilde{B}^T
=
\begin{pmatrix}
0 & -1 & -1 & 1 \\ 1 & 0 & 0 & -1
\end{pmatrix}
\,.
$$
The quasi-commutative structure on $ \mathcal{S} $ is given by
$$
X_i \ast X_j = {t_1}^{\Lambda_1(i,j)}{t_2}^{\Lambda_{2}(i,j)} X_j \ast  X_i
$$
where
$$
\Lambda_1
=
\begin{pmatrix}
0 & 1 & 1 & 0 \\ -1 & 0 & 1 & 1 \\ -1 & -1 & 0 & -1 \\ 0 & -1 & 1 & 0
\end{pmatrix}
\,,
\quad
\Lambda_{2}
=
\begin{pmatrix}
0 & 0 & 1 & 0 \\ 0 & 0 & 1 & 1 \\ -1 & -1 & 0 & -1 \\ 0 & -1 & 1 & 0
\end{pmatrix}
\,.
$$
Since $ \widetilde{B}^T\Lambda_1 = \begin{pmatrix} 2 \, \text{Id}_2 \mid 0 \end{pmatrix} $ and $ \widetilde{B}^T\Lambda_{2} = \begin{pmatrix} \text{Id}_2 \mid 0 \end{pmatrix} $,
both pairs $ (\widetilde{B}, \Lambda_1) $ and $ (\widetilde{B}, \Lambda_2) $ are compatible.
By Proposition \ref{prop:mutationtoroidalseed} all mutated pairs $(\mu_k(\widetilde{B}), \mu_k(\Lambda_a))$ ($ a = 1,2$) are still compatible.

We can mutate the initial toroidal seed $ \mathcal{S}$ in direction $ 1 $ or $ 2 $, corresponding to the exchangeable variables.
The result of all possible iterated mutations is encoded in the exchange graph below.
At each step, the matrix $ \widetilde{B}$ mutates according to the usual rule \eqref{20190801:eq2}, and the quasi-commutation matrices  $\Lambda_1\,, \Lambda_2$ mutate according to \eqref{20190801:eq1}.
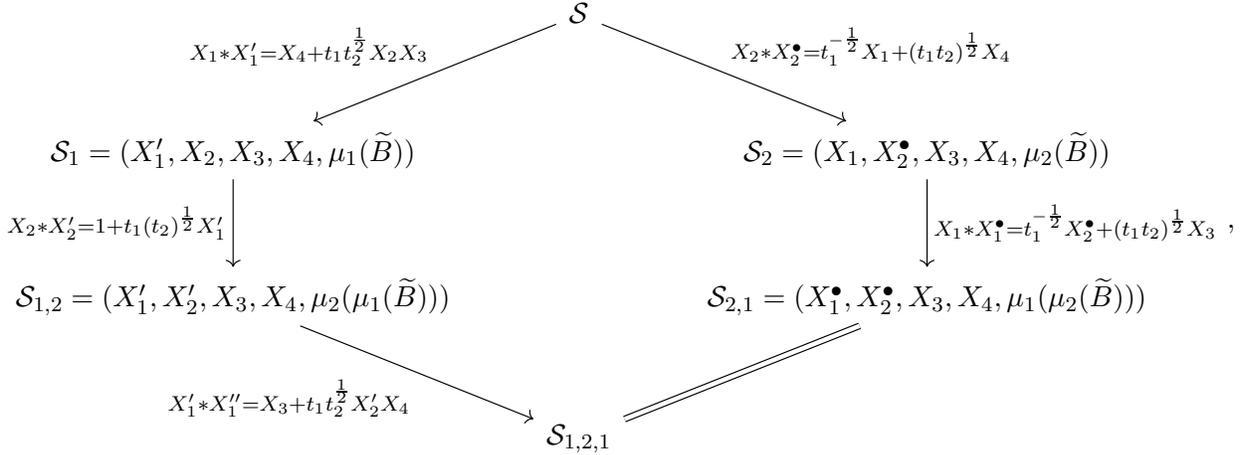
\begin{equation*}
\begin{tikzcd}[row sep=3em]
& \mathcal{S} \arrow[ld, swap, "X_1 \ast X_1'  = X_4 +t_1t_2^{\frac12} X_2X_3 "] \arrow[rd, "X_2 \ast X_2^{\bullet} = t_1^{-\frac12}X_1 +(t_1t_2)^{\frac12} X_4
"] & \\
\mathcal{S}_1 = (X_1', X_2,X_3,X_4, \mu_1(\widetilde{B}))  \arrow[d, swap, "X_2\ast X_2' = 1 + t_1(t_2)^{\frac12}X_1' "]
&& \mathcal{S}_2 = (X_1, X_2^{\bullet},X_3,X_4, \mu_2(\widetilde{B})) \arrow[d, "X_1\ast X_1^{\bullet} = t_1^{-\frac12}X_2^{\bullet} +(t_1t_2)^{\frac12} X_3 "]\\
\mathcal{S}_{1,2} = (X_1',X_2',X_3,X_4,\mu_2(\mu_1(\widetilde{B}))) \arrow[dr, swap, "X_1' \ast X_1'' = X_3 + t_1t_2^{\frac12} X_2'X_4 "] &&  \mathcal{S}_{2,1} = (X_1^{\bullet}, X_2^{\bullet}, X_3, X_4, \mu_1(\mu_2(\widetilde{B}))) \arrow[dl, equal] \\
& \mathcal{S}_{1,2,1} & 
\end{tikzcd}\,,
\end{equation*}
where the last toroidal seed is $\mathcal{S}_{1,2,1}  = (X_1'',X_2',X_3,X_4,\mu_1(\mu_2(\mu_1(\widetilde{B})))) $, which we will see coincides with $ \mathcal{S}_{2,1}$. The toroidal exchange graph is therefore finite (see Theorem \ref{thm:exchgraph}).
In fact, by computing the corresponding quasi-commutation relations inside the toroidal Grothendieck ring $\mathscr{K}_{\infty}(\mathscr{C}^{ob}_1)$ we obtain the identification
$$
\begin{array}{c}
X_1' =  \widetilde{[L(Y_{1,0}Y_{2,3})]}_{q,\infty} = Y_{1,0}Y_{2,3} + Y_{1,2}^{-1}Y_{2,1}Y_{2,3}
\,,\,\,\,
X_2^{\bullet} = X_1'' = \widetilde{[V_2(q)]}_{q,\infty} = Y_{2,1} + Y_{1,2}Y_{2,3}^{-1}
\,,\\
X_1^{\bullet} = X_2' \, = \widetilde{[V_1(1)]}_{q,\infty} = Y_{1,0} + Y_{1,2}^{-1}Y_{2,1} + Y_{2,3}^{-1}
\end{array}
$$
Thus, by swapping $X_2^{\bullet} $ and $X_1^{\bullet} $ we can identify the seeds $ \mathcal{S}_{1,2,1} $ and $\mathcal{S}_{2,1}$, while their corresponding exchange matrices and quasi-commutation matrices
coincide up to a reordering of the first two rows and columns.
We obtain a structure of toroidal cluster algebra of type $A_2$ with two quantum parameters.

By combining the exchange relations between the seeds $ \mathcal{S}_1 \mapsto \mathcal{S}_{1,2} $ and $ \mathcal{S} \mapsto \mathcal{S}_1 $, we can write $ X_1^{\bullet} $ as
$$
X_1^{\bullet} = 
X_2^{-1}
+
X_1^{-1}
X_2^{-1}
X_4
+
X_1^{-1}
X_3
\,,
$$
which is an instance of the Laurent phenomenon for toroidal cluster algebras.
\end{ex}

\subsection{Grothendieck rings as cluster algebras for the category $ \mathscr{C}_1$}
Let $ \xi: I \longrightarrow \{0,1\} $ be a height function associated to a bipartite quiver of Dynkin type ADE.
We can attach to $ \overline{\mathfrak{g}}$ a finite quiver $ \mathcal{Q}_1$ whose vertex set is given by
$$ V_{\xi,1} := \{(i,r) \in \hat{I} \mid r = \xi_i-1, \xi_i+1\},$$
and arrows
$$ ((i,r) \rightarrow (j,s)) \leftrightarrow (C_{i,j} \neq  0 \text{ and } s = r + C_{i,j}).$$

This amount to assuming that $ \xi_i = 1 $ whenever $ i $ is a source, the vertices of the Dynkin diagram corresponding to the vertices $(i, \xi_i+1)$ in $ V_{\xi,1}$.
This quiver corresponds to (a shifted version of) the full subquiver described in \cite{hljems}, which is built from one of the two isomorphic connected components of the infinite quiver $\widetilde{\mathcal{Q}}$.
Let $\mathbf{z}_1:= \{z_{(i,r)} \mid (i,r) \in V_{\xi, 1} \} $ be a set of commuting variables, and let $ \mathcal{A}(\mathcal{Q}_1) \subset \mathbb{Q}(\mathbf{z}_1)$ be the classical cluster algebra with initial seed $(\mathbf{z}_1, \mathcal{Q}_1)$,
where the variables $ z_{(i,r)} $ with $ r = \xi_i-1$ are considered as coefficients.
For $ (i,r) \in V_{\xi,1}$, let $ m_{i,r} := \text{max}\{ k \mid r + 2k \leq 2\} +1 $.
On the other hand, recall that the simple modules in $ \mathscr{C}_1$ are indexed by dominant commutative monomials in the variables $Y_{i,p}$, where $ i \in I $ and $ p \in \{\xi_i, \xi_i+2\}$. Note that in this section we will thus often use the notation $ L(Y_{i,r}) $ for the fundamental module $ V_i(q^r)$.
Then
 
\begin{thm}\cite{hljems}\label{20190801:thm1}
The assignment
\begin{equation}\label{20190805:eq1}
z_{(i,r)} \mapsto \chi_q\Big(W^{(i)}_{m_{i,r},r+1}\Big) \,,
\end{equation}
extends to a ring isomorphism $ i_1: \mathcal{A}(\mathcal{Q}_1) \overset{\sim}{\longrightarrow} \mathscr{K}(\mathscr{C}_1)$.
\end{thm}

We recall that $W^{(i)}_{m_{i,r},r+1} = L(m^{(i)}_{m_{i,r},r+1}) $. 
It follows from the definition of the vertex set $ V_{\xi,1} $ that $m_{i,\xi_i +1} = 1$ and $m_{i,\xi_i -1} = 2$. Thus, \eqref{20190805:eq1} reads
\begin{equation*}
\begin{array}{c}
z_{(i,\xi_i +1)} \mapsto  \chi_q\Big(W^{(i)}_{1,\xi_i+2}\Big) = \chi_q\Big(L(Y_{i, \xi_i+2})\Big)
\,,\quad
z_{(i,\xi_i -1)} \mapsto  \chi_q\Big(W^{(i)}_{2,\xi_i}\Big) = \chi_q\Big(L(Y_{i, \xi_i}Y_{i,\xi_i+2})\Big).
\end{array}
\end{equation*}

\begin{rem}
In \cite{hl}, the cluster algebra was constructed starting from a different initial seed attached to the quiver $\mathcal{Q}'$ (see \cite[Example 4.1]{hl} for instance).
It is obtained from a copy of the Dynkin diagram of $\overline{\mathfrak{g}}$ with a bipartite orientation, associated to the same height function $ \xi: I \longrightarrow \{0,1\} $, but with arrows $ i \rightarrow j $ whenever $ C_{i,j} = -1 $ and $\xi_i = 0 $.
To each vertex $ i \in I $ we then attach a new vertex $ i' $ corresponding to a coefficient, together with an arrow $ i' \rightarrow i $  (resp. $ i \rightarrow i' $ ) if $ \xi_i = 0 $ (resp. $ \xi_i = 1 $). The quivers $ \mathcal{Q}_1$ and $\mathcal{Q}'$ are mutation equivalent, as $ \mathcal{Q}_1$ can be obtained from $ \mathcal{Q}' $ by performing the sequence of (commuting) mutations along all vertices $ i \in I $ such that $ \xi_i = 1 $.

For example, when $\overline{\mathfrak{g}} $ is of type $ A_3 $, with height function $ \xi_1 = \xi_3 = 0 $, $ \xi_2 = 1 $, the initial seed constructed in \cite{hl} can be depicted
as
\[
\xymatrix@R=2em@C=2em{
\chi_q(V_1(q^2)) \ar[r] & \chi_q(V_2(q))\ar[d] & \chi_q(V_3(q^2)) \ar[l] \\
\chi_q(L(Y_{1,0}Y_{1,2})) \ar[u] & \chi_q(L(Y_{2,1}Y_{2,3})) & \chi_q(L(Y_{3,0}Y_{3,2}))\ar[u]
}
\]
By mutating in direction $ 2$ we obtain a new cluster variable with exchange relation
$$ \chi_q(V_2(q)) \chi_q(V_2(q))' = \chi_q(V_1(q^2))\chi_q(V_3(q^2)) + \chi_q(L(Y_{2,1}Y_{2,3}))\,,$$
and direct computation shows that $ \chi_q(V_2(q))'  =  \chi_q(V_2(q^3)) $.
Thus, the mutated quiver is
\[
\xymatrix@R=2em@C=2em{
\chi_q(V_1(q^2)) \ar[dr] & \chi_q(V_2(q^3))\ar[l]\ar[r] & \chi_q(V_3(q^2)) \ar[dl] \\
\chi_q(L(Y_{1,0}Y_{1,2})) \ar[u] & \chi_q(L(Y_{2,1}Y_{2,3}))\ar[u] & \chi_q(L(Y_{3,0}Y_{3,2}))\ar[u]
}
\]
which coincides with the initial seed obtained through the quiver $ \mathcal{Q}_1$ and the assignment \eqref{20190805:eq1}.  \end{rem}

We study the toroidal Grothendieck ring $ \mathscr{K}_{\infty}(\mathscr{C}_1) $ in the same spirit.
\begin{rem}
In the quantum case, a similar approach is followed by Bittmann \cite{b}, who on the other hand used quantum cluster algebras in order to construct quantum Grothendieck rings of a larger category $\mathcal{O}$ of representations of the quantum Borel subalgebra $ U_q(\mathfrak{b})$.
\end{rem}

\subsection{A toroidal equivalent}

\subsubsection{The toroidal Grothendieck ring $ \mathscr{K}_{\infty}(\mathscr{C}_1) $}\label{sec:6.2.1}
By \cite[Example 6.4]{hl}, the $q$-truncated character of the fundamental modules $ L(Y_{i,\xi_i+2}) $ and of the Kirillov-Reshetikhin modules $ L(Y_{i,\xi_i}Y_{i,\xi_i+2})$ only consists of one dominant monomial, namely
$$ \widetilde{\chi_q(L(Y_{i,\xi_i+2}))} = Y_{i,\xi_i+2}\,,\quad \widetilde{\chi_q(L(Y_{i,\xi_i}Y_{i,\xi_i+2}))} = Y_{i,\xi_i}Y_{i,\xi_i+2}.$$
Moreover, the truncated $ q $-characters for the fundamental modules $ L(Y_{i,\xi_i}) $ are multiplicity-free.
Therefore, following Section \ref{truncinf}, we define the truncated $(q,\infty)$-characters to coincide with the truncated $q$-characters in \cite[Example 6.4]{hl}:
$$ \widetilde{[L(Y_{i,\xi_i+2})]}_{q,\infty} := Y_{i,\xi_i+2} \,,\quad \widetilde{[L(Y_{i,\xi_i}Y_{i,\xi_i+2})]}_{q,\infty} := Y_{i,\xi_i}Y_{i,\xi_i+2}\,,$$
and
$$
\widetilde{[L(Y_{i,\xi_i})]}_{q,\infty} := \begin{cases}
Y_{i,0}(1 + A_{i,1}^{-1}\prod_{j\sim i}(1 + A_{j,2}^{-1}))\,, & \text{if } \xi_i = 0 \\
Y_{i,1}(1+A_{i,2}^{-1})\,, & \text{if } \xi_i = 1 \end{cases}\,.
$$
As the only variables appearing in the expressions above are $ Y_{i,\xi_{i}}^{\pm 1},\, Y_{i,\xi_i+2}^{\pm 1}$, we can deduce that the relations between the truncated $(q,\infty)$-characters only involve three kind of exponents: $\mathcal{N}_s(i,0;k,2)$ where $ \xi_{i} = \xi_{k}$, $\mathcal{N}_s(i,0;j,1)$ and $\mathcal{N}_s(i,0;j,3)$ where $ \xi_{i} = \xi_{j}\pm 1 $ ($ s \in \mathbb{Z}$).

By Remark \ref{20190807:rem1}, all these exponents vanishes for $ s \leq -4$. This implies that the parameters $t_s$, with $s\leq -4$, do not occur in the relations 
describing the quantum toroidal ring of $\mathscr{C}_1$. Moreover, thanks to the explicit form of the relations $\mathscr{R}_k$ in Definition \ref{quotientorus}, this implies that the 
specializations 
\begin{equation}\label{20210111:eq2}
t_{a}^{\frac12}\mapsto 1 \text{ for all } a > 0\,,
\end{equation}
are well-defined (under these specializations, the relations just imply that $t_{-2s} = t_{-2}$ for $s\geq 2$, but the parameters $t_{-2s}$ do not occur).
Note that in general, by performing these specializations we lose some information (see Example \ref{20210111:ex1}), but we still 
have an interesting toroidal structure with two independent deformation parameters $t_0$, $t_{-2}$. In the rest of this Section, we assume these specializations.

\begin{ex}\label{20210111:ex1}
Let $\overline{\mathfrak{g}}$ be of type $A_4$, and consider the height function $ \xi : \{1,2,3,4\} \longrightarrow \{0,1\}$ given by $ \xi_1 = \xi_3 = 1 $ and $ \xi_2 = \xi_4 = 0$. Then the subring of $\widetilde{\mathscr{Y}}_{\infty}$ of the truncated $(q,\infty)$-characters is generated by
$$ Y_{1,1}^{\pm 1},\,Y_{1,3}^{\pm 1},\,Y_{2,0}^{\pm 1},\,Y_{2,2}^{\pm 1},\,Y_{3,1}^{\pm 1},\,Y_{3,3}^{\pm 1},\,Y_{4,0}^{\pm 1},\,Y_{4,2}^{\pm 1}.$$ 
Let us consider the fundamental modules $W^{(1)}_{1,3}$ and $W^{(4)}_{1,3}$, whose truncated $ (q,\infty)$-characters are given by $\widetilde{[W^{(1)}_{1,3}]}_{q,\infty} = Y_{1,3} $ and $\widetilde{[W^{(4)}_{1,3}]}_{q,\infty} = Y_{4,3} $.
Following \eqref{20190409:eq1} we have
$$ Y_{1,3} * Y_{4,2} = \big( \prod_{k\geq 0} t_{2+10k}^{-1}t_{4+10k}^3t_{6+10k}^{-3}t_{8+10k}\big) Y_{4,2} * Y_{1,3}\,.$$
As the first quantum parameter which appears in this quasi-commutation relation is $t_2$, the classes of the fundamental modules $W^{(1)}_{1,3} $ and $ W^{(4)}_{1,3}$ commute with each other under the specialization \eqref{20210111:eq2}, although they do not commute in $ \widetilde{\mathscr{Y}}_\infty$.
\end{ex}

With abuse of notation, we keep the same notation $\widetilde{\mathscr{Y}}_{\infty}$ (the quantum torus after specialization) and $ \mathscr{K}_{\infty}(\mathscr{C}_1)$ (the ring generated by the truncated $(q,\infty)$-characters, see Section \ref{truncinf}), for simplicity.
Thus we have (still denoting the product by $\ast$)
\begin{equation}\label{20190917:eq1}
Y_{i,p} \ast Y_{j,s} = t_{-2}^{\mathcal{N}_{-2}(i,p;j,s)}t_{0}^{\mathcal{N}_{0}(i,p;j,s)} Y_{j,s}\ast Y_{i,p}\,.
\end{equation}

Let $\Lambda_a $, $ a = -2,0$, be the $ 2n \times 2n $-matrix whose rows and columns are indexed by the set $ V_{\xi,1}$ and whose $((i,p),(j,s))$-th entry $\Lambda_a\big((i,p),(j,s)\big) $ is given by the power of $t_a$ in the $\ast$-product between the monomials $ m^{(i)}_{m_{i,p},p+1} $ and $ m^{(j)}_{m_{j,s},s+1}$ in $ \widetilde{\mathscr{Y}}_{\infty} $. By \eqref{20190917:eq1} we have:
\begin{equation}\label{20190723:eq2}
\begin{array}{l}
\displaystyle{
\vphantom{\Big(}
\Lambda_a\big((i,\xi_i+1),(j,\xi_j+1)\big) = \mathcal{N}_a(i,0;j,\xi_j-\xi_i)\,,
}\\
\displaystyle{
\vphantom{\Big(}
\Lambda_a\big((i,\xi_i+1),(j,\xi_j-1)\big) =  \mathcal{N}_a(i,0;j,\xi_j-\xi_i-2) +  \mathcal{N}_a(i,0;j,\xi_j-\xi_i)\,,
}\\
\displaystyle{
\vphantom{\Big(}
\Lambda_a\big((i,\xi_i-1),(j,\xi_j-1)\big) = 2 \mathcal{N}_a(i,0;j,\xi_j-\xi_i) +  \mathcal{N}_a(i,0;j,\xi_j-\xi_i-2) +  \mathcal{N}_a(i,0;j,\xi_j-\xi_i+2)\,.
}\end{array}
\end{equation}
\begin{rem}
We chose the following ordering on the set $ V_{\xi,1}$:
$$ (1,\xi_1+1) < (2, \xi_2+1) < \ldots < (n,\xi_n+1) < (1, \xi_1-1) < (2, \xi_2-1) < \ldots < (n, \xi_n-1)\,.$$
\end{rem}

\subsubsection{The toroidal cluster algebra $\mathcal{A}_{tor}(\mathcal{Q}_1)$}
We want to construct a toroidal cluster algebra structure with 2 deformation parameters $t_{-2}$, $t_0$ whose specialization $ t_a \mapsto 1$ ($a= -2,0$) recovers the cluster algebra $\mathcal{A}(\mathcal{Q}_1)$.
We consider the same set of variables $\mathbf{z}_1$ as before, but we impose the quasi-commutation relations given by the matrices $\Lambda_{-2}$, $\Lambda_0$ above. We claim that the resulting quantum torus is compatible with the cluster algebra structure encoded in the quiver $\mathcal{Q}_1$, thus providing a toroidal cluster algebra that we denote by $\mathcal{A}_{tor}(\mathcal{Q}_1)$.

Indeed, attached to the quiver $\mathcal{Q}_1 $ there is a $ 2n \times n$-matrix $ \widetilde{B}$, with rows indexed by the set $ V_{\xi,1}$ and columns indexed by its subset $ \overline{V}_{\xi,1} = \{ (i, \xi_i +1) \mid i \in I \}$.
It can be described explicitly as follows:
\begin{equation}\label{20190723:eq1}
\widetilde{B}_{(i,r),(j,s)} = \begin{cases}
1\,, & \text{if } i = j, s = r+2, \text{ or } i \sim j, s = r-1\\
-1\,, & \text{if } i \sim j, s = r+1.
\end{cases}
\end{equation}

We have
\begin{prop}
The pairs $(\widetilde{B}, \Lambda_0) $ and $(\widetilde{B}, \Lambda_{-2}) $ are compatible pairs.
\end{prop}

\begin{proof}
Let us consider a generic entry of the product $ \widetilde{B}^{T}\Lambda_a$, ($ a = -2,0$). For $ (i,r) \in \overline{V}_{\xi,1}$, $ (j,s) \in V_{\xi,1}$ we have
$$
\begin{array}{l}
\displaystyle{
\vphantom{\Big(}
\Big(\widetilde{B}^{T}\Lambda_a\Big)_{(i,r),(j,s)} = \sum_{(k,t) \in V_{\xi,1}} \widetilde{B}_{(k,t),(i,r)}\Lambda_a\big((k,t),(j,s)\big) }\\
\displaystyle{
\vphantom{\Big(}
=
\Lambda_a\big((i,r-2),(j,s)\big) + \sum_{k \sim i} \Big( \Lambda_a\big((k,r+1),(j,s)\big) - \Lambda_a\big((k,r-1),(j,s)\big)\Big)
\,,}\end{array}
$$
where the last equation follows from the expression \eqref{20190723:eq1} for the entries of $ \widetilde{B} $.
We shall now analyze separately the cases $ (j,s) = (j, \xi_j+1) $ (the left block in $\widetilde{B}^{T}\Lambda_a$) and $ (j,s) = (j, \xi_j-1) $ (the right block).
By \eqref{20190723:eq2}, we obtain the following expression for the left block:
\begin{equation}\label{20190723:eq3}
\begin{array}{l}
\displaystyle{
\vphantom{\Big(}
\Big(\widetilde{B}^{T}\Lambda_a\Big)_{(i,\xi_i+1),(j,\xi_j+1)}
}\\
\displaystyle{
\vphantom{\Big(}
= \Lambda_a\big((i,\xi_i-1),(j,\xi_j+1)\big) + \sum_{k \sim i} \Big( \Lambda_a\big((k,\xi_i+2),(j,\xi_j+1)\big)
 - \Lambda_a\big((k,\xi_i),(j,\xi_j+1)\big)\Big)
}\\
\displaystyle{
\vphantom{\Big(}
= - \mathcal{N}_a(i,0,j,\xi_i-\xi_j-2) - \mathcal{N}_a(i,0,j,\xi_i-\xi_j) - \sum_{k\sim i}\mathcal{N}_a(k,0,j,\xi_j-\xi_i+1)
\,,}\end{array}
\end{equation}
The last equality is obtained by observing that for $ k \sim i $ we have
$$
\begin{array}{l}
\Lambda_a\big((k,\xi_i+2),(j,\xi_j+1)\big) - \Lambda_a\big((k,\xi_i),(j,\xi_j+1)\big)
=
\delta_{\xi_i,0} \Big(  \Lambda_a\big((k,\xi_k+1),(j,\xi_j+1)\big)
\\
- \Lambda_a\big((k,\xi_k-1),(j,\xi_j+1)\big) \Big) - \delta_{\xi_i,1} \Lambda_a\big((k,\xi_k+1),(j,\xi_j+1)\big)
\,,
\end{array}
$$
and then by using \eqref{20190723:eq2} and the skew-symmetry of $\mathcal{N}_a(i,p;j,s)$.

By \eqref{20190801:eq3}, equation \eqref{20190723:eq3} can be rewritten as
\begin{equation}\label{20190801:eq5}
\begin{array}{l}
\displaystyle{
\vphantom{\Big(}
\Big(\widetilde{B}^{T}\Lambda_a\Big)_{(i,\xi_i+1),(j,\xi_j+1)} =
\widetilde{C}_{i,j}(\xi_i-\xi_j+a-3)+\widetilde{C}_{i,j}(-\xi_i+\xi_j+a+3)
}\\
\displaystyle{
\vphantom{\Big(}-\widetilde{C}_{i,j}(-\xi_i+\xi_j+a-1)-\widetilde{C}_{i,j}(\xi_i-\xi_j+a+1)
- \sum_{k\sim i}\Big(\widetilde{C}_{k,j}(-\xi_j+\xi_i+a-2)
}\\
\displaystyle{
\vphantom{\Big(}
-\widetilde{C}_{k,j}(\xi_j-\xi_i+a)-\widetilde{C}_{k,j}(-\xi_j+\xi_i+a)+\widetilde{C}_{k,j}(\xi_j-\xi_i+a+2)
\Big)
\,,}\end{array}
\end{equation}
Since $ \mid \xi_i - \xi_j \mid \leq 1 $, and $ \widetilde{C}_{i,j}(m) = 0 $ whenever $ m \leq 0 $ ($1 \leq i,j\leq n$), for $ a = 0 $ we obtain
$$
\begin{array}{l}
\displaystyle{
\vphantom{\Big(}
\Big(\widetilde{B}^{T}\Lambda_0\Big)_{(i,\xi_i+1),(j,\xi_j+1)}
=
\widetilde{C}_{i,j}(-\xi_i+\xi_j+3)-\widetilde{C}_{i,j}(\xi_i-\xi_j+1)
}\\
\displaystyle{
\vphantom{\Big(}
- \sum_{k\sim i}\Big(
-\widetilde{C}_{k,j}(\xi_j-\xi_i)-\widetilde{C}_{k,j}(-\xi_j+\xi_i)+\widetilde{C}_{k,j}(\xi_j-\xi_i+2)
\Big)
\,.}\end{array}
$$
This expression vanishes in the case when $ \xi_j - \xi_i = -1,1 $ (in the former case the terms mutually cancel, and in the latter by application of the induction relation \eqref{20190801:eq4}), and it is equal to $ - 2 \delta_{ij} $ when $ \xi_i = \xi_j $. On the other hand, for $ a = -2$ we obtain
$$
\begin{array}{l}
\displaystyle{
\vphantom{\Big(}
\Big(\widetilde{B}^{T}\Lambda_{-2}\Big)_{(i,\xi_i+1),(j,\xi_j+1)}
=
\widetilde{C}_{i,j}(\xi_j-\xi_i+1)
- \sum_{k\sim i}\widetilde{C}_{k,j}(\xi_j-\xi_i)
\,.}\end{array}
$$
Again, this expression vanishes in the case when $ \xi_j - \xi_i = -1, 1 $ (in the former case because each summand vanishes, in the latter by application of the induction relation \eqref{20190801:eq4}), and it is equal to $ \delta_{ij} $ in the case when $ \xi_i = \xi_j $.

Next, by \eqref{20190723:eq2} we obtain the following expression for the right block of $\widetilde{B}^T \Lambda_a$:
\begin{equation}\label{20190801:eq6}
\begin{array}{l}
\displaystyle{
\vphantom{\Big(}
\big( \widetilde{B}^T \Lambda_a\Big)_{(i, \xi_i+1),(j,\xi_j-1)}
=
\widetilde{C}_{i,j}(\xi_i-\xi_j + a -1) - \widetilde{C}_{i,j}(-\xi_i+\xi_j + a -1)
}\\
\displaystyle{
\vphantom{\Big(}
-\widetilde{C}_{i,j}(\xi_i-\xi_j + a +1) + \widetilde{C}_{i,j}(-\xi_i+\xi_j + a +1)+\widetilde{C}_{i,j}(\xi_i-\xi_j + a -3) 
}\\
\displaystyle{
\vphantom{\Big(}
+ \widetilde{C}_{i,j}(-\xi_i+\xi_j + a +3)-\widetilde{C}_{i,j}(-\xi_i+\xi_j + a -3) - \widetilde{C}_{i,j}(\xi_i-\xi_j + a +3)
}\\
\displaystyle{
\vphantom{\Big(}
- \sum_{k\sim i} \Big( - \widetilde{C}_{k,j}(-\xi_i+\xi_j + a -2) - \widetilde{C}_{k,j}(\xi_i-\xi_j + a +2)
}\\
\displaystyle{
\vphantom{\Big(}
+ \widetilde{C}_{k,j}(-\xi_i+\xi_j + a+2) + \widetilde{C}_{k,j}(\xi_i-\xi_j + a-2) \Big)
\,.
}\end{array}
\end{equation}
The RHS of \eqref{20190801:eq6} coincides with the difference $ \Big(\widetilde{B}^{T}\Lambda_a\Big)_{(i,\xi_i+1),(j,\xi_j+1)} - \Big(\widetilde{B}^{T}\Lambda_a\Big)_{(j,\xi_j+1),(i,\xi_i+1)} $, which clearly vanishes.

In particular, we obtain $ \Big(\widetilde{B}^{T}\Lambda_0\Big) = \begin{pmatrix} -2 \, \text{Id}_2 \mid 0 \end{pmatrix} $ and  $ \Big(\widetilde{B}^{T}\Lambda_{-2}\Big) = \begin{pmatrix} \text{Id}_2 \mid 0 \end{pmatrix} $.
\end{proof}
\begin{rem}
\begin{enumerate}
\item
The matrices $ \Lambda_{-2} $ and $\Lambda_0 $ are linearly independent when the Lie algebra $\overline{\mathfrak{g}}$ is of rank $ n \geq 2 $. In fact, it is immediate from \eqref{20190723:eq2} that the top-left block of size $ n \times n $ in $\Lambda_{-2} $ is constantly zero, while this is not the case for $\Lambda_0$.
\item
When considering all the parameters $t_a$ inside the quantum torus $ \widetilde{\mathscr{Y}}_{\infty} $ without specializing any quantum parameter, we can analogously build matrices $ \Lambda_a $ for any $ a \in \mathbb{Z}$. A similar argument shows that the pair $(\widetilde{B}, \Lambda_2) $ is compatible (but it is in some examples obtained as a linear combination of $ \Lambda_{-2} $ and $\Lambda_0$). However, it follows from the expressions \eqref{20190801:eq5} and \eqref{20190801:eq6} that all other pairs $(\widetilde{B}, \Lambda_a) $ for $ a \neq -2,0,2 $ are not.
In particular, since $ | \xi_i - \xi_j | \leq 1$ whenever $ a > 4 $ the arguments in each summand in both \eqref{20190801:eq5} and \eqref{20190801:eq6} are greater than $0 $, and hence by repeated application of induction relation \eqref{20190801:eq4} the whole expression vanishes. Moreover, by the properties of $ \widetilde{C}(z)$, each summand in the RHS of both \eqref{20190801:eq5} and \eqref{20190801:eq6} vanishes for $ a \leq -4$.
\end{enumerate}
\end{rem}

\subsubsection{A toroidal isomorphism}
We can finally prove the main Theorem of this section.

\begin{thm}\label{thm:toroidaliso}
The assignment
\begin{equation}\label{20190805:eq2}
z_{(i,r)} \mapsto \widetilde{\Big[ W^{(i)}_{m_{i,r},r+1}\Big]}_{q,\infty}\,,
\end{equation}
extends to a ring isomorphism
$$i_{t_{-2},t_0,1}: \mathcal{A}_{tor}(\mathcal{Q}_1) \overset{\sim}{\longrightarrow} \mathscr{K}_{\infty}(\mathscr{C}_1)\,.$$
Besides, the classes of fundamental modules are toroidal cluster variables.
\end{thm}
\begin{proof}
First note that \eqref{20190805:eq2} is explicitly given by
\begin{equation}\label{20190805:eq3}
\begin{array}{l}
z_{(i,\xi_i +1)} \mapsto  \widetilde{\Big[ W^{(i)}_{1,\xi_i+2}\Big]}_{q,\infty} = \widetilde{\Big[ L(Y_{i, \xi_i+2})\Big]}_{q,\infty} = Y_{i, \xi_i+2}
\,,\\
z_{(i,\xi_i -1)} \mapsto \widetilde{\Big[ W^{(i)}_{2,\xi_i}\Big]}_{q,\infty} = \widetilde{\Big[ L(Y_{i,\xi_i}Y_{i, \xi_i+2})\Big]}_{q,\infty} = Y_{i,\xi_i}Y_{i, \xi_i+2}.
\end{array}
\end{equation}
This makes sense, as the classes $\widetilde{\Big[ L(Y_{i, \xi_i+2})\Big]}_{q,\infty} $, $ \widetilde{\Big[ L(Y_{i,\xi_i}Y_{i, \xi_i+2})\Big]}_{q,\infty} $ consist of a monomial which are algebraically independent.

By construction, both algebras $\mathscr{K}_{\infty}(\mathscr{C}_1)$ and $ \mathcal{A}_{tor}(\mathcal{Q}_1) $ live inside the same quantum torus $\mathscr{Y}_\infty$.
Moreover, they both provide deformations of their classical counterparts $ \mathscr{A}(\mathcal{Q}_1) $ and $\mathscr{K}(\mathscr{C}_1)$, which are isomorphic by Theorem \ref{20190801:thm1}.

Next, let us prove the inclusion $\mathscr{K}_{\infty}(\mathscr{C}_1)\subset \mathcal{A}_{tor}(\mathcal{Q}_1) $. As  $\mathscr{K}_{\infty}(\mathscr{C}_1) $ is generated by the truncated classes of its fundamental modules, we need to prove that all these classes can be obtained as toroidal cluster variables.
For the classes $ \widetilde{[L(Y_{i,\xi_i+2})]}_{q,\infty}  $, it is clear by \eqref{20190805:eq3} that they belong to the initial toroidal cluster seed for $\mathcal{A}_{tor}(\mathcal{Q}_1) $. Thus, we are left to consider the classes $  \widetilde{[L(Y_{i,\xi_i})]}_{q,\infty} $.

Using the explicit formulas for the truncated $ (q,\infty)$-characters at the beginning of Section \ref{sec:6.2.1},
we can compute the associated toroidal $T$-system (already known in the case $ t_{-2} = 1 $) and obtain,
\begin{equation}\label{20190724:eq4}
\begin{array}{r}
\displaystyle{
\vphantom{\Big(}
\widetilde{[L(Y_{k,\xi_k+2})]}_{q,\infty} \ast \widetilde{[L(Y_{k,\xi_k})]}_{q,\infty}
=
\prod_{a = 0,-2} t_a^{\frac{-N_a(k,0;k,2)}{2}} \widetilde{[L(Y_{k,\xi_k}Y_{k,\xi_k+2})]}_{q,\infty} 
}\\
\displaystyle{
\vphantom{\Big(}
+ \prod_{a=0,-2} t_a^{-\sum_{i \sim k} \frac{N_a(i,0;k,1)}{2}}
\Big(\ast_{i \sim k } \widetilde{[L(Y_{i,\xi_k+1})]}_{q,\infty}\Big)
\,.}\end{array}
\end{equation}

It can be easily checked that the product of the classes $ \widetilde{[L(Y_{i,\xi_k+1})]}_{q,\infty} $, for $ k\sim i $ does not depend on the order, and it is in fact bar-invariant, hence the expression makes sense.

Note that equation \eqref{20190724:eq4} always corresponds to an exchange mutation in the toroidal cluster algebra $ \mathcal{A}_{tor}(\mathcal{Q}_1)$, however it is a one-step mutation only in the case when $ \xi_k = 1 $
where, according to the toroidal mutation rule \eqref{mutrel} (cfr. \cite[Prop. 4.9]{bz}),
we obtain
\begin{equation}\label{20190724:eq1}
\begin{array}{c}
\displaystyle{
\vphantom{\Big(}
z_{(k,2)}\ast z_{(k,2)}'
=
\prod_{a = 0,-2} t_a^{-\frac12 \Lambda_a((k,0),(k,2))} z_{(k,0)}
+ \prod_{a=0,-2} t_a^{\frac12 \sum_{i \sim k} \Lambda_a((k,2),(i,1))}\Big( \ast_{i \sim k } z_{(i,1)}\Big)\,.
}\end{array}
\end{equation} 
Since $\Lambda_a((k,0),(k,2))= \mathcal{N}_a(k,0;k,2) $ and $\Lambda_a((k,2),(i,1)) = -\mathcal{N}_a(k,0;i,1)$,
the map \eqref{20190805:eq3}, gives an identification of the truncated class $ \widetilde{[L(Y_{k,1})]}_{q,\infty}$ with the mutated cluster variable $ z_{(k,2)}' $ inside $ \mathcal{A}_{tor}(\mathcal{Q}_1)$.

On the other hand, when $ k $ is such that $\xi_k = 0 $, the first-step mutation for the variable $z_{(k,1)}$ does not coincide with the corresponding toroidal $T$-system \eqref{20190724:eq4} (see Remark \ref{20190805:rem1}).
However, we know from the classical theory that after performing the sequence of mutations $ \mathscr{S} := \mu_{i_s} \circ \cdot \circ \mu_{i_1} $ on the quiver $ \mathcal{Q}_1$ (where $ i_1, \ldots, i_s $ are all the indices in $ I $ such that $ \xi_{i_1} = \cdots = \xi_{i_s} = 1 $), we obtain a new quiver $ \mathcal{Q}_1'$ where the exchange relations $ z_{(k,1)}\ast z_{(k,1)}^{(\mathscr{S})} $ coincide with the image of the toroidal $T$-system \eqref{20190724:eq4} for the product $ \widetilde{[L(Y_{k,2})]}_{q,\infty} \ast \widetilde{[L(Y_{k,0})]}_{q,\infty} $, under the specialization $t_a\mapsto 1 $.
This identifies the class of the fundamental module $ L(Y_{k,0})$ with the mutated toroidal variable $z_{(k,1)}^{(\mathscr{S})}$ in the classical case.

By Theorem \ref{thm:exchgraph} it is therefore sufficient to check that this choice of parameters makes \eqref{20190724:eq4} a toroidal cluster algebra exchange relation, namely that we obtain a bar-invariant expression for $\widetilde{[L(Y_{k,0})]}_{q,\infty}$.
Since both summands
$$
\prod_{a = 0,-2} t_a^{\frac{-N_a(k,0;k,2)}{2}} \widetilde{[L(Y_{k,\xi_k+2})]}_{q,\infty}^{-1} \ast \widetilde{[L(Y_{k,\xi_k}Y_{k,\xi_k+2})]}_{q,\infty}
$$
and
$$
\prod_{a=0,-2} t_a^{-\sum_{i \sim k} \frac{N_a(i,0;k,1)}{2}}
\widetilde{[L(Y_{k,\xi_k+2})]}_{q,\infty}^{-1} \ast \Big(\ast_{i \sim k } \widetilde{[L(Y_{i,\xi_k+1})]}_{q,\infty}\Big)
$$
are bar-invariant, the claim follows. Thus, $\mathscr{K}_{\infty}(\mathscr{C}_1) \subset  \mathcal{A}_{tor}(\mathcal{Q}_1) $. 

Let us prove the other inclusion.
For any $ a, b \in \mathbb{Z} $, let us consider the specialization $ (t_{-2},t_0) \mapsto (t^a,t^b)$ as in the Proof of Theorem \ref{pos}.
This gives a quantum cluster algebra $ \mathcal{A}_{t^a,t^b}(\mathcal{Q}_1)$, generated by the fundamental cluster variables, i.e. the images of the classes of the fundamental modules in $\mathscr{C}_1 $. Indeed, this is true for the classical cluster algebra, and then the quantum counter-part of this statement follows from 
\cite{GLS2}. As the cluster algebras are of finite-type, this also follows simply by explicit description of finite-type quantum cluster algebras (as pointed out to the authors by Bernard Leclerc). An analog explicit description would also work in the toroidal case, but let us discuss another argument. By Theorem \ref{thm:laurentphen}, we can expand an element $X\in  \mathcal{A}_{tor}(\mathcal{Q}_1)$ as
$$X = \sum_{M\in \mathcal{L}} P_M(t_{-2},t_0) M\,\text{ where }P_M(t_{-2},t_0)\in\mathbb{Z}[t_{-2}^{\pm \frac{1}{2}},t_0^{\pm \frac{1}{2}}]$$
and $\mathcal{L}$ is a finite set of bar-invariant Laurent monomials in the toroidal variables of the initial seed $ \textbf{z}_1$.
Specializing $ (t_{-2},t_0) \mapsto (t^a,t^b)$ with $a$, $b$ distinct prime numbers, we are reduced to a quantum cluster algebra
discussed above and there is a finite set $\mathcal{M}$ of monomials in the fundamental cluster variables so that
$$X_{(t_{-2},t_0) \mapsto (t^a,t^b)}  = \sum_{M\in\mathcal{L}} P_M(t^a,t^b) M = \sum_{m\in\mathcal{M}} Q_{m,a,b}(t) m\text{ with }Q_{m,a,b}(t)\in\mathbb{Z}[t^{\pm 1/2}].$$ 
Each element $m\in \mathcal{M}$, now seen in $\mathcal{A}_{tor}(\mathcal{Q}_1)$, can be expanded as
$$m = \sum_{M\in\mathcal{L}} R_{M,m}(t_{-2},t_0) M.$$
($\mathcal{L}$ is enlarged with a finite number of bar-invariant Laurent monomials in the initial toroidal variables).
As $\mathcal{M}$ is free over $\mathbb{Q}(t_{-2}^{\frac{1}{2}},t_0^{\frac{1}{2}})$, there are 
$S_{M,m}(t_{-2},t_0)\in\mathbb{Q}(t_{-2}^{\frac{1}{2}},t_0^{\frac{1}{2}})$ 
so that 
$$\sum_{M\in\mathcal{L}} S_{M,m'}(t_{-2},t_0) R_{M,m}(t_{-2},t_0) = \delta_{m,m'}\text{ for any $m,m'\in\mathcal{M}$}.$$ 
They can be specialized at $(t_{-2},t_0) \mapsto (t^a,t^b) $ for all distinct primes $a$, $b$ large enough. Let
$$Q_m(t_{-2},t_0) = \sum_{M\in\mathcal{L}} S_{M,m}(t_{-2},t_0)P_M(t_{-2},t_0).$$
So $P_M(t^a,t^b) = \sum_{m\in\mathcal{M}} Q_{m,a,b}(t)R_{M,m}(t^a,t^b)$ and we obtain
$$Q_m(t^a,t^b) = \sum_{M\in\mathcal{L}} S_{M,m}(t^a,t^b) \sum_{m'\in\mathcal{M}} Q_{m',a,b}(t)R_{M,m'}(t^a,t^b)=   
Q_{m,a,b}(t)\in \mathbb{Z}[t^{\pm 1/2}],$$ 
whose number of terms is bounded independently of $a,b$ (as above, from explicit description of finite-type quantum cluster algebras).
As this is true for all distinct prime numbers $a,b$ large enough, this implies that $Q_m(t_{-2},t_0)$ is a Laurent polynomial.
Indeed, consider the development as a Laurent formal power series 
$$Q_m(t_{-2},t_0) = t_{-2}^At_0^B\sum_{\alpha,\beta\geq 0} \lambda_{\alpha,\beta}t_{-2}^{\frac{\alpha}{2}} t_0^{\frac{\beta}{2}} \in \mathbb{C}((t_{-2}^{\frac{1}{2}},t_0^{\frac{1}{2}})).$$
Fix $(\alpha,\beta)$ so that $t_{-2}^{\frac{\alpha}{2}} t_0^{\frac{\beta}{2}}$ occurs and consider $(\alpha',\beta')$ so that $a\alpha + b\beta = a\alpha' + b \beta'$. Then $a(\alpha - \alpha') = b(\beta' - \beta)$ and, as $a\neq b$ are distinct prime numbers, there is $\mu\in\mathbb{Z}$ so that $\alpha - \alpha' = b \mu$ and $\beta - \beta' = -a\mu$. 
We get $ -\beta/a \leq \mu\leq \alpha/b$. Then $\mu$ depends on $a,b$ but is an integer, so for $a,b$ large enough, we obtain $\mu = 0$, and $(\alpha,\beta) = (\alpha',\beta')$.
Hence $\lambda_{\alpha,\beta}t_{-2}^{\frac{\alpha}{2}} t_0^{\frac{\beta}{2}}$ in the only term contributing to the coefficient of $t^{\frac{\alpha a + \beta b}{2}}$ in $Q_{m,a,b}(t)$. And so if $\lambda_{\alpha,\beta}\neq 0$, then $t^{\frac{\alpha a + \beta b}{2}}$ occurs in $Q_{m,a,b}(t)$. 
As the number of monomials in $Q_{m,a,b}(t)$ is bounded as explained above, this implies that only a finite number of terms occur in the 
development of $Q_m(t_{-2},t_0)$. 

Now $X$ and  $\sum_{m\in\mathcal{M}} Q_m(t_{-2},t_0) m$ 
coincide after specialization for infinitely many primes $a,b$, hence they coincide as 
in the Proof of Theorem \ref{pos}. This proves the other inclusion.
\end{proof}

\begin{rem}\label{20190805:rem1}
When $ \xi_k = 0 $, the exchange relations for the initial seed of $ \mathcal{A}_{tor}(\mathcal{Q}_1)$ are
\begin{equation}\label{20190828:eq1}
\begin{array}{l}
\displaystyle{
\vphantom{\Big(}
z_{(k,1)}\ast z_{(k,1)}'
=
\prod_{a = 0,-2} t_a^{-\frac12\Big[ \Lambda_a((k,-1),(k,1)) - \sum_{i \sim k}\Lambda_a((k,1),(i,2)) \Big]}
\Big(\prod_{i \sim k } z_{(i,2)} \Big) z_{(k,-1)}
+
\prod_{i \sim k } z_{(i,0)}
\,.}
\end{array}
\end{equation} 

This identifies the mutated variable $z_{(k,1)}'$ with the class of the minimal affinization module $ \widetilde{\big[ L(Y_{i,0}\prod_{j\sim i}Y_{j,3})\big]}_{q,\infty} $ in $\mathscr{K}_{\infty}(\mathscr{C}_1)$.
In fact, following Section \ref{truncinf}, by the formulas for the truncated $q$-characters in \cite[Prop 6.7]{hl} we have
$$ \widetilde{[ L(Y_{i,0}\prod_{j\sim i}Y_{j,3})]}_{q,\infty} 
:=  \widetilde{\chi_q(L(Y_{i,0}\prod_{j\sim i}Y_{j,3}))}  
= Y_{i,0}\prod_{j\sim i}Y_{j,3}\big(1 + A_{i,1}^{-1}\big) \,. $$
Thus, 
$ \widetilde{\big[ L(Y_{i,2})\big]}_{q,\infty}  \ast \widetilde{\big[ L(Y_{i,0}\prod_{j\sim i}Y_{j,3})\big]}_{q,\infty} $
coincides with the RHS of \eqref{20190828:eq1}, under the identification \eqref{20190805:eq3}.
\end{rem}


\section{$A-Y$ Poisson brackets}\label{seven}

In this section we discuss similarities between the various quantum tori used in this paper as well as 
relations to the Poisson structures introduced by Frenkel-Reshetikhin in \cite{Fre:def}.

Let us study the quasi-commutation relations between the monomials $ A_{i,r} $ and $ Y_{j,s} $ inside the quantum torus $ \widetilde{\mathscr{Y}}_{\infty}$ for $i,j \in I $, $r, s \in \mathbb{Z}$, $r - s + \xi_i - \xi_j \in 1 + 2\mathbb{Z}$.
From Equation \eqref{20190320:eq2} we obtain \footnote{Again, under a sign change. Here, with the purpose that the result is compatible with the quasi-commutation product between the variables $ Y_{i,p}^{\pm 1}$.}
$$ A_{i,r}\ast Y_{j,s} = \prod_{a \in \mathbb{Z}} t_a^{N_a(i,r;j,s)} Y_{j,s} \ast A_{i,r}$$ 
where
$$ 
N_a(i,r;j,s) = \delta_{i,j}\big(\delta_{a,r-s-1} - \delta_{a,r-s+1} - \delta_{a,s-r-1} + \delta_{a,s-r+1}\big)\,,
$$
and therefore
\begin{equation*}
A_{i,r}\ast Y_{j,s} = t_{r-s-1}^{\delta_{ij}}t_{r-s+1}^{-\delta_{ij}}t_{s-r-1}^{-\delta_{ij}}t_{s-r+1}^{\delta_{ij}} Y_{j,s} \ast A_{i,r}
\,.
\end{equation*}
When we consider the $\ast$-product between the (images of the) same elements in the quotient quantum torus $ \mathscr{Y}_{\infty}$, we observe that because of the relations $\mathscr{R}_k$ (Definition \ref{quotientorus}),
the elements $ A_{i,r} $, $ Y_{j,s}$ commute whenever $| s - r | >  1$, while 
$$
A_{i,r}\ast Y_{j,s} = (t_{-2}t_0^{-2}t_2)^{\delta_{i,j}(\delta_{s-r,1} - \delta_{r-s,1})}  Y_{j,s} \ast A_{i,r}
\,.
$$
The quasi-commutation relations between the monomials $A_{i,r}$ and $ Y_{j,s}$ only depends on the one parameter $ t:= t_{-2}t_{0}^{-2}t_{2} $.
It coincides with the relation between the same elements (possibly up to a power $ t \mapsto t^{-1} $) obtained in 
\cite{Nak:quiver}, \cite{VV:qGro}, \cite{H1}. These quasi-commutation relations are the same in these papers although the quasi-commutation
relations between $Y$-variables are different (see also \cite[Remark 3.1]{HL:qGro}). 
These relations correspond to the following Poisson brackets : 
$$\{A_{i,r} , Y_{j,s}\} = \delta_{i,j}(\delta_{s-r,1} - \delta_{r-s,1}) A_{i,r}Y_{j,s}$$
which is to be fundamental in this picture. In fact, they can also be derived from the earlier work of 
Frenkel and Reshetikhin \cite{Fre:def}. They introduced certain elements $ A_{i}(z) $ and $ Y_{j}(z)$ as generating series for a Heisenberg algebra depending on two parameters $ \mathscr{H}_{q,t}$, which does not coincide with our $\mathscr{H}$ but from which the construction in \cite{H1} was inspired.

When considering the limit for $ t\mapsto 1 $, the Heisenberg algebra $ \mathscr{H}_{q,t} $ admits a Poisson structure. Here, up to a power of $ q$, we obtain the following identity relating the series  $ A_{i}(z) $ and $ Y_{j}(z)$ (cfr. \eqref{20190624:eq1}):
$$ A_{i}(z) = Y_i(zq^{-1})Y_i(zq)\prod_{j\sim i} Y_j(z)^{-1}\,.$$

By the formulas in \cite[Appendix A]{Fre:def}, we then obtain that the Poisson bracket between $ A_{i,r} := A_i(zq^r) $ and $ Y_{j,s} := Y_j(zq^s) $ in the limit $ \mathscr{H}_{q,1} $, coincides (up to a power $ t \mapsto t^{\pm 2} $) with all the Poisson brackets considered above.

We can therefore deduce that the Poisson bracket between the $A_{i,r} $ and the $ Y_{j,s}$ is somehow extremely rigid, despite the Poisson bracket between the variables $ Y_{i,p}$, $Y_{j,s}$ depending strongly on the deformation.
\begin{rem}
Note that it is already known that the quotient group $\mathcal{P}_q / \mathcal{Q}_q $, where $\mathcal{P}_q $ is the free abelian group generated by the $ Y_{i,a}$ ($i \in I, a \in \mathbb{C}^\ast$) and $\mathcal{Q}_q$ is the subgroup of $\mathcal{P}_q$ generated by the $ A_{i,a}$, plays an important role
in the study of these categories (see \cite{CM}).
\end{rem}

One of the original motivations to study quantum Grothendieck rings
is the existence of analogs of Kazhdan-Lusztig polynomials introduced by Nakajima \cite{Nak:quiver} as coefficients of the transition matrix
between the standard and canonical classes (see \cite{HL:qGro} for a review). These polynomials depend up to a factor only on the quasi-commutation 
relations between the $A$ and $Y$-variables as explained in \cite[Section 5.9]{HL:qGro}. That is why in our toroidal context, although the 
structure of the toroidal Grothendieck rings depends on various parameters, we obtain
polynomial depending only on one parameter (or a fractional power of)  $ t = t_{-2}t_{0}^{-2}t_{2}$.

\begin{ex}
Let $ \overline{\mathfrak{g}} = sl_3 $ and $ \mathscr{C}' = \mathscr{C}_{\mathcal{Q}}$.
Let $ X_1$, $ X_2 $, $ X_3 $ and $X_1'$ as in Example \ref{ex:quantumtoroidal1}.
By \eqref{20190503:eq1} (or, equivalently, \eqref{20181130:eq1b}) we have 
$$
{X_1'}^2  \ast X_1^2
=
(t_1t_2)^{2} \big(X_3^2 + (s^{-1} + s^{-3})X_2 X_3 + s^{-4} X_2^2\big),
$$
$$
{X_1'}^3  \ast X_1^3
=
(t_1t_2)^{\frac92} \big(X_3^3 + (s^{-1} + s^{-3} + s^{-5})X_2X_3^2 + (s^{-4} + s^{-6} + s^{-8})X_2^2 X_3 + s^{-9}X_2^3 \big),
$$
and
$$
\begin{array}{l}
\displaystyle{
{X_1'}^4  \ast X_1^4
=
(t_1t_2)^{8} \Big(X_3^4 +
 ( s^{-1} + s^{-3} + s^{-5} + s^{-7}) X_2X_3^3 
+ (s^{-4} + s^{-6} + 2s^{-8}
}\\
\displaystyle{
+ s^{-10} + s^{-12}) X_2^2X_3^2 
+ (s^{-9} + s^{-11} + s^{-13}+ s^{-15})X_2^3 X_3 
+  s^{-16}X_2^4
\Big),
}
\end{array}
$$
where $ s=t_1t_2^{\frac12} =  t_{-2}^{-\frac12}t_0t_2^{\frac12}$ by \eqref{20190503:eq5}.
\end{ex}


\section{Further possible developments}\label{eight}

We discuss various questions arising from the main results of this paper: the analogs of our result for non ADE-types, 
the existence of toroidal $T$-systems, the toroidal cluster structure for larger categories and the interpretation of 
toroidal cluster monomials.

\subsection{Beyond ADE types}
It would be interesting to study what happens for $\overline{\mathfrak{g}}$ of non simply-laced type.
In this case, some technical adjustments are required, and the construction is more complicated for the lack of symmetry of the (quantum) Cartan matrix. For instance, the relations $(\mathscr{R}_k)_{k \geq 1}$ would be modified in a form depending on $i\in I$,
the definition and properties of truncated $(q,\infty)$-characters (as in Section \ref{truncinf}) would be modified, and, more importantly, it is now known that cluster algebra structures arising from the representation theory of quantum affine algebras are controlled by simply-laced quivers, even
when the type of the underlying Lie algebra is non simply-laced (see \cite{hljems}).

However, we have some results for $ \overline{\mathfrak{g}}$ of type $ B_2$.
Here, the Cartan matrix associated to $ \overline{\mathfrak{g}}$ is $ \begin{pmatrix} 2 & -2 \\ -1 & 2 \end{pmatrix} $, which is symmetrizable with symmetrizing matrix $ D = (d_i\delta_{ij})_{1\leq i,j \leq 2} = \text{diag}(1,2)$.
As a consequence, the corresponding quantum Cartan matrix is 
$$
C(z) = \begin{pmatrix} z + z^{-1} & -z - z^{-1} \\ -1 & z^2 + z^{-2} \end{pmatrix}\,,
$$
which is invertible with inverse (cfr. \cite[Example 4.1]{HO}, or \cite[Appendix A]{gtl})
$$
\widetilde{C}(z) = \frac{1}{z^3 + z^{-3}} \begin{pmatrix} z^2 + z^{-2} & z + z^{-1} \\ 1 & z + z^{-1} \end{pmatrix}.
$$
We can expand each entry $ \widetilde{C}_{i,j}(z) $ as a formal power series $ \widetilde{C}_{i,j}(z) = \sum_{m \geq 1} \widetilde{C}_{i,j}(z) z^m $ , and 
\begin{equation}
\begin{array}{l}
\displaystyle{
\vphantom{\Big(}
\widetilde{C}_{1,1}(z)
=
z + z^5 - z^7 - z^{11} + z^{13} + z^{17} \pm \ldots
\,}\\
\displaystyle{
\vphantom{\Big(}
\widetilde{C}_{2,1}(z)
=
z^3 - z^9 + z^{15}  \pm \ldots
\,}\\
\displaystyle{
\vphantom{\Big(}
\widetilde{C}_{2,2}(z) = \widetilde{C}_{1,2}(z)
=
z^2 + z^4 - z^8 - z^{10} + z^{14} + z^{16} \pm \ldots
\,.}
\end{array}
\end{equation}
The counterpart to the properties in \eqref{20190801:eq4} are proved for type $ B_n $ in \cite{HO}.
In the case $ n = 2$, they provide the periodicity condition
$$\widetilde{C}_{i,j}(m+6) = - \widetilde{C}_{i,j}(m)\,$$
for $i,j\in I$ and $m \ge 1$, together with the condition $\widetilde{C}_{j,i}(d_j) = \delta_{ij}$
and the induction relation
$$
\widetilde{C}_{2,i}(m-2) +  \widetilde{C}_{2,i}(m+2) -  \widetilde{C}_{1,i}(m) = 0\,,\quad m \geq 1
\,.
$$
The (extended) quantum torus $\widetilde{\mathscr{Y}}_{\infty}(B_2)$ is defined, as in Section \ref{sec:quantumtori}, as the
$\mathbb{Z}[t_a^{\pm 1} \mid a \in \mathbb{Z}]$-algebra with generators $Y_{i, r}^{\pm 1} $, $(i,r) \in \{1,2\}\times \mathbb{Z} =: \hat{I}$ (we denote $Y_{i, r}^{\pm 1}:= Y_{i, q^r}^{\pm 1} $), and relations\footnote{Note that in \cite{HO} the authors consider the same quantum torus up to a minus sign. Moreover, the quantum Cartan matrices coincides up to a swap of rows (resp. columns) $1$ and $2$, and its inverse $\widetilde{C}(z)$ is expanded for negative $ z$.}
\begin{equation}\label{20190805:eq5}
Y_{i, p}\ast  Y_{j, s} = \left( \prod_{a\in \mathbb{Z}} {t_a}^{\mathcal{N}_a(i, p; j, s)}\right) Y_{j, s} \ast Y_{i, p},
\end{equation}
where the formula for $ \mathcal{N}_a: \hat{I} \times \hat{I} \longrightarrow \mathbb{Z}$ here generalizes the one for the simply-laced case:
\begin{equation}\label{20190805:eq6}
\begin{array}{c}
\displaystyle{
\vphantom{\Big(}
\mathcal{N}_a(i,p; j,s)
=
\widetilde{C}_{j,i}(p-s- d_j+a) -\widetilde{C}_{j,i}(s-p- d_j+a) 
}\\
\displaystyle{
\vphantom{\Big(}
-\widetilde{C}_{j,i}(p-s+ d_j+a) + \widetilde{C}_{j,i}(s-p+d_j+a) 
.}
\end{array}
\end{equation}

As before, let $\mathscr{C}$ be the category of finite-dimensional representations of $ U_q(\mathfrak{g})$ of type $ 1 $.
Analogues of the subcategories $ \mathscr{C}_{\mathcal{Q}}$ for type $ B $ are the categories $ \mathscr{C}_{\mathcal{Q}^\flat} $ introduced in \cite{OS}, and exploited in \cite{HO} in order to prove an analogue isomorphism with a quantized coordinate algebra.
It is remarkable that the $t$-deformed quantum Grothendieck ring of the category $ \mathscr{C}_{\mathcal{Q}^\flat} $ in type $ B_n $ is isomorphic to a quantized coordinate algebra of type $A_{2n-1}$.
As a consequence of this isomorphism, the quantum Grothendieck ring $ \mathscr{K}_t(\mathscr{C}_{\mathcal{Q}^\flat}) $ possesses a quantum cluster algebra structure. Moreover, there is an isomorphism of $t$-deformed quantum Grothendieck rings $ \mathscr{K}_t(\mathscr{C}_{\mathcal{Q}^\flat,B_n}) \overset{\sim}{\longrightarrow}  \mathscr{K}_t(\mathscr{C}_{\mathcal{Q},A_{2n-1}}) $.

Our purpose is therefore to see if and to what extent these results can be extended to the toroidal setting, in particular in the case $ n = 2$.
The construction of the category $ \mathscr{C}_{\mathcal{Q}^\flat} $ is described in \cite{HO}, and it depends on some twist of a quiver of type $ A_2$ and of a certain choice $ \flat \in \{ <, >\}$. Here, we will assume $ \flat = >$.

Following the construction in \cite[Section 3]{HO}, $\mathscr{C}_{\mathcal{Q}^\flat}$ can be described as the full subcategory of $\mathscr{C}$ whose simple components are indexed by dominant commutative monomials $\mathbb{Y}^\flat= \{Y_{1,r}, Y_{2,r+1} \mid r = 0,2,4\}$. It is a monoidal subcategory \cite[Lemma 3.26]{HO}.

As all Kirillov-Reshetikhin (and hence fundamental) modules are thin for $\overline{\mathfrak{g}}$ of type $ B_2$, we can define the (generalized) toroidal Grothendieck ring $ \widetilde{\mathscr{K}}_{\infty}(\mathscr{C}_{\mathcal{Q}^\flat}) $ as the $\mathbb{Z}[t_a^{\pm 1} \mid a \in \mathbb{Z}]$-subalgebra of the quantum torus $\widetilde{\mathscr{Y}}_{\infty}(B_2)$ generated by the $(q,\infty)$-classes of the fundamental modules indexed by the set $\mathbb{Y}^\flat$, with product $\ast$ (see Remark \ref{usesmall}):
$$ [ V_i(q^{r}) ]_{q,\infty} := \chi_{q,t}( V_i(q^{r})) = \chi_q( V_i(q^{r}) )\,.$$
We may use the truncated $(q,\infty)$-characters as in Section \ref{truncinf}.
Following \cite[Appendix]{H1}, we have: 
$$
\begin{array}{l l l}
\widetilde{[V_1(1)]}_{q,\infty} = Y_{1,0} + Y_{1,2}^{-1}Y_{2,1} + Y_{2,5}^{-1}Y_{1,4}\,,
& \widetilde{[V_1(q^2)]}_{q,\infty} = Y_{1,2} + Y_{1,4}^{-1}Y_{2,3}\,,
& \widetilde{[V_1(q^4)]}_{q,\infty} = Y_{1,4}\,,\\ 
\widetilde{[V_2(q)]}_{q,\infty} = Y_{2,1} + Y_{2,5}^{-1}Y_{1,2}Y_{1,4}\,,
& \widetilde{[V_2(q^3)]}_{q,\infty} = Y_{2,3}\,,
& \widetilde{[V_2(q^5)]}_{q,\infty} = Y_{2,5}\,.
\end{array}
$$

However, the issue of constructing a flat deformation again forces us to impose additional relations among the parameters $ t_a^{\pm\frac12}$ in the quantum torus $\widetilde{\mathscr{Y}}_{\infty}(B_2)$, and to modify our definition of toroidal Grothendieck ring.
An instance of this phenomenon appears when computing the $\ast$-product between the classes $ \widetilde{[V_1(1)]}_{q,\infty} $ and $ \widetilde{[V_1(q^2)]}_{q,\infty} $ (or equivalently between the classes $ \widetilde{[V_1(1)]}_{q,\infty} $ and $ \widetilde{[V_1(q^4)]}_{q,\infty} $).
In fact, according to \eqref{20190805:eq5}, we obtain
$$
\begin{array}{l}
\displaystyle{
\vphantom{\Big(}
\widetilde{[V_1(1)]}_{q,\infty}\ast
\widetilde{[V_1(q^2)]}_{q,\infty} 
=
\prod_{a \in\mathbb{Z}} t_a^{\frac{\mathcal{N}_a(1,0;1,2)}{2}}
\Big(Y_{1,0}Y_{1,2}+Y_{2,3}Y_{2,5}^{-1} 
+Y_{1,2}^{-1}Y_{1,4}^{-1}Y_{2,1}Y_{2,3}
}\\
\displaystyle{
\vphantom{\Big(}
+ \prod_{a \in\mathbb{Z}} t_a^{\frac{2\mathcal{N}_a(1,0;1,2)-\mathcal{N}_a(2,0;2,4)+\mathcal{N}_a(1,0;2,1)}{2}}
Y_{1,0}Y_{1,4}^{-1}Y_{2,3}\Big)
+ \prod_{a \in\mathbb{Z}} t_a^{\frac{\mathcal{N}_a(1,0;2,1)}{2} }\widetilde{[V_2(q)]}_{q,\infty}
\,,}
\end{array}
$$
where
\begin{equation*}
\begin{array}{l}
\displaystyle{
\vphantom{\Big(}
\prod_{a \in\mathbb{Z}} t_a^{\mathcal{N}_a(1,0;1,2)}
=
t_{-2}t_0^{-1}t_4^{-1}t_8t_{10}\prod_{k \geq 2} {t_{6k+2}}^{-(-1)^{k}}{t_{6k+4}}^{-(-1)^{k}}
\,,}\\
\displaystyle{
\vphantom{\Big(}
\prod_{a \in\mathbb{Z}} t_a^{\mathcal{N}_a(1,0;2,1)}
=
t_0t_2^{-1}t_4^{-1}t_8t_{10}\prod_{k \geq 2} {t_{6k+2}}^{-(-1)^{k}}{t_{6k+4}}^{-(-1)^{k}}
\,,}\\
\displaystyle{
\vphantom{\Big(}
\prod_{a \in\mathbb{Z}} t_a^{\mathcal{N}_a(2,0;2,4)}
=
t_{-4}t_{-2}t_0^{-1}t_2^{-2}t_4^{-2}t_8^3t_{10}^3\prod_{k \geq 2} {t_{6k+2}}^{-3(-1)^{k}}{t_{6k+4}}^{-3(-1)^{k}}
\,.}
\end{array}
\end{equation*}
Therefore, when combining with the product of the same classes in the opposite order, there appears a term not indexed by a dominant monomial:
$$
\begin{array}{l}
\displaystyle{
\vphantom{\Big(}
\widetilde{[V_1(1)]}_{q,\infty} 
\ast
\widetilde{[V_1(q^2)]}_{q,\infty} 
-
\prod_{a \in\mathbb{Z}} t_a^{\mathcal{N}_a(1,0;1,2)}
\widetilde{[V_1(q^2)]}_{q,\infty} 
\ast
\widetilde{[V_1(1)]}_{q,\infty} 
\,}\\
\displaystyle{
\vphantom{\Big(}
=
\Big(
1
-
t_{-2}t_{0}^{-2}t_2
\Big)
\prod_{a \in\mathbb{Z}} t_a^{\frac{\mathcal{N}_a(1,0;2,1)}{2} }
\widetilde{[V_2(q)]}_{q,\infty} 
\,}\\
\displaystyle{
\vphantom{\Big(}
+\Big(
1
-
t_{-4}t_{-2}^{-1}t_2^{-1}t_4
\Big)
\prod_{a \in\mathbb{Z}} t_a^{\frac{3\mathcal{N}_a(1,0;1,2)-\mathcal{N}_a(2,0;2,4)+\mathcal{N}_a(1,0;2,1)}{2}}
Y_{1,0}Y_{1,4}^{-1}Y_{2,3}
\,.}
\end{array}
$$
Thus, we let $\mathscr{Y}_{\infty}(B_2)$ be the quotient of the quantum torus $\widetilde{\mathscr{Y}}_{\infty}(B_2)$ by the relation 
$$t_{-4}^{\frac12}t_{-2}^{-\frac12}t_2^{-\frac12}t_4^{\frac12} = 1,$$ 
and we let $ \mathscr{K}_{\infty}(\mathscr{C}_{\mathcal{Q}^\flat}) $ be the subring of this quotient quantum torus generated by the images of the same fundamental classes
$  \widetilde{[V_1(q^{r})]}_{q,\infty}$, $\widetilde{[ V_2(q^{r+1}) ]}_{q,\infty}$ for $ r = 0,2,4$.
This also provides a consistent definition of the $(q,\infty)$-character of the simple modules
$$
\begin{array}{l}
\widetilde{[L(Y_{1,0}Y_{1,2})]}_{q,\infty} = Y_{1,0}Y_{1,2} +
Y_{1,2}^{-1}Y_{1,4}^{-1}Y_{2,1}Y_{2,3}+Y_{2,3}Y_{2,5}^{-1} + Y_{1,0}Y_{1,4}^{-1}Y_{2,3}\,,
\\
\widetilde{[L(Y_{1,0}Y_{1,4})]}_{q,\infty} = Y_{1,0}Y_{1,4} + Y_{1,2}^{-1}Y_{1,4}Y_{2,1}+Y_{1,4}^2Y_{2,5}^{-1} \,,
\\
\widetilde{[L(Y_{1,0}Y_{2,1})]}_{q,\infty} = Y_{1,0}Y_{2,1}+Y_{1,2}^{-1}Y_{2,1}^2+Y_{1,2}Y_{1,4}^2Y_{2,5}^{-2}+Y_{1,0}Y_{1,2}Y_{1,4}Y_{2,5}^{-1}
\\
\quad
+\Big(
\prod_{a \in\mathbb{Z}} t_a^{\frac{\mathcal{N}_a(1,0;1,2)-\mathcal{N}_a(1,0;2,1)}{2}}
+
\prod_{a \in\mathbb{Z}} t_a^{\frac{-\mathcal{N}_a(1,0;1,2)+\mathcal{N}_a(1,0;2,1)}{2}}
\Big)
Y_{1,4}Y_{2,1}Y_{2,5}^{-1}\,.
\end{array}
$$
The truncated $(q,\infty)$-characters of the other simple modules which arise by computing the $\ast$-product between the fundamental modules can be defined to coincide with the corresponding truncated $ q$-characters.

We prove $ \mathscr{K}_{\infty}(\mathscr{C}_{\mathcal{Q}^\flat}) $ possesses a toroidal cluster algebra structure with initial seed
$$ \Sigma =  (X_1,X_2,X_3,X_4,X_5,X_6, \widetilde{B})\,,$$
$$
\begin{array}{l l l}
X_1 = \widetilde{[V_2(q^5)]}_{q,\infty}\,,
& X_2 = \widetilde{[L(Y_{1,0}Y_{2,5})]}_{q,\infty} = Y_{1,0}Y_{2,5} + Y_{1,2}^{-1}Y_{2,1}Y_{2,5} \,,
& X_3 = \widetilde{[V_1(q^4)]}_{q,\infty}\,,
\\
X_4 =  \widetilde{[L(Y_{2,1}Y_{2,5})]}_{q,\infty} = Y_{2,1}Y_{2,5}\,,
& X_5 =  \widetilde{[L(Y_{1,0}Y_{1,2}Y_{1,4})]}_{q,\infty} = Y_{1,0}Y_{1,2}Y_{1,4}\,,
& X_6 = \widetilde{[V_2(q^3)]}_{q,\infty}
\,,\end{array}
$$
the last three variables being coefficients,
and
$$ \widetilde{B}^T = \begin{pmatrix} 0 & -1 & 1 & 0 & 0 & 0 \\ 1 & 0 & -1 & -1 & 1 & 0 \\ -1 & 1 & 0 & 0 & -1 & 1 \end{pmatrix}\,.$$
We exploit the quantum cluster algebra structure of $ \mathscr{K}_{t}(\mathscr{C}_{\mathcal{Q},A_{3}}) $ and the correspondence between $(q,t)$-characters of simple modules described in \cite[Theorem 12.9]{HO}.

If we choose the parametrization given by 
$$ 
t_{(1)} : =  \prod_{a \in \mathbb{Z}}t_a^{\mathcal{N}_a(1,0;2,1) } \,,\quad
t_{(2)} :=  \prod_{a \in \mathbb{Z}}t_a^{\mathcal{N}_a(1,0;1,2)+\mathcal{N}_a(1,0;2,1)} 
$$
the relations between the variables of the initial seed can be encoded in the matrices
\[
\Lambda_1 = \begin{pmatrix} 0 & 1 & -1 & 1 & 0 & -1 \\ -1 & 0 & 0 & 1 & 0 & -1 \\ 1 & 0 & 0 & 1 & 0 & -1 \\ -1 & -1 & -1 & 0 & 0 & 0 \\ 0 & 0 & 0 & 0 & 0 & 0 \\ 1 & 1 & 1 & 0 & 0 & 0 \end{pmatrix}\,,\quad
\Lambda_2 = \begin{pmatrix} 0 & -1 & 0 & -2 & -2 & -1 \\ 1 & 0 & 0 & -1 & -1 & 0 \\ 0 & 0 & 0 & -1 & -1 & 0 \\ 2 & 1 & 1 & 0 & -1 & 0 \\ 2 & 1 & 1 & 1 & 0 & 1 \\ 1 & 0 & 0 & 0 & -1 & 0 \end{pmatrix}\,.
\]

As $ \widetilde{B}^T\Lambda_1 = \big( 2\,\text{Id} \mid 0 \big) $ and $\widetilde{B}^T\Lambda_2 = \big( -\text{Id} \mid 0 \big) $, $ \Lambda_1$, $\Lambda_2$ are compatible with $ \widetilde{B}$.

The class of $\widetilde{[V_1(1)]}_{q,\infty}$
is obtained as a first-step mutation in direction $ 1 $ :
$$
\widetilde{[V_2(q^5)]}_{q,\infty} \ast  \widetilde{[V_1(1)]}_{q,\infty}
=
t_{(1)}^\frac12 t_{(2)}^{-\frac12} \widetilde{[L(Y_{1,0}Y_{2,5})]}_{q,\infty} + t_{(1)}^{-\frac12} \widetilde{[V_1(q^4)]}_{q,\infty}\,.
$$
We can recover the classes of the other fundamental modules as follows.
First perform the mutation in direction $2$.
We obtain a new cluster variable $ X_2'$ 
$$ X_2 \ast X_2' = t_{(1)}^{-\frac12} X_1X_5 + t_{(1)}^\frac12 t_{(2)}^{-\frac12} X_3X_4\,.$$
A direct computation shows that this provides the identification $ X_2' = \widetilde{[L(Y_{1,2}Y_{1,4})]}_{q,\infty} = Y_{1,2}Y_{1,4}$.
The mutated exchange matrix $ \widetilde{B}' $ is
$$ \widetilde{B}^T = \begin{pmatrix} 0 & 1 & 0 & -1 & 0 & 0 \\ -1 & 0 & 1 & 1 & -1 & 0 \\ 0 & -1 & 0 & 0 & 0 & 1 \end{pmatrix}\,.$$

Now, we know that for the underlying classical cluster algebra, the mutation in direction $ 1 $ produces a new cluster variable $ X_1'' $ which corresponds to the $ q $-character of the fundamental module $ V_2(1)$.
On the other hand, performing the mutation in direction $ 3 $ (of the same seed obtained after the first mutation in direction $ 2 $), we obtain a new cluster variable  $ X_3'' $ which corresponds to the $ q $-character of $ V_1(2)$.
We can compute the $ \ast$-product between the $(q,\infty)$-characters associated to the same modules:
\begin{equation}\label{20190829:eq1}
\begin{array}{l}
\displaystyle{
\vphantom{\Big(}
\widetilde{[V_2(q^5)]}_{q,\infty}
\ast
\widetilde{[V_2(q)]}_{q,\infty}
=
t_{(1)}^{\frac12} t_{(2)}^{-1} \widetilde{[L(Y_{2,1}Y_{2,5})]}_{q,\infty} + \big(t_{(1)}t_{(2)}\big)^{-\frac12} \widetilde{[L(Y_{1,2}Y_{1,4})]}_{q,\infty}
}\\\displaystyle{
\vphantom{\Big(}
\widetilde{[V_1(q^4)]}_{q,\infty}
\ast
\widetilde{[V_1(q^2)]}_{q,\infty}
=
t_{(1)}^{\frac12} t_{(2)}^{-\frac12} \widetilde{[L(Y_{1,2}Y_{1,4})]}_{q,\infty} + t_{(1)}^{-\frac12} \widetilde{[V_2(q^3)]}_{q,\infty}\,.
}\end{array}
\end{equation}

As we can check that both equations in \eqref{20190829:eq1} give bar-invariant expressions for the classes $\widetilde{[V_2(q)]}_{q,\infty}$ and $\widetilde{[V_1(q^2)]}_{q,\infty}$, we can conclude they are genuine toroidal exchange relations.
Hence, we get an inclusion of the toroidal Grothendieck ring $ \mathscr{K}_{\infty}(\mathscr{C}_{\mathcal{Q}^\flat}) $ in the toroidal cluster algebra of type $A_3$ with two parameters.

Note moreover that the identities in \eqref{20190829:eq1} are instances of the quantum $T$-system in type B, for $ k = 1$, as in \cite[Theorem 9.6]{HO}  (up to a swap of the indices $ 1, 2 \in I$).
\begin{rem} This parametrization coincides with the one given by the parameters $t_0$, $t_{-2}$ above. In particular, 
for the quantum tori which encode the commutation relations in the initial seed, we have $\Lambda_0 = \Lambda_1$ and $\Lambda_{-2}  = \Lambda_2$.
\end{rem}

\subsection{Toroidal $T$-systems}
In the proof of Theorem \ref{thm:toroidaliso} we have written a toroidal version of certain $T$-system relations in types $ADE$, using truncated $ (q,\infty)$-characters in the category $ \mathscr{C}_1$.
We would like to give a general formulation of toroidal $T$-systems for all simply-laced Lie algebras, in the toroidal Grothendieck ring $\mathscr{K}_{\infty}$. The case of $\overline{\mathfrak{g}} = sl_2$ has been analyzed in Section \ref{sec:casesl2}.

More generally, in type $A$, we would like to define the $ (q,\infty)$-characters
$$[W^{(i)}_{k,p}]_{q,\infty} := \chi_{q}(W^{(i)}_{k,p}) \in \mathscr{Y}_{\infty}\,, \quad i \in I\,, k \geq 1\,, p\in 2\mathbb{Z}+\xi_i \,.$$
One difficulty lies in proving that these classes actually lie in the toroidal Grothendieck ring, by establishing analogs of the recursive formula in Proposition \ref{20190228:prop1}. If true, we can use the same argument of \cite[Prop 5.6]{HL:qGro} to prove
\begin{equation}\label{20190115:eq7}
\begin{array}{l}
\displaystyle{
\vphantom{\Big(}
[W^{(i)}_{k,p}]_{q,\infty} \ast [W^{(i)}_{k,p+2}]_{q,\infty}
=
\prod_{s \in \mathbb Z} t_s^{\alpha_s(i,k)}
[W^{(i)}_{k-1,p+2}]_{q,\infty}
\ast
[W^{(i)}_{k+1,p}]_{q,\infty}
\,}\\
\displaystyle{
\vphantom{\Big(}
\phantom{[W^{(i)}_{k,p}]_{q,\infty} \ast [W^{(i)}_{k,p+2}]_{q,\infty}
=}
+
\ast_{j \sim i} \prod_{s \in \mathbb Z} t_s^{\gamma_s(i,k)} [W^{(j)}_{k,p+1}]_{q,\infty}
\,,}
\end{array}
\end{equation}
where $\alpha_s(i,k)$ and $\gamma_s(i,k)$ can be computed explicitly. We expect an analog picture to hold 
for general simply-laced types and we plan to come back to this in another publication.


\subsection{Larger categories and toroidal cluster monomials}
As recalled in Section \ref{cinqun}, for larger monoidal subcategories of finite-dimensional representations
of quantum affine algebras, such as the categories $\mathscr{C}_Q$ or the categories $\mathscr{C}_\ell$, ($\ell\geq 1$), 
it is known that their Grothendieck ring has a natural cluster algebra structure (see \cite{hlrev} for a recent review). 
Hence, in the view of our main result for the toroidal Grothendieck ring of the category $\mathscr{C}_1$ 
(Theorem \ref{thm:toroidaliso}) and from our general flatness result (Theorem \ref{mainflat}), we expect that 
the toroidal Grothendieck ring of the categories $\mathscr{C}_Q$ and of the categories $\mathscr{C}_\ell$ have 
a natural structure of a toroidal cluster algebras (with at least $2$ independent parameters). Moreover, in the case
of the category $\mathscr{C}_Q$, we expect to recover multi-parameter quantum groups, as explained above in 
the case of $sl_3$.

In this paper we focused on the structure of the various toroidal rings, which is much more involved than
for the classical structure (for instance, the non-deformed Grothendieck rings are just polynomial rings 
in the fundamental representations).
But one of the salient feature of cluster theory is the existence of the cluster monomials. In the context of
monoidal categorification of cluster algebras, they correspond to certain classes of simple representations in 
the Grothendieck ring (see Section \ref{cinqun}). Toroidal cluster monomials can be defined in a toroidal cluster algebras exactly as in 
the classical setting : these are monomials of toroidal cluster variables which belong to a same toroidal cluster. 
When the toroidal Grothendieck ring of a monoidal category has a toroidal cluster algebra structure, it is 
natural to wonder how the cluster monomials can be interpreted, or written in terms of a standard basis 
(that is ordered product of classes of fundamental representations). We expect they belong to a 
certain canonical basis, as for the quantum case (\cite{Nak:quiver, HL:qGro}).

\end{document}